\documentclass[11pt,epsfig,amsfonts,reqno]{amsart}

 \usepackage{hyperref} 
 \usepackage{amssymb}
 \usepackage{amsthm}
 \usepackage{mathrsfs}
 \usepackage{upgreek}
 \usepackage{mathtools}
 \usepackage[mathcal]{euscript}
 \usepackage{amsmath}
  \usepackage{lineno}
  \usepackage{xcolor}
\usepackage[text={170mm,230mm},centering]{geometry}

\allowdisplaybreaks

\newtheorem{theorem}{Theorem}[section]
\newtheorem{lemma}[theorem]{Lemma}
\newtheorem{proposition}[theorem]{Proposition}
\newtheorem{assumption}[theorem]{Assumption}

\newtheorem{definition}[theorem]{Definition}

\newtheorem{remark}[theorem]{Remark}

\let\originalleft\left
\let\originalright\right
\renewcommand{\left}{\mathopen{}\mathclose\bgroup\originalleft}
\renewcommand{\right}{\aftergroup\egroup\originalright}

\renewcommand{\d}{\/\mathrm{d}\/}

\def\w{\textbf{W}^{\varepsilon}_{{\theta}^{\varepsilon}}}

\def\L{\mathbb{L}}
\def\A{\mathrm{A}}
\def\I{\mathrm{I}}

\def\C{\mathrm{C}}
\def\f{\boldsymbol{f}}

\def\B{\mathrm{B}}
\def\D{\mathrm{D}}

\def\X{\mathbb{X}}
\def\x{\boldsymbol{x}}

\def\g{\boldsymbol{g}}

\def\v{\boldsymbol{v}}
\def\w{\boldsymbol{w}}
\def\W{\mathrm{W}}

\def\N{\mathbb{N}}

\def\V{\mathbb{V}}
\def\wi{\widetilde}

\def\P{\mathrm{P}}
\def\u{\boldsymbol{u}}
\def\H{\mathbb{H}}

\newcommand{\R}{\mathbb{R}}

\renewcommand{\d}{\/\mathrm{d}\/}



\author{Kush Kinra} \address[Kush Kinra]
{  Department of Mathematics\\
	Indian Institute of Technology Roorkee-IIT Roorkee \\
		Haridwar Highway, Roorkee, Uttarakhand 247667, INDIA.}
\email[K.~Kinra]{kkinra@ma.iitr.ac.in}

\author{Manil T. Mohan} \address[Manil T. Mohan]
{Corresponding author \\ Department of Mathematics\\
	Indian Institute of Technology Roorkee-IIT Roorkee \\
	Haridwar Highway, Roorkee, Uttarakhand 247667, INDIA.}
\email[M. T.~Mohan]{maniltmohan@ma.iitr.ac.in, maniltmohan@gmail.com}

\author{Renhai  Wang}
\address[Renhai  Wang]
{School of Mathematics and Statistics\\
 Southwest University\\
 Chongqing 400715, CHINA.}
\email[R.~Wang]{rwang-math@outlook.com}

\begin{document}

\baselineskip=1.1\baselineskip

\begin{abstract}
\textsf{This article is concerned with the \emph{asymptotically autonomous robustness} (almost surely and in probability) of non-autonomous random attractors for two stochastic versions of 3D convective Brinkman-Forchheimer (CBF) equations defined on the whole space $\mathbb{R}^3$:	
	$$\frac{\partial\boldsymbol{v}}{\partial t}-\mu \Delta\boldsymbol{v}+(\boldsymbol{v}\cdot\nabla)\boldsymbol{v}	+\alpha\boldsymbol{v}+	\beta|\boldsymbol{v}|^{r-1}\boldsymbol{v}+\nabla p=\boldsymbol{f}(t)+``\mbox{stochastic terms}",\quad \nabla\cdot\boldsymbol{v}=0,$$
	with initial and boundary vanishing conditions, where $\mu,\alpha,\beta >0$, $r\geq1$ and $\boldsymbol{f}(\cdot)$ is a given time-dependent external force field. By the asymptotically autonomous robustness of a non-autonomous random attractor $ \mathscr{A}=\{ \mathscr{A}(\tau,\omega): \tau\in\mathbb{R}, \omega\in\Omega\}$  we mean its time-section $\mathscr{A}(\tau,\omega)$ is robust to a time-independent random set as time $\tau$ tends to negative infinity according to the Hausdorff semi-distance of the underlying space. Our goal is to study this topic, almost surely and in probability, for the	non-autonomous 3D CBF equations when the stochastic term is a linear multiplicative or additive noise, and the time-dependent forcing converges towards a time-independent function. Our main results contain two cases: i) $r\in(3,\infty)$ with any $\beta,\mu>0$;	ii) $r=3$ with $2\beta\mu\geq1$. The main procedure to achieve our goal is how to justify that the usual  pullback asymptotic compactness of the solution operators is uniform on some \emph{uniformly} tempered universes over an \emph{infinite} time-interval $(-\infty,\tau]$.	This can be done by a method based on Kuratowski's measure of noncompactness by showing the backward uniform ``tail-smallness'' and ``flattening-property'' of the solutions over $(-\infty,\tau]$ in order to overcome the lack of compact Sobolev embeddings on unbounded domains. Several rigorous calculations dealing the pressure term $p$ and the nonlinear term $\beta|\boldsymbol{v}|^{r-1}\boldsymbol{v}$ in the whole analysis. The existence of regular solutions of stochastic CBF equations plays a crucial role in establishing uniform tail-estimates and backward flattening-property.} 
\end{abstract}
\title[Asymptotic autonomy of random attractors for 3D stochastic CBF equations]{Asymptotically autonomous robustness in Probability of non-autonomous random attractors for stochastic convective Brinkman-Forchheimer equations on $\mathbb{R}^3$}
\maketitle

	
	\section{Introduction} \label{sec1}\setcounter{equation}{0}
	\subsection{The model}
In this  article, we consider a stochastic fluid dynamic model concerning the 3D convective Brinkman-Forchheimer (CBF)
equations driven by stochastic and
non-autonomous forcing simultaneously defined on the whole space $\R^3$:
	\begin{equation}\label{1}
		\left\{
		\begin{aligned}
			\frac{\partial \v}{\partial t}-\mu \Delta\v+(\v\cdot\nabla)\v+\alpha\v+\beta|\v|^{r-1}\v+\nabla p&=\boldsymbol{f}+S(\v)\circ\frac{\d \W(t)}{\d t},     \text{ in }\  \R^3\times(\tau,\infty), \\ \nabla\cdot\v&=0,  \hspace{28mm}   \text{ in } \R^3\times[\tau,\infty), \\
			\v(x)|_{t=\tau}&=\v_0(x),  \hspace{23mm}    x\in \R^3 \ \text{ and }\ \tau\in\R,\\
			|\v(x)|&\to 0,  \hspace{28mm}   \text{ as }\ |x|\to\infty,
		\end{aligned}
		\right.
	\end{equation}
	where $\v(x,t)\in \R^3$, $p(x,t)\in\R$ and $\f(x,t)\in \R^3$ represent the velocity field, pressure field and external forcing, respectively. The constants $\mu, \alpha, \beta>0$ stand for the \emph{Brinkman} (effective viscosity), \emph{Darcy} (permeability of the porous medium) and \emph{Forchheimer} coefficients, respectively. Note that $r\in[1,\infty)$ is called the absorption exponent and $r=3$ is called the \emph{critical} absorption exponent. The one-dimensional two-sided Wiener process $\W(\cdot)$ is defined on a probability space $(\Omega, \mathscr{F}, \mathbb{P})$ (see Subsection \ref{2.5}). The diffusion coefficient $S(\v)$ of the noise is either
equal to $\v$ (multiplicative noise) or independent of $\v$ (additive noise).  The symbol $\circ$ means that the stochastic integral should be understood in the sense of Stratonovich.

The CBF equations are also referred to as the
 \textbf{tamed} Navier-Stokes equations with a modified damping term
$\alpha\v+\beta|\v|^{r-1}\v$. If $\alpha=\beta=0$, then the system \eqref{1} is reduced to the standard 3D Navier-Stokes equations.
It has been proved by Hajduk and Robinson \cite[Proposition 1.1]{HR} that the CBF equations and the Navier-Stokes equations
have the same scaling only when $r=3$ and $\alpha=0$,
but have no scaling invariance for other values of $\alpha$ and $r$. From the physical point of views, system  \eqref{1} is applied to the flows when the velocities are sufficiently high and porosities are not too small, that is, when the Darcy law for a porous medium does not apply, see \cite{PAM}. In this case, system \eqref{1} is also referred as the \emph{non-Darcy} model.
	
	\subsection{Literature results for CBF equations}
	For deterministic 2D/3D CBF equations, the existence and uniqueness of weak/strong solutions on bounded, periodic and unbounded domains has been investigated in \cite{SNA,FHR,HR,PAM,MTM}; the existence, uniqueness, regularity, stability and  Hausdorff/fractal dimension of global/pullback/exponential attractors have studied in \cite{KM6,MTM2,MTM3} and the references therein. For stochastic 2D/3D CBF equations, the existence of a global in time pathwise mild solution
was justified in \cite{MTM6} when the equations are defined on the whole space and driven by fractional Brownian noise; the existence and uniqueness of strong solutions (in probabilistic sense) was established in \cite{KM} when the equations are defined on unbounded Poincar\'e domains and forced by Gaussian noise.

However,  similar to the 3D Navier-Stokes equations, the existence of the unique global weak solution and unique pathwise strong solution of 3D deterministic and stochastic CBF equations are still open problems for $r\in[1,3)$ with $\beta,\mu>0$, and $r=3$ with $2\beta\mu<1$.
	
\subsection{Literature survey for random attractors}	
The theory of various types of attractors, such as global, pullback, exponential and trajectory attractors, of deterministic dynamical systems has been extensively studied in \cite{Ball,BCL,CLR,CLR1,CLR2,Chueshov2,CV2,Robinson2,Robinson1,R.Temam} and many others. For the long term dynamics of stochastic ordinary/partial differential equations which generates a random dynamical system (RDS) \cite{Arnold}, the extension of global attractors to the random attractors was introduced in \cite{BCF,CDF,CF,Schmalfussr} and successfully applied to 2D stochastic Navier-Stokes equations
and other stochastic equations. Since the evolution equations arriving from physics and other fields of science are often driven by non-autonomous and stochastic forcing simultaneously, random attractors of autonomous RDS are generalized to pullback random attractors in \cite{SandN_Wang} under the framework of non-autonomous RDS. In view of these abstract results, there is a large number of literature on the random attractors for autonomous and non-autonomous stochastic equations, see \cite{BCLLLR,
Caraballo201hf6na,Caraballomb2,GGW,KM6,KM7,KRM,PeriodicWang,RKM,WSY,XC}, etc.
As per the existing literature, the existence of random attractors for stochastic systems is based on some  transformation which converts the stochastic system into a pathwise deterministic system. This transformation is available in the literature only when the noise is either linear multiplicative or additive one, see  \cite{BCLLLR,FY,KM6,LG,PeriodicWang}, etc. In order to deal with the nonlinear diffusion term of the noise, the concept of  mean random attractors was introduced in \cite{Wang} and applied to stochastic Navier-Stokes and CBF equations (It\^o sense) with Lipschitz nonlinear diffusion term in \cite{Wang1} and \cite{KM4}, respectively, see \cite{Wang9,WangGUOWANG} for other physically relevant stochastic models. Another different approach in the direction of random attractors, when the diffusion term is nonlinear, is the Wong-Zakai approximation of pathwise random attractors, see \cite{GGW,GLW,KM7,XC} and the references therein.
	
	\subsection{Motivation, assumptions and main results} In general, a non-autonomous random attractor carries the form
	$\mathscr{A}_{\varsigma}=\{ \mathscr{A}_{\varsigma}(\tau,\omega): \tau\in\R,\ \omega\in\Omega\}$,
	where $\varsigma$ stands for some external perturbation parameter. In the literature, the robustness of pullback random attractors of stochastic CBF equations have been established in \cite{KM6,KM7} with respect to the external parameter $\varsigma$. For the robustness with respect to the external/internal parameters of pullback random attractors of 2D stochastic Navier-Stokes, we refer interested readers to the works \cite{HCPEK,GGW,GLW,KRM,RKM} and the references therein.
Currently, the questions of robustness of pullback random attractors of stochastic CBF equations defined on \emph{unbounded} domains with respect to the \emph{internal} parameter $\tau$, however, is still unsolved.

As per our expectations, if the time-dependent forcing term $\f(x,t)$ converges to some time-in- dependent forcing term $\f_{\infty}(x)$ in some sense, the non-autonomous random dynamics of the system \eqref{1} becomes more and more autonomous. Our main motivation is to examine the asymptotically autonomous robustness of pullback random attractors of \eqref{1}
when the parameters in \eqref{1} are discussed in the following two cases:
\begin{table}[ht]
		\begin{tabular}{|c|c|c|}
			\hline
			\textbf{Cases}&$ r$& conditions on $\mu$ \& $\beta$ \\
			\hline
			\textbf{I}& $r>3$& for any $\mu>0$ and $\beta>0$ \\
			\hline
			\textbf{II}& $r=3$&for $\mu>0$ and $\beta>0$ with $2\beta\mu\geq1$ \\
			\hline
		\end{tabular}
		\vskip 0.1 cm
		\caption{Values of $\mu,\beta$ and $r$.}
		\label{Table}
	\end{table}

\begin{assumption}\label{Hypo_f-N}
		 $\f\in\mathrm{L}^{2}_{\emph{loc}}
		(\R;\L^2(\R^3))$  converges to $\f_{\infty}\in\L^2(\R^3):$
		\begin{align}\label{Hyp1}
			\lim_{\tau\to -\infty}\int^{\tau}_{-\infty}
			\|\f(t)-\f_{\infty}\|^2_{\L^2(\R^3)}\d t=0.
		\end{align} 
	Moreover, there exists a numbers $\delta_1\in[0,\alpha)$ such that for all $t\in\R$
	\begin{align}\label{Hyp2}
		&\sup_{s\leq \tau}\int_{-\infty}^{0} e^{\delta_1 t}\|\f(t+s)\|^2_{\dot{\mathbb{H}}^{-1}(\R^3)}\d t<+\infty.
	\end{align}
	\end{assumption}

Under the above assumption on the external forcing term, let us state our main results of this work.

	\begin{theorem}[\texttt{Multiplicative noise case}]\label{MT1-N}
		Let Assumption \ref{Hypo_f-N} be satisfied. Then, for both the cases given in Table \ref{Table}, the non-autonomous RDS $\Phi$ generated by \eqref{1} with $S(\v)=\v$ has a
		unique pullback random attractor
		$\mathscr{A}
		=\{\mathscr{A}(\tau,\omega):\tau\in\mathbb{R},
		\omega\in\Omega\}$ such that
		$\bigcup\limits_{s\in(-\infty,\tau]}\mathscr{A}
		(s,\omega)$ is precompact in $\L^2(\R^3)$ and  $\lim\limits_{t \to +\infty} e^{- \gamma t}\sup\limits_{s\in(-\infty,\tau]
		}\|\mathscr{A}(s-t,\vartheta_{-t} \omega )\|_{\L^2(\R^3)} =0,$ for any
		$\gamma>0$, $\tau\in \mathbb{R}$ and  $\omega\in\Omega$.  In addition, the time-section  $\mathscr{A}(\tau,\omega)$ is asymptotically
		autonomous robust in $\L^2(\R^3)$, and the limiting set of $\mathscr{A}(\tau,\omega)$ as $\tau\rightarrow-\infty$ is just determined
		by the random attractor $\mathscr{A}_{\infty}=
		\{\mathscr{A}_{\infty}(\omega):
		\omega\in\Omega\}$ of stochastic CBF equations \eqref{1} with the autonomous forcing $\f_\infty$, that is,
		\begin{align}\label{MT2-N}
			\lim_{\tau\to -\infty}\mathrm{dist}_{\L^2(\R^3)}
			(\mathscr{A}(\tau,\omega),
			\mathscr{A}_{\infty}(\omega))=0, \ \ \ \ \  \mathbb{P}\text{-a.s.}
		\end{align}
		Moreover, the
		\texttt{asymptotically autonomous robustness in probability} is also justified:
		\begin{align}\label{MT3-N}
			\lim_{\tau\to -\infty}\mathbb{P}\Big(\omega\in\Omega:\mathrm{dist}_{\L^2(\R^3)}
			(\mathscr{A}(\tau,\omega),
			\mathscr{A}_{\infty}(\omega))\geq\delta\Big){=0},\ \ \ \forall\ \delta>0.
		\end{align}
		In addition,  for any $\varepsilon>0$ and sequence $\tau_{n}\to-\infty$, there exists $\Omega_{\varepsilon}\in\mathscr{F}$ with $\mathbb{P}(\Omega_{\varepsilon})>1-\varepsilon$ such that
		\begin{align}\label{MT4-N}
			\lim_{n\to\infty}\sup_{\omega\in\Omega_{\varepsilon}}\mathrm{dist}_{\L^2(\R^3)}
			(\mathscr{A}(\tau_n,\omega),
			\mathscr{A}_{\infty}(\omega))=0.
		\end{align}
	\end{theorem}
	\begin{theorem}[\texttt{Additive noise case}]\label{MT1}
		Under the Assumption \ref{Hypo_f-N}, for both the cases given in Table \ref{Table}, all results in  Theorem \ref{MT1-N} hold for	the non-autonomous RDS generated by \eqref{1} with $S(\v)=\g$ with $\g\in\D(\A)$, where $\D(\A)$ is the domain of the Stokes operator $\A$ defined in \eqref{StokesO}.
	\end{theorem}

	\begin{remark}
(i) An example of Assumption \ref{Hypo_f-N} is
		$\f(x,t)=\f_\infty(x)e^t+\f_\infty(x)$ with $\f_{\infty}\in\L^2(\R^3)\cap\dot{\mathbb{H}}^{-1}(\R^3)$, where $\dot{\mathbb{H}}^{-1}(\R^3)$ is the dual of the homogeneous Sobolev space $\dot{\mathbb{H}}^{1}(\R^3)$. We also have by \cite[Theorem 1.38]{BCD} (see Remark \ref{Hdot} below) that the embedding $\L^{\frac{6}{5}}(\R^3)\hookrightarrow\dot{\mathbb{H}}^{-1}(\R^3)$ is continuous.

(ii) Assumption \ref{Hypo_f-N} implies the following conditions (cf. \cite{CGTW}):
\begin{align}\label{G3}
				&\mbox{Uniform integrability:}\ \ \  \sup_{s\leq \tau}\int_{-\infty}^{s}e^{\kappa(\xi-s)} \|\f(\xi)\|^2_{\L^2(\R^3)}\d \xi<+\infty,  \ \mbox{$\forall\ \kappa>0$, $\tau\in\mathbb{R}$},
\\&\label{f3-N}
\mbox{Uniform tails-smallness:}	\ \ \ \lim_{k\rightarrow\infty}\sup_{s\leq \tau}\int_{-\infty}^{s}e^{\kappa(\xi-s)}
				\int_{|x|\geq k}|\f(x,\xi)|^{2}\d x\d \xi=0,  \ \mbox{$\forall\  \kappa>0$, $\tau\in\mathbb{R}$.}
\end{align}

(iii) We only use Assumption \ref{Hypo_f-N} for $\f$ in the whole paper.

(iv) In Poincar\'e domains (bounded or unbounded), one can relax the condition \eqref{Hyp2} (see \cite{KM2,RKM}). 


(v) In the additive noise case, we do not need to assume, as in \cite[Hypothesis 1.3]{RKM}, that there exists a constant ${\aleph}>0$ such that
		$\g\in\D(\A)$ satisfies
		\begin{align}\label{GA}
			\bigg|
			\sum_{i,j=1}^3\int_{\R^3}v_i(x)
			\frac{\partial g_j(x)}{\partial x_i}v_j(x)\d x\bigg|\leq {\aleph}\|\v\|^2_{\mathbb{L}^2(\R^3)}, \ \ \forall\ \v\in\L^2(\R^3).
		\end{align}
	
\end{remark}
	\vskip 0.1mm
	\noindent	

	\subsection{Novelties, difficulties and approaches}\label{D&A}
	In order to prove  Theorems \ref{MT1-N} and \ref{MT1}, the uniform precompactness of $\bigcup\limits_{s\in(-\infty,\tau]}\mathscr{A} (s,\omega)$ in $\L^2(\R^3)$ is a pivotal point. The well-known abstract theory of pullback random attractors from \cite{SandN_Wang} tells us that the pullback asymptotic compactness of $\Phi$ gives the  compactness of $\mathscr{A} (\tau,\omega)$ for each $\tau\in\R$, but it cannot provide the precompactness of $\bigcup\limits_{s\in(-\infty,\tau]}\mathscr{A} (s,\omega)$ in $\L^2(\R^3)$, since $(-\infty,\tau]$ is an \emph{infinite} interval. However, motivated by the ideas of \cite{SandN_Wang}, this can be done if one is able to show that the usual pullback asymptotic compactness of $\Phi$ is uniform with respect to a uniformly tempered universe (see \eqref{D-NSE}) over $(-\infty,\tau]$.
	
Note that in the bounded domain case, one can obtain the \emph{uniform} pullback asymptotic compactness of $\Phi$ over $(-\infty,\tau]$ via a compact uniform pullback absorbing set by using compact Sobolev embeddings (see \cite[Theorem 3.10]{KRM}). Moreover, the same idea is used for several stochastic Navier-Stokes, $g$-Navier-Stokes,
magneto-hydrodynamics, Brinkman-Forchheimer equations on bounded domains, see \cite{KRM,LX1,WL,ZL}, etc. Due to the lack of compact Sobolev embeddings in unbounded domains as considered in the present work, to demonstrate such \emph{backward uniform} pullback asymptotic compactness  is therefore harder than that in the bounded domain case. We mention that the criteria of \emph{Kuratowski's measure of noncompactness} (\cite{Kuratowski,Rakocevic}) is useful to resolve the difficulty created by the noncompactness of Sobolev embeddings on unbounded domains (cf. \cite[Lemma 2.7]{LGL}). In order to apply such criteria, we use the idea of \emph{uniform tail-estimates} introduced by Wang \cite{UTE-Wang}, and \emph{flattening-properties} introduced by Ma et. al. \cite{MWZ} (deterministic case) and Kloeden and Langa \cite{Kloeden2007prsl} (random case). Using a cut-off technique, we show that the solutions of \eqref{1} are sufficiently small in $\L^2(\mathcal{O}^c_k)$ uniformly over $(-\infty,\tau],$ when $k$ is large enough, where $\mathcal{O}_{k}=\{x\in\R^3:|x|\leq k\}$ and $\mathcal{O}^{c}_{k}=\mathbb{R}^3\setminus \mathcal{O}_{k}$, that is, we obtain the backward uniform tail-estimates for the solutions. Furthermore, using the same cut-off function, we  establish the backward flattening-properties of the solutions.
	
	Note that parabolic and hyperbolic stochastic models as considered in the works \cite{
CGTW,CLR1HGHDGF,CLR1HFJFFGHDGF,chenp,LGL,Tuan1,UTE-Wang,rwang1} etc., do not contain pressure term $p$. But some physically relevant models such as Navier-Stokes equations (cf. \cite{RKM}), Brinkman-Forchheimer equations (cf. \cite{ZL}) and many others, contain the pressure term $p$. While proving the backward uniform tail-estimates as well as backward flattening-properties of the solutions, when we take a  suitable inner product, the pressure term $p$ does not vanish  with the help of divergence free condition (or incompressibility condition) of the solutions of \eqref{1}. However, by taking the divergence in \eqref{1} formally and using the divergence free condition, we end up with the rigorous expression of the pressure term
	\begin{align}\label{Pressure}
		p=(-\Delta)^{-1}\bigg[\sum_{i,j=1}^{3} \frac{\partial^2}{\partial x_i\partial x_j}(v_iv_j)+\nabla\cdot\{|\v|^{r-1}\v\}-\nabla\cdot \f\bigg],
	\end{align}
	in the weak sense, which is one of the most difficult terms to handle in an appropriate way.  In order to obtain the backward uniform tail-estimates as well as backward flattening-property (Lemmas \ref{largeradius}-\ref{Flattening} and \ref{largeradius-A}-\ref{Flattening-N}), we handle the term \eqref{Pressure} with the help of Plancherel's theorem and the continuous embedding of homogeneous Sobolev spaces given in \cite[Theorem 1.38]{BCD} (see Remark \ref{Hdot} below).

	An another major difficulty to demonstrate the backward flattening-property of the solution to the system \eqref{1} is to handle the integral 
	\begin{align}\label{FL-ES}
	\int_{\mathcal{O}_{\sqrt{2}k}}\left\{1-\uprho\left(\frac{|x|^2}{k^2}\right)\right\}|\v(x)|^{r-1}\v(x)\cdot(\I-\mathrm{P}_i)\left[\left\{1-\uprho\left(\frac{|x|^2}{k^2}\right)\right\}\v(x)\right]\d x,
	\end{align}
	where $\P_i$ is the orthogonal  projection operator which project $\L^2(\mathcal{O}_{\sqrt{2}k})$ to a finite-dimensional space and $\uprho$ is a cut-off function (see Subsection \ref{BUTE-BFP} below for more details). In order to handle the integral \eqref{FL-ES}, the existence of regular solutions of stochastic CBF equations has a major contribution. Due to regularization effect, we are able to find  upper bounds of the terms which contain $\H^{1}$-norm (see \eqref{AB12} and \eqref{AB13} below) and $\L^{3(r+1)}$-norm (see \eqref{AB12-A} and \eqref{AB18V-A} below) of $\v$. We estimate the integral \eqref{FL-ES} in terms of $\H^{1}$-norm and $\L^{3(r+1)}$-norm of $\v$, and obtain the backward flattening-property of the solution to the system \eqref{1} (cf. Lemmas \ref{Flattening} and \ref{Flattening-N} below). Due to the fast growing nonlinearities, the usual energy estimates are not enough to establish the uniform tail-estimates as well as backward flattening-property. Therefore, we use the higher order energy estimates of the solution to handle  the fast growing nonlinearities. 
	
	 As a  result of these backward uniform tail-estimates and backward flattening-property of the solutions to \eqref{1}, the backward uniform pullback asymptotic compactness of $\Phi$ in $\L^2(\R^3)$ follows. The wide-spread idea of energy equations introduced in \cite{Ball} can be used to overcome the non-compactness of Sobolev embeddings on unbounded domains, see the works \cite{BCLLLR,BL,GGW,KM2,KM7,Wang2011Tran,PeriodicWang,rwang2}, etc., and many others. A remark is  that we are currently unable to use the idea of energy equations to prove the backward uniform pullback asymptotic compactness of $\Phi$ in $\L^2(\R^3)$ since $(-\infty,\tau]$ is an infinite time-interval.

	Since we have to consider the uniformly tempered universe to prove the backward uniform pullback asymptotic compactness of $\Phi$, we shall establish the measurability of the uniformly compact  attractor. This is not straightforward compared with the usual case, since the radii of the uniform pullback absorbing set is taken as the supremum over an uncountable set $(-\infty,\tau]$ (see Proposition \ref{IRAS}). In order to  overcome this difficulty, we first observe that the measurability of the usual random attractor is known in the literature, see for example, \cite{BCLLLR,BL,GGW,SandN_Wang}, etc., and then prove that such a uniformly compact attractor is just equal to the usual random attractor. This idea has been successfully used by the authors in \cite{CGTW,WL,RKM}, etc., for different stochastic models.

	\vskip 1mm
	\subsection{Advantages of the linear and nonlinear damping terms}CBF equations are also known as damped Navier-Stokes equations (cf. \cite{HZ}). The damping arises from the resistance to the motion of the flow or by  friction effects. Due to the presence of the damping term $\alpha\v+\beta|\v|^{r-1}\v$, we are able to establish better results than which are available for the Navier-Stokes equations. The existence of global  as well as random  attractors for the Navier-Stokes equations on the whole space   or general unbounded domains is an interesting and challenging open problem. In the literature, for Navier-Stokes equations, these types of results are available on Poincar\'e domains (bounded or unbounded) only (cf. \cite{KRM,RKM}).   For 2D Navier-Stokes equations forced by a linear multiplicative noise on the whole space, we refer to \cite{KM8}. For stochastic CBF equations \eqref{1}, we are considering the whole space, where the linear damping term $\alpha\v$ plays a crucial role to establish the required results on the whole space. This is different from the 2D Navier-Stokes equations  on unbounded Poincar\'e domains, see \cite{RKM}. Moreover, the nonlinear damping $\beta|\v|^{r-1}\v$ with $r>1$ has big advantage in the case of stochastic CBF equations driven by additive noise. We are able to relax the condition \eqref{GA} on noise coefficient $\g$ due to the nonlinear damping. However, for SNSE as well as stochastic CBF equations perturbed by additive noise with $r=1$, one needs the assumption \eqref{GA} on noise coefficient $\g$ (see \cite{RKM}).

\subsection{Outline of the article}
In the next section, we provide the necessary function spaces and abstract formulation of \eqref{1}, and discuss the Ornstein-Uhlenbeck process with its properties. In Section \ref{sec3}, we prove Theorem \ref{MT1-N} for the system \eqref{1} driven by multiplicative noise. In the final section, we prove Theorem \ref{MT1} for the problem \eqref{1} driven by additive noise.

	\section{Mathematical formulation}\label{sec2}\setcounter{equation}{0}
	We start this section with some necessary function spaces whose elements satisfy the divergence free conditions, that is, $\nabla\cdot\v=0.$ Next, in order to obtain the abstract formulation of the system \eqref{1}, we define linear, bilinear and nonlinear operators along with their properties. Finally, we discuss the Ornstein-Uhlenbeck process with some of its properties and the backward tempered random sets.
	\subsection{Function spaces and operators}\label{FnO}
	Let $\C_0^{\infty}(\mathbb{R}^3;\mathbb{R}^3)$ denote  the space of all $\mathbb{R}^3$-valued, infinitely differentiable functions with compact support in $\mathbb{R}^3$.
Let $\mathbb{L}^s(\mathbb{R}^3):=
\mathrm{L}^s(\mathbb{R}^3;\mathbb{R}^3)$ and  $\mathbb{H}^k(\mathbb{R}^3):=
\mathrm{H}^k(\mathbb{R}^3;\mathbb{R}^3)$ for
$s\in[2,\infty)$ and $k\in\mathbb{N}$.
Define the spaces
	\begin{align*}
		\H&:=\overline{\{\v\in\C_0^{\infty}(\mathbb{R}^3;\mathbb{R}^3)
:\nabla\cdot\v=0\}}^{\mathbb{L}^2(\mathbb{R}^3)},\\ \V&:=\overline{\{\v\in\C_0^{\infty}(\mathbb{R}^3;
\mathbb{R}^3):\nabla\cdot\v=0\}}^{\mathbb{H}^1(\mathbb{R}^3)},\\
		\wi\L^p&:=\overline{\{\v\in\C_0^{\infty}
(\mathbb{R}^3;\mathbb{R}^3):\nabla\cdot\v=0\}
}^{\mathbb{L}^p(\mathbb{R}^3)},\ \ p>2.
	\end{align*}
The spaces $\H$, $\V$ and $\widetilde{\L}^{p}$ are endowed with the norms $$\|\v\|_{\H}^2:=\int_{\R^3}|\v(x)|^2\d x,\ \|\v\|^2_{\V}=\int_{\R^3}|\v(x)|^2\d x+\int_{\R^3}|\nabla\v(x)|^2\d x \text{ and } \|\v\|_{\wi \L^p}^p:=\int_{\R^3}|\v(x)|^p\d x,$$ for $p\in(2,\infty)$, respectively. The inner product in the Hilbert space $\H$ is represented by $( \cdot, \cdot)$. The duality pairing between the spaces $\V$ and $\V'$, and $\widetilde{\L}^p$ and its dual $\widetilde{\L}^{\frac{p}{p-1}}$ is denoted by $\langle\cdot,\cdot\rangle.$ Also, the space $\H$ can be identified with its own dual $\H'$. We endow the space $\V\cap\widetilde{\L}^{p}$ with the norm $\|\v\|_{\V}+\|\v\|_{\widetilde{\L}^{p}},$ for $\v\in\V\cap\widetilde{\L}^p$ and its dual $\V'+\widetilde{\L}^{p'}$ with the norm (cf. \cite[Subsection 2.1]{sum_and_inter}) $$\inf\left\{\|\u_1\|_{\V'}+\|\u_2\|_{\widetilde{\L}^{p'}}:\u=\u_1+\u_2, \ \u_1\in\V', \ \u_2\in\widetilde{\L}^{p'}\right\}.$$

Let us now discuss the definition of homogeneous Sobolev spaces and a  result on the continuous embedding which we use in the sequel.
	\begin{definition}[{\cite[Definition 1.31]{BCD}}]
	Let $s\in\R$. The homogeneous Sobolev space $\dot{\mathbb{H}}^{s}(\R^3)$ is the space of tempered distributions $\v$ over $\R^3$, the Fourier transform of which belongs to $\mathbb{L}^1_{\mathrm{loc}}(\R^3)$ and satisfies
	\begin{align*}
		\|\u\|^2_{\dot{\mathbb{H}}^s(\R^3)}:= \int_{\R^3}|\xi|^{2s}|\hat{\v}(\xi)|^2\d\xi<+\infty.
	\end{align*}
\end{definition}
\begin{theorem}[{\cite[Theorem 1.38]{BCD}}]
	If $s\in[0,\frac{3}{2})$, then the space $\dot{\mathbb{H}}^{s}(\R^3)$ is continuously embedded in $\L^{\frac{6}{3-2s}}(\R^3)$.
\end{theorem}
\begin{remark}\label{Hdot}
	For  $s=1$, we have the embedding $\dot{\mathbb{H}}^{1}(\R^3)\hookrightarrow\L^{6}(\R^3)$ is continuous. This implies that the embedding of dual spaces $\L^{\frac{6}{5}}(\R^3)\hookrightarrow\dot{\mathbb{H}}^{-1}(\R^3)$ is continuous.
\end{remark}

	\subsubsection{Linear operator}\label{LO}
	Let $\mathscr{P}: \L^2(\R^3) \to\H$ be the Helmholtz-Hodge (or Leray) projection. Note that the projection operator $\mathscr{P}$ can be expressed in terms of the Riesz transform (cf. \cite{MTSS}).  We define the Stokes operator
	\begin{equation}\label{StokesO}
		\A\v:=-\mathscr{P}\Delta\v,\;\v\in\D(\A):=\V\cap\H^{2}(\R^3).
	\end{equation}
	We also have that $\mathscr{P}$ and $\Delta$ commutes in $\R^3$, that is, $\mathscr{P}\Delta=\Delta\mathscr{P}$.

	\subsubsection{Bilinear operator}\label{BO}
	Let us define the \emph{trilinear form} $b(\cdot,\cdot,\cdot):\V\times\V\times\V\to\R$ by $$b(\v_1,\v_2,\v_3)=\int_{\R^3}(\v_1(x)\cdot\nabla)\v_2(x)\cdot\v_3(x)\d x=\sum_{i,j=1}^{3}\int_{\R^3}v_{1,i}(x)\frac{\partial v_{2,j}(x)}{\partial x_i}v_{3,j}(x)\d x.$$ If $\v_1, \v_2$ are such that the linear map $b(\v_1, \v_2, \cdot) $ is continuous on $\V$, the corresponding element of $\V'$ is denoted by $\B(\v_1, \v_2)$. We also denote $\B(\v) = \B(\v, \v)=\mathscr{P}[(\v\cdot\nabla)\v]$. An integration by parts yields
	\begin{equation}\label{b0}
		\left\{
		\begin{aligned}
			b(\v_1,\v_2,\v_2) &= 0,\ \text{ for all }\ \v_1,\v_2 \in\V,\\
			b(\v_1,\v_2,\v_3) &=  -b(\v_1,\v_3,\v_2),\ \text{ for all }\ \v_1,\v_2,\v_3\in \V.
		\end{aligned}
		\right.\end{equation}

	\begin{remark}
		Note that $\langle\B(\v_1,\v_1-\v_2),\v_1-\v_2\rangle=0$ (for all $\v_1, \v_2 \in \V$) gives us
		\begin{equation}\label{441}
			\begin{aligned}
				\langle \B(\v_1)-\B(\v_2),\v_1-\v_2\rangle =\langle\B(\v_1-\v_2,\v_2),\v_1-\v_2\rangle=-\langle\B(\v_1-\v_2,\v_1-\v_2),\v_2\rangle.
			\end{aligned}
		\end{equation}
	\end{remark}
	\subsubsection{Nonlinear operator}
	Let us consider the nonlinear operator $\mathcal{C}(\v):=\mathscr{P}(|\v|^{r-1}\v),$ for $\v\in\V\cap\wi\L^{r+1}$. The map $\mathcal{C}(\cdot):\V\cap\widetilde{\L}^{r+1}\to\V'+\widetilde{\L}^{\frac{r+1}{r}}$ and  $\langle\mathcal{C}(\v),\v\rangle =\|\v\|_{\widetilde{\L}^{r+1}}^{r+1}$.
	\begin{remark}
		For any $\v_1, \v_2 \in \V\cap\widetilde{\L}^{r+1}$, we have (cf. \cite[Subsection 2.4]{MTM1})
		\begin{align}\label{MO_c}
			\langle\mathcal{C}(\v_1)-\mathcal{C}(\v_2),\v_1-\v_2\rangle \geq\frac{1}{2}\||\v_1|^{\frac{r-1}{2}}(\v_1-\v_2)\|_{\H}^2+\frac{1}{2}\||\v_2|^{\frac{r-1}{2}}(\v_1-\v_2)\|_{\H}^2\geq 0,\ \text{ for all }\ r\geq 1.
		\end{align}

	\end{remark}
	\subsection{Abstract formulation and Ornstein-Uhlenbeck process}\label{2.5}
	By taking the projection $\mathscr{P}$ on the 3D stochastic CBF equations \eqref{1}, we obtain the following abstract formulation by linear, bilinear and nonlinear operators:
	\begin{equation}\label{SCBF}
		\left\{
		\begin{aligned}
			\frac{\d\v}{\d t}+\mu \A\v+\B(\v)+\alpha\v +\beta\mathcal{C}(\v)&=\mathscr{P}\f +S(\v)\circ\frac{\d \W}{\d t},\ \ \  t>\tau, \\
			\v(x)|_{t=\tau}&=\v_{\tau}(x),\hspace{24mm}x\in \R^3,
		\end{aligned}
		\right.
	\end{equation}
	where $S(\v)=\v$ (multiplicative noise) or $S(\v)$ is independent of $\v$ (additive noise). Here, the symbol $\circ$ represents that the stochastic integral is understood in the sense of Stratonovich and $\W(t,\omega)$ is the standard scalar Wiener process on the probability space $(\Omega, \mathscr{F}, \mathbb{P}),$ where $\Omega=\{\omega\in C(\R;\R):\omega(0)=0\},$ endowed with the compact-open topology given by the metric
	\begin{align*}
		d_{\Omega}(\omega,\omega'):=\sum_{m=1}^{\infty} \frac{1}{2^m}\frac{\|\omega-\omega'\|_{m}}{1+\|\omega-\omega'\|_{m}},\text{ where } \|\omega-\omega'\|_{m}:=\sup_{-m\leq t\leq m} |\omega(t)-\omega'(t)|,
	\end{align*}
	$\mathscr{F}$ is the Borel sigma-algebra induced by the compact-open topology of $(\Omega,d_{\Omega})$ and $\mathbb{P}$ is the two-sided Wiener measure on $(\Omega,\mathscr{F})$. From \cite{FS}, it is clear that  the measure $\mathbb{P}$ is ergodic and invariant under the translation-operator group $\{\vartheta_t\}_{t\in\R}$ on $\Omega$ defined by
	\begin{align*}
		\vartheta_t \omega(\cdot) = \omega(\cdot+t)-\omega(t), \ \text{ for all }\ t\in\R, \ \omega\in \Omega.
	\end{align*}
	The operator $\vartheta(\cdot)$ is known as \emph{Wiener shift operator}.
	\subsubsection{Ornstein-Uhlenbeck process}
	Consider for some $\sigma>0$
	\begin{align}\label{OU1}
		y(\vartheta_{t}\omega) =  \int_{-\infty}^{t} e^{-\sigma(t-\xi)}\d \W(\xi), \ \ \omega\in \Omega,
	\end{align} which is the stationary solution of the one-dimensional Ornstein-Uhlenbeck equation
	\begin{align}\label{OU2}
		\d y(\vartheta_t\omega) + \sigma y(\vartheta_t\omega)\d t =\d\W(t).
	\end{align}
	It is known from \cite{FAN} that there exists a $\vartheta$-invariant subset $\widetilde{\Omega}\subset\Omega$ of full measure such that $y(\vartheta_t\omega)$ is continuous in $t$ for every $\omega\in \widetilde{\Omega},$ and
	\begin{align}
		\lim_{t\to \pm \infty} \frac{|y(\vartheta_t\omega)|}{|t|}=	\lim_{t\to \pm \infty} \frac{1}{t} \int_{0}^{t} y(\vartheta_{\xi}\omega)\d\xi =\lim_{t\to \infty} e^{-\delta t}|y(\vartheta_{-t}\omega)| =0,\label{Z3}
	\end{align}
	for all $\delta>0$. For further analysis of this work, we do not distinguish between $\widetilde{\Omega}$ and $\Omega$.
	
	Since, $\omega(\cdot)$ has sub-exponential growth  (cf. \cite[Lemma 11]{CGSV}), $\Omega$ can be written as $\Omega=\bigcup\limits_{N\in\N}\Omega_{N}$, where
	\begin{align*}
		\Omega_{N}:=\{\omega\in\Omega:|\omega(t)|\leq Ne^{|t|},\text{ for all }t\in\R\}, \text{ for all } \ N\in\N.
	\end{align*}
	Moreover, for each $N\in\N$, $(\Omega_{N},d_{\Omega_{N}})$ is a polish space (cf. \cite[Lemma 17]{CGSV}).
	\begin{lemma}\label{conv_z}
		For each $N\in\N$, suppose $\omega_k,\omega_0\in\Omega_{N}$ are such that $d_{\Omega}(\omega_k,\omega_0)\to0$ as $k\to\infty$. Then, for each $\tau\in\R$, $T\in\R^+$ and $a\in\R$,
		\begin{align}
			\sup_{t\in[\tau,\tau+T]}&\bigg[|y(\vartheta_{t}\omega_k)-y(\vartheta_{t}\omega_0)|+|e^{a y(\vartheta_{t}\omega_k)}-e^{a y(\vartheta_{t}\omega_0)}|\bigg]\to 0 \ \text{ as } \ k\to\infty,\label{conv_z1}\\
			\sup_{k\in\N}\sup_{t\in[\tau,\tau+T]}&|y(\vartheta_{t}\omega_k)|\leq C(\tau,T,\omega_0).\label{conv_z2}
		\end{align}
	\end{lemma}
	\begin{proof}
		See the proofs of \cite[Corollary 22]{CLL} and \cite[Lemma 2.5]{YR}.
	\end{proof}
	\subsubsection{Backward-uniformly tempered random set}
	A bi-parametric set $\mathcal{D}=\{\mathcal{D}(\tau,\omega)\}$ in a Banach space $\X$ is said to be \emph{backward-uniformly tempered} if
	\begin{align}\label{BackTem}
		\lim_{t\to +\infty}e^{-ct}\sup_{s\leq \tau}\|\mathcal{D}(s-t,\vartheta_{-t}\omega)\|^2_{\X}=0\ \ \forall \ \  (\tau,\omega,c)\in\R\times\Omega\times\R^+,  \ \ \ \text{	where }\ \|\mathcal{D}\|_{\X}=\sup\limits_{\x\in \mathcal{D}}\|\x\|_{\X}.
	\end{align}

	\subsubsection{Class of random sets}
	\begin{itemize}
		\item Let ${\mathfrak{D}}$ be the collection of subsets of $\H$ defined as:
		\begin{align}\label{D-NSE}
			{\mathfrak{D}}=\left\{{\mathcal{D}}=\{{\mathcal{D}}(\tau,\omega):(\tau,\omega)\in\R\times\Omega\}:\lim_{t\to +\infty}e^{-ct}\sup_{s\leq \tau}\|{\mathcal{D}}(s-t,\vartheta_{-t}\omega)\|^2_{\H}=0,\ \forall \ c>0\right\}.
		\end{align}
		\item Let ${\mathfrak{B}}$ be the collection of subsets of $\H$ defined as:
		\begin{align*}
			{\mathfrak{B}}=\left\{{\mathcal{B}}=\{{\mathcal{B}}(\tau,\omega):(\tau,\omega)\in\R\times\Omega\}:\lim_{t\to +\infty}e^{-ct}\|{\mathcal{B}}(\tau-t,\vartheta_{-t}\omega)\|^2_{\H}=0,\ \forall\ c>0\right\}.
		\end{align*}
		\item Let ${\mathfrak{D}}_{\infty}$ be the collection of subsets of $\H$ defined as:
		\begin{align*}
			{\mathfrak{D}}_{\infty}=\left\{\widehat{\mathcal{D}}=\{\widehat{\mathcal{D}}(\omega):\omega\in\Omega\}:\lim_{t\to +\infty}e^{-ct}\|\widehat{\mathcal{D}}(\vartheta_{-t}\omega)\|^2_{\H}=0,\ \forall \ c>0\right\}.
		\end{align*}
	\end{itemize}

	\section{3D stochastic CBF equations: Multiplicative noise}\label{sec3}\setcounter{equation}{0}
	In this section, we consider 3D stochastic CBF equations driven by a linear multiplicative white noise, that is, $S(\v)=\v$ and establish the asymptotic autonomy of pullback random attractors. Let us define $$\u(t,\tau,\omega,\u_{\tau}):=e^{-y(\vartheta_{t}\omega)}\v(t,\tau,\omega,\v_{\tau})\ \text{ with  }\  \u_{\tau}=e^{-y(\vartheta_{\tau}\omega)}\v_{\tau},$$ where $y$ satisfies \eqref{OU2} and $\v(\cdot):=\v(\cdot,\tau,\omega,\v_{\tau})$ is the solution of \eqref{1} with $S(\v)=\v$. Then $\u(\cdot):=\u(\cdot,\tau,\omega,\u_{\tau})$ satisfies:
	\begin{equation}\label{2}
		\left\{
		\begin{aligned}
			\frac{\d\u(t)}{\d t}-\mu \Delta\u(t)&+e^{y(\vartheta_{t}\omega)}(\u(t)\cdot\nabla)\u(t)+\alpha\u(t)+\beta e^{(r-1)y(\vartheta_{t}\omega)} \left|\u(t)\right|^{r-1}\u(t)\\&=-e^{-y(\vartheta_{t}\omega)}\nabla p(t)+\f(t) e^{-y(\vartheta_{t}\omega)} +\sigma y(\vartheta_t\omega)\u(t),  \ \  \ \  \text{ in }\  \R^3\times(\tau,\infty), \\ \nabla\cdot\u&=0, \hspace{79mm}   \text{ in } \  \R^3\times[\tau,\infty), \\
			\u(x)|_{t=\tau}&=\u_{0}(x)=e^{-y(\vartheta_{\tau}\omega)}\v_{0}(x),  \hspace{45mm}  x\in \R^3 \ \text{ and }\ \tau\in\R,\\
			|\u(x)|&\to 0\hspace{80.1mm}  \text{ as }\ |x|\to\infty,
		\end{aligned}
		\right.
	\end{equation}
	as well as (projected form)
	\begin{equation}\label{CCBF}
		\left\{
		\begin{aligned}
			\frac{\d\u(t)}{\d t}+\mu \A\u(t)&+e^{y(\vartheta_{t}\omega)}\B\big(\u(t)\big) +\alpha\u(t)+\beta e^{(r-1)y(\vartheta_{t}\omega)}\mathcal{C}\big(\u(t)\big)\\&=e^{-y(\vartheta_{t}\omega)}\mathscr{P}\f(t) + \sigma y(\vartheta_t\omega)\u(t) , \quad t> \tau,\  \tau\in\R ,\\
			\u(x)|_{t=\tau}&=\u_{0}(x)=e^{-y(\vartheta_{\tau}\omega)}\v_{0}(x), \hspace{16mm} x\in\R^3,
		\end{aligned}
		\right.
	\end{equation}
	in $\V'+\widetilde{\L}^{\frac{r+1}{r}}$, where $r\geq 1$.

	\subsection{Non-autonomous random dynamical system (NRDS)}
	Lusin continuity helps us to define the NRDS. The following lemma (energy inequality) will be frequently used.
	\begin{lemma}
		For both the cases given in Table \ref{Table}, assume that $\f\in\mathrm{L}^2_{\mathrm{loc}}(\R;\L^2(\R^3))$. Then, the solution of \eqref{CCBF} satisfies the following energy inequality: 
		\begin{align}\label{EI1}
			&	\frac{\d}{\d t} \|\u\|^2_{\H}+ \left(\alpha-2\sigma y(\vartheta_{t}\omega)+\frac{\alpha}{2}\right)\|\u\|^2_{\H}+2\mu\|\nabla\u\|^2_{\H}+ 2\beta e^{(r-1)y(\vartheta_{t}\omega)} \|\u\|^{r+1}_{\wi \L^{r+1}}\leq \frac{2}{\alpha}e^{2|y(\vartheta_{t}\omega)|}\|\f\|^2_{\L^2(\R^3)},
		\end{align}
for a.e $t\geq\tau$ and 
	\begin{align}\label{EI4}
		&\frac{\d}{\d t} \|\nabla\u\|^2_{\H} + (\alpha-2\sigma y(\vartheta_{t}\omega))\|\nabla\u\|^2_{\H}+C^*e^{(r-1)y(\vartheta_{t}\omega)}\|\u\|^{r+1}_{\wi\L^{3(r+1)}} \leq C\|\nabla\u\|^2_{\H}+Ce^{2|y(\vartheta_{t}\omega)|}\|\f\|^2_{\L^2(\R^3)},
	\end{align}
for a.e $t>\tau$, where $C, C^*>0$ are some constants.
	\end{lemma}
	\begin{proof}
		From the first equation of the system \eqref{CCBF}, using \eqref{b0} and the Cauchy-Schwarz inequality, one can obtain \eqref{EI1} immediately. 
		
		Now taking inner product to the first equation of the system \eqref{CCBF} by $\A\u$, we find
			\begin{align}\label{EI2}
			&\frac{1}{2}\frac{\d}{\d t} \|\nabla\u\|^2_{\H} +\mu\|\A\u\|^2_{\H} + (\alpha-\sigma y(\vartheta_{t}\omega))\|\nabla\u\|^2_{\H} + \beta e^{(r-1)y(\vartheta_{t}\omega)} \left(\mathcal{C}(\u),\A\u\right)\nonumber\\&=-e^{y(\vartheta_{t}\omega)}b(\u,\u,\A\u)+e^{-y(\vartheta_{t}\omega)}\big(\f,\A\u\big)\nonumber\\&\leq |e^{y(\vartheta_{t}\omega)}b(\u,\u,\A\u)|+\frac{\mu}{4}\|\A\u\|^2_{\H}+\frac{e^{2|y(\vartheta_{t}\omega)|}}{\mu}\|\f\|^2_{\L^2(\R^3)}.
		\end{align}
	From H\"older's and Young's inequalities, we estimate
	\begin{align}
		|e^{y(\vartheta_{t}\omega)}b(u,u,\A\u)| \leq\begin{cases}
			\frac{\mu}{2}\|\A\u\|^2_{\H}+\frac{\beta}{2}e^{(r-1)y(\vartheta_{t}\omega)}\||\u|^{\frac{r-1}{2}}\nabla\u\|^2_{\H}+C\|\nabla\u\|^2_{\H}, &\text{ for }  r>3, \\
			\frac{\mu}{2}\|\A\u\|^2_{\H}+\frac{1}{2\mu}e^{2y(\vartheta_{t}\omega)}\||\u|\nabla\u\|^2_{\H},  &\text{ for } r=3.
		\end{cases}
	\end{align}
	Since $\mathscr{P}$ and $\Delta$ commute on the whole space, we deduce
		\begin{align}\label{CuAu}
		\big(\mathcal{C}(\u),\A\u\big)&=\int_{\mathbb{R}^3}|\nabla\u(x)|^2|\u(x)|^{r-1}\d x+4\left[\frac{r-1}{(r+1)^2}\right]\int_{\mathbb{R}^3}|\nabla|\u(x)|^{\frac{r+1}{2}}|^2\d x.
		\end{align} 
	From the Gagliardo-Nirenberg-Sobolev inequality (Theorem 1, Section 5.6.1, \cite{LCE}), we have
	\begin{align}\label{3(r+1)}
		\|\u\|^{r+1}_{\wi\L^{3(r+1)}}=\||\u|^{\frac{r+1}{2}}\|_{\mathrm{L}^6(\mathbb{R}^3)}\leq C \|\nabla|\u|^{\frac{r+1}{2}}\|^2_{\L^{2}(\R^3)}=\frac{8\beta}{C^*}\left[\frac{r-1}{(r+1)^2}\right] \int_{\mathbb{R}^3}|\nabla|\u(x)|^{\frac{r+1}{2}}|^2\d x,
	\end{align}
where $C^*>0$ is a constant. Combining \eqref{EI1} and \eqref{EI2}-\eqref{3(r+1)}, we arrive at \eqref{EI4}. This completes the proof.
	\end{proof}

	\begin{lemma}\label{Soln}
		For both the cases given in Table \ref{Table}, let $\f\in\mathrm{L}^2_{\mathrm{loc}}(\R;\L^2(\R^3))$. For each $(\tau,\omega,\u_{\tau})\in\R\times\Omega\times\H$, the system \eqref{CCBF} has a unique weak solution $\u(\cdot,\tau,\omega,\u_{\tau})\in\mathrm{C}([\tau,+\infty);\H)\cap\mathrm{L}^2_{\mathrm{loc}}(\tau,+\infty;\V)\cap\mathrm{L}^{r+1}_{\mathrm{loc}}(\tau,+\infty;\widetilde{\L}^{r+1})$ such that $\u$ is continuous with respect to the initial data.
	\end{lemma}
	\begin{proof}
		One can prove the existence and uniqueness of solutions by a standard Faedo-Galerkin approximation method, cf. \cite{HR,KM,PAM}, etc. For continuity with respect to initial data $\u_{\tau}$, see \cite[Lemma 3.5]{KM7}.
	\end{proof}

	Next result shows the Lusin continuity of the mapping of solution to the system \eqref{CCBF} in sample points.
	\begin{proposition}\label{LusinC}
		For both the cases given in Table \ref{Table}, suppose that $\f\in\mathrm{L}^2_{\mathrm{loc}}(\R;\L^2(\R^3))$. For each $N\in\N$, the mapping $\omega\mapsto\u(t,\tau,\omega,\u_{\tau})$ (solution of \eqref{CCBF}) is continuous from $(\Omega_{N},d_{\Omega_N})$ to $\H$, uniformly in $t\in[\tau,\tau+T]$ with $T>0$.
	\end{proposition}

	\begin{proof}
		Assume that $\omega_k,\omega_0\in\Omega_N,\ N\in\mathbb{N}$ such that $d_{\Omega_N}(\omega_k,\omega_0)\to0$ as $k\to\infty$. Let $\mathscr{U}^k(\cdot):=\u^k(\cdot)-\u^0(\cdot),$ where $\u^k(\cdot):=\u(\cdot,\tau,\omega_k,\u_{\tau})$ and $\u^0:=\u(\cdot,\tau,\omega_0,\u_{\tau})$. Then, $\mathscr{U}^k(\cdot)$ satisfies:
		\begin{align}\label{LC1}
			\frac{\d\mathscr{U}^k}{\d t}&=-\mu \A\mathscr{U}^k-\left(\alpha-\sigma y(\vartheta_t\omega_k)\right) \mathscr{U}^k-e^{y(\vartheta_{t}\omega_k)}\left[\B\big(\u^k\big)-\B\big(\u^0\big)\right]-\left[e^{y(\vartheta_{t}\omega_k)}-e^{y(\vartheta_{t}\omega_0)}\right]\B\big(\u^0\big)\nonumber\\&\quad -\beta e^{(r-1)y(\vartheta_{t}\omega_k)}\left[\mathcal{C}\big(\u^k\big)-\mathcal{C}\big(\u^0\big)\right]-\beta\left[ e^{(r-1)y(\vartheta_{t}\omega_k)}- e^{(r-1)y(\vartheta_{t}\omega_0)}\right]\mathcal{C}\big(\u^0\big)\nonumber\\&\quad+\f \left[e^{-y(\vartheta_{t}\omega_k)}-e^{-y(\vartheta_{t}\omega_0)}\right]+\sigma\left[y(\vartheta_t\omega_k)-y(\vartheta_t\omega_0)\right]\u^0,
		\end{align}
		in $\V'+\widetilde{\L}^{\frac{r+1}{r}}$.  Taking the inner product with $\mathscr{U}^k(\cdot)$ in \eqref{LC1}, and using \eqref{b0} and \eqref{441}, we obtain
		\begin{align}\label{LC2}
			\frac{1}{2}\frac{\d }{\d t}\|\mathscr{U}^k\|^2_{\H}&=-\mu\|\nabla\mathscr{U}^k\|^2_{\H}-\left(\alpha-\sigma y(\vartheta_t\omega_k)\right)\|\mathscr{U}^k\|^2_{\H}-\beta e^{(r-1)y(\vartheta_{t}\omega_k)}\left\langle\mathcal{C}\big(\u^k\big)-\mathcal{C}\big(\u^0\big),\u^k-\u^0\right\rangle\nonumber\\&\quad+e^{y(\vartheta_{t}\omega_k)}b(\mathscr{U}^k,\mathscr{U}^k,\u^0)+\left[e^{y(\vartheta_{t}\omega_k)}-e^{y(\vartheta_{t}\omega_0)}\right]b(\u^0,\mathscr{U}^k,\u^0) \nonumber\\&\quad-\beta\left[ e^{(r-1)y(\vartheta_{t}\omega_k)}- e^{(r-1)y(\vartheta_{t}\omega_0)}\right]\left\langle\mathcal{C}\big(\u^0\big),\mathscr{U}^k\right\rangle+\left[e^{-y(\vartheta_{t}\omega_k)}-e^{-y(\vartheta_{t}\omega_0)}\right](\f,\mathscr{U}^k)\nonumber\\&\quad+\sigma\left[y(\vartheta_t\omega_k)-y(\vartheta_t\omega_0)\right](\u^0,\mathscr{U}^k).
		\end{align}
		We know by \eqref{MO_c} that
		\begin{align}\label{LC3}
			-&\left\langle\mathcal{C}\big(\u^k\big)-\mathcal{C}\big(\u^0\big),\u^k-\u^0\right\rangle \leq-\frac{1}{2}\||\u^k|^{\frac{r-1}{2}}(\u^k-\u^0)\|_{\H}^2-\frac{1}{2}\||\u^0|^{\frac{r-1}{2}}(\u^k-\u^0)\|_{\H}^2.
		\end{align}
		Using H\"older's and Young's inequalities, we obtain
		\begin{align}
			\left|\left[e^{-y(\vartheta_{t}\omega_k)}-e^{-y(\vartheta_{t}\omega_0)}\right](\f,\mathscr{U}^k)\right|&\leq C\left|e^{-y(\vartheta_{t}\omega_k)}-e^{-y(\vartheta_{t}\omega_0)}\right|^2\|\f\|^2_{\L^2(\R^3)}+\frac{\alpha}{4}\|\mathscr{U}^k\|^2_{\H},\label{LC4}\\
			\left|\sigma\left[y(\vartheta_t\omega_k)-y(\vartheta_t\omega_0)\right](\u^0,\mathscr{U}^k)\right|&\leq C\left|y(\vartheta_t\omega_k)-y(\vartheta_t\omega_0)\right|^2\|\u^0\|^2_{\H}+\frac{\alpha}{4}\|\mathscr{U}^k\|^2_{\H},\label{LC5}\\
			\left|\left[ e^{(r-1)y(\vartheta_{t}\omega_k)}- e^{(r-1)y(\vartheta_{t}\omega_0)}\right]\left\langle\mathcal{C}\big(\u^0\big),\mathscr{U}^k\right\rangle\right|&\leq C\left|e^{(r-1)y(\vartheta_{t}\omega_k)}- e^{(r-1)y(\vartheta_{t}\omega_0)}\right|\big[\|\u^0\|^{r+1}_{\wi\L^{r+1}}+\|\u^k\|^{r+1}_{\wi\L^{r+1}}\big].\label{LC6}
		\end{align}
		Next, we estimate the remaining terms of \eqref{LC2} separately.
		\vskip 2mm
		\noindent
		\textbf{Case I:} \textit{When $r>3$.} Using H\"older's and Young's inequalities, we infer
		\begin{align}\label{LC9}
			\left|e^{y(\vartheta_{t}\omega_k)}b(\mathscr{U}^k,\mathscr{U}^k,\u^0)\right|&\leq\frac{\mu}{4}\|\nabla\mathscr{U}^k\|_{\H}^2+\frac{\beta}{4}e^{(r-1)y(\vartheta_{t}\omega_k)}\||\mathscr{U}^k||\u^0|^{\frac{r-1}{2}}\|^2_{\H}+C\|\mathscr{U}^k\|^2_{\H},
		\end{align}
		and
		\begin{align}\label{LC10}
			&\left|\left[e^{y(\vartheta_{t}\omega_k)}-e^{y(\vartheta_{t}\omega_0)}\right]b(\u^0,\mathscr{U}^k,\u^0)\right|\nonumber\\&\leq\left|1-e^{y(\vartheta_{t}\omega_0)-y(\vartheta_{t}\omega_k)}\right|e^{y(\vartheta_{t}\omega_k)}\|\nabla\u^0\|_{\H}\||\mathscr{U}^k||\u^0|\|_{\H}\nonumber\\&\leq\frac{\beta}{4}e^{(r-1)y(\vartheta_{t}\omega_k)}\||\mathscr{U}^k||\u^0|^{\frac{r-1}{2}}\|^2_{\H}+C\left|1-e^{y(\vartheta_{t}\omega_0)-y(\vartheta_{t}\omega_k)}\right|^2\|\nabla\u^0\|^2_{\H}+C\|\mathscr{U}^k\|^2_{\H}.
		\end{align}
		\vskip 2mm
		\noindent
		\textbf{Case II:} \textit{When $r=3$ with $2\beta\mu\geq1$.} Applying \eqref{b0}, H\"older's and Young's inequalities, we obtain
		\begin{align}
			\left|e^{y(\vartheta_{t}\omega_k)}b(\mathscr{U}^k,\mathscr{U}^k,\u^0)\right|&=\left|e^{y(\vartheta_{t}\omega_k)}b(\mathscr{U}^k,\mathscr{U}^k,\u^k)\right|\leq\frac{1}{2\beta}\|\nabla\mathscr{U}^k\|_{\H}^2+\frac{\beta}{2}e^{2y(\vartheta_{t}\omega_k)}\||\mathscr{U}^k||\u^k|\|^2_{\H},\label{LC11}
		\end{align}
		and
		\begin{align}
			&\left|e^{y(\vartheta_{t}\omega_k)}-e^{y(\vartheta_{t}\omega_0)}||b(\u^0,\mathscr{U}^k,\u^0)\right|\leq C\left|1-e^{y(\vartheta_{t}\omega_0)-y(\vartheta_{t}\omega_k)}\right|^2\|\nabla\u^0\|^2_{\H}+\frac{\beta}{2}e^{2y(\vartheta_{t}\omega_k)}\||\mathscr{U}^k||\u^0|\|^2_{\H}.\label{LC12}
		\end{align}
		Combining \eqref{LC2}-\eqref{LC12}, we arrive at
		\begin{align}\label{LC13}
			\frac{\d }{\d t}\|\mathscr{U}^k(t)\|^2_{\H}\leq
			P(t)\|\mathscr{U}^k(t)\|^2_{\H}+Q(t),\text{  for a.e. } t\in[\tau,\tau+T] \text{ with } T>0,
		\end{align}
		where $P=y(\vartheta_{t}\omega_k)+C$ and
		\begin{align*}
		Q&=C\left|e^{-y(\vartheta_{t}\omega_k)}-e^{-y(\vartheta_{t}\omega_0)}\right|^2\|\f\|^2_{\L^2(\R^3)}+C\left|y(\vartheta_t\omega_k)-y(\vartheta_t\omega_0)\right|^2\|\u^0\|^2_{\H}\nonumber\\&\quad+C\left|e^{(r-1)y(\vartheta_{t}\omega_k)}-e^{(r-1)y(\vartheta_{t}\omega_0)}\right|\left[\|\u^0\|^{r+1}_{\wi\L^{r+1}}+\|\u^k\|^{r+1}_{\wi\L^{r+1}}\right]
		+C\left|1-e^{y(\vartheta_{t}\omega_0)-y(\vartheta_{t}\omega_k)}\right|^2\|\nabla\u^0\|^2_{\H}.
		\end{align*}
		From \eqref{EI1}, we deduce
		\begin{align*}
			\int_{\tau}^{\tau+T}2\beta e^{(r-1)y(\vartheta_{t}\omega_k)}\|\u^k(t)\|^{r+1}_{\wi\L^{r+1}}\d t&\leq\|\u_{\tau}\|^2_{\H}+C\int_{\tau}^{\tau+T}e^{-2y(\vartheta_{t}\omega_k)}\|\f(t)\|^2_{\L^2(\R^3)}\d t\nonumber\\&\leq\|\u_{\tau}\|^2_{\H}+C\sup_{t\in[\tau,\tau+T]}\left[e^{-2y(\vartheta_{t}\omega_k)}\right]\int_{\tau}^{\tau+T}\|\f(t)\|^2_{\L^2(\R^3)}\d t,
		\end{align*}
		which gives
		\begin{align}\label{LC14}
			\sup_{k\in\N}\int_{\tau}^{\tau+T}e^{(r-1)y(\vartheta_{t}\omega_k)}\|\u^k(t)\|^{r+1}_{\wi\L^{r+1}}\d t\leq C(\tau,T,\omega_0,\u_{\tau},\f),
		\end{align}
		where we have used \eqref{conv_z2} and the fact $\f\in\mathrm{L}^2_{\text{loc}}(\R;\L^2(\R^3))$. Using \eqref{conv_z2} and $\u^0\in\mathrm{L}^2_{\mathrm{loc}}(\tau,+\infty;\V)$, we deduce
		\begin{align}\label{LC15}
			\int_{\tau}^{\tau+T}P(t)\d t\leq C(\tau,T,\omega_0).
		\end{align}
		Now, from \eqref{LC14}, $\f\in\mathrm{L}^2_{\text{loc}}(\R;\L^2(\R^3))$, $\u^0\in\mathrm{C}([\tau,+\infty);\H)\cap\mathrm{L}^2_{\mathrm{loc}}(\tau,+\infty;\V)\cap\mathrm{L}^{r+1}_{\mathrm{loc}}(\tau,+\infty;\widetilde{\L}^{r+1})$ and Lemma \ref{conv_z}, we conclude
		\begin{align}\label{LC16}
			\lim_{k\to+\infty}\int_{\tau}^{\tau+T}Q(t)\d t=0.
		\end{align}
		Making use of the Gronwall inequality  in \eqref{LC13}, we get
		\begin{align}\label{LC17}
			\|\mathscr{U}^k(t)\|^2_{\H}\leq e^{\int_{\tau}^{\tau+T}P(t)\d t}\left[\int_{\tau}^{\tau+T}Q(t)\d t\right], \ \ \text{ for all } t\in[\tau,\tau+T].
		\end{align}
		In view of \eqref{LC15}-\eqref{LC17}, we complete the proof.
	\end{proof}

	Lemma \ref{Soln} ensures that we can define a mapping $\Phi:\R^+\times\R\times\Omega\times\H\to\H$ by
	\begin{align}\label{Phi}
		\Phi(t,\tau,\omega,\v_{\tau}):=\v(t+\tau,\tau,\vartheta_{-\tau}\omega,\v_{\tau})=e^{y(\vartheta_{t}\omega)}\u(t+\tau,\tau,\vartheta_{-\tau}\omega,\u_{\tau}).
	\end{align}
	The Lusin continuity in Proposition \ref{LusinC} provides the $\mathscr{F}$-measurability of $\Phi$. Consequently, in view of Lemma \ref{Soln} and Proposition \ref{LusinC}, we have the following result for NRDS.
	\begin{proposition}\label{NRDS}
		The mapping $\Phi$ defined by \eqref{Phi} is an  NRDS on $\H$, that is, $\Phi$ has  the following properties:
		\begin{itemize}
			\item [(i)]$\Phi$ is $(\mathscr{B}(\R^+)\times\mathscr{B}(\R)\times\mathscr{F}\times\mathscr{B}(\H);\mathscr{B}(\H))$-measurable,
			\item [(ii)] $\Phi$ satisfies the cocycle property: $\Phi(0,\tau,\omega,\cdot)=\I$, and
			\begin{align*}
				\Phi(t+s,\tau,\omega,\v_{\tau})=\Phi(t,\tau+s,\vartheta_s\omega,\Phi(s,\tau,\omega,\v_{\tau})), \ \ \ t,s\geq0.
			\end{align*}
		\end{itemize}
	\end{proposition}

	\subsection{Backward convergence of NRDS}
	Consider the following autonomous 3D stochastic CBF equations driven by linear multiplicative white noise:
	\begin{equation}\label{A-SCBF}
		\left\{
		\begin{aligned}
			\frac{\d\widetilde{\v}(t)}{\d t}+\mu \A\widetilde{\v}(t)+\B(\widetilde{\v}(t))+\alpha\widetilde{\v}(t) +\beta\mathcal{C}(\widetilde{\v}(t))&=\mathscr{P}\f_{\infty} +\widetilde{\v}(t)\circ\frac{\d \W(t)}{\d t}, \ \ t>0, \\
			\widetilde{\v}(x,0)&=\widetilde{\v}_{0}(x),	\hspace{30mm} x\in \R^3.
		\end{aligned}
		\right.
	\end{equation}
	Let $\widetilde{\u}(t,\omega)=e^{-y(\vartheta_{t}\omega)}\widetilde{\v}(t,\omega)$. Then, $\widetilde{\u}(\cdot)$ satisfies
	\begin{equation}\label{A-CCBF}
		\left\{
		\begin{aligned}
			\frac{\d\widetilde{\u}(t)}{\d t}+\mu \A\widetilde{\u}(t)&+e^{y(\vartheta_{t}\omega)}\B\big(\widetilde{\u}(t)\big) +\alpha\widetilde{\u}(t)+\beta e^{(r-1)y(\vartheta_{t}\omega)}\mathcal{C}\big(\widetilde{\u}(t)\big)\\&=\mathscr{P}\f_{\infty} e^{-y(\vartheta_{t}\omega)} + \sigma y(\vartheta_t\omega)\widetilde{\u}(t) , \quad t> 0,\\
			\widetilde{\u}(x,0)&=\widetilde{\u}_{0}(x)=e^{-y(\omega)}\widetilde{\v}_{0}(x), \hspace{18mm} x\in\R^3,
		\end{aligned}
		\right.
	\end{equation}
	in $\V'+\widetilde{\L}^{\frac{r+1}{r}}$.

	\begin{proposition}\label{Back_conver}
		For both the cases given in Table \ref{Table}, suppose that Assumption \ref{Hypo_f-N} is satisfied. Then, $\lim\limits_{\tau\to-\infty}\|\u_{\tau}-\widetilde{\u}_0\|_{\H}=0$ implies that the solution $\u(\cdot)$ of the system \eqref{CCBF} backward converges to the solution $\widetilde{\u}(\cdot)$ of the system \eqref{A-CCBF}, that is,
		\begin{align}
			\lim_{\tau\to -\infty}\|\u(T+\tau,\tau,\vartheta_{-\tau}\omega,\u_{\tau})-\widetilde{\u}(t,\omega,\widetilde{\u}_0)\|_{\H}=0, \ \ \text{ for all } \ T>0 \text{ and } \omega\in\Omega.
		\end{align}
		
	\end{proposition}
	\begin{proof}
		Let $\mathscr{U}^{\tau}(\cdot):=\u(\cdot+\tau,\tau,\vartheta_{-\tau}\omega,\u_{\tau})-\widetilde{\u}(\cdot,\omega,\widetilde{\u}_0)$. From \eqref{CCBF} and \eqref{A-CCBF}, we obtain
		\begin{align}\label{BC1}
			\frac{\d\mathscr{U}^{\tau}}{\d t}&=-\mu \A\mathscr{U}^{\tau}-\alpha\mathscr{U}^{\tau}-e^{y(\vartheta_{t}\omega)}\left[\B\big(\u\big)-\B\big(\widetilde{\u}\big)\right] -\beta e^{(r-1)y(\vartheta_{t}\omega)}\left[\mathcal{C}\big(\u\big)-\mathcal{C}\big(\widetilde{\u}\big)\right]\nonumber\\&\quad+e^{-y(\vartheta_{t}\omega)}\left[\mathscr{P}\f(t+\tau)-\mathscr{P}\f_{\infty}\right] + \sigma y(\vartheta_t\omega)\mathscr{U}^{\tau},
		\end{align}
		in $\V'+\widetilde{\L}^{\frac{r+1}{r}}$.  Taking the inner product with $\mathscr{U}^{\tau}(\cdot)$ in \eqref{BC1}, and using \eqref{b0} and \eqref{441}, we get
		\begin{align}\label{BC2}
			\frac{1}{2}\frac{\d }{\d t}\|\mathscr{U}^{\tau}\|^2_{\H}&=-\mu\|\nabla\mathscr{U}^{\tau}\|^2_{\H}-\left(\alpha-\sigma y(\vartheta_t\omega)\right)\|\mathscr{U}^{\tau}\|^2_{\H}-\beta e^{(r-1)y(\vartheta_{t}\omega)}\left\langle\mathcal{C}\big(\u\big)-\mathcal{C}\big(\widetilde{\u}\big),\u-\widetilde{\u}\right\rangle\nonumber\\&\quad+e^{y(\vartheta_{t}\omega)}b(\mathscr{U}^{\tau},\mathscr{U}^{\tau},\widetilde{\u})+e^{-y(\vartheta_{t}\omega)}(\f(t+\tau)-\f_{\infty},\mathscr{U}^{\tau}).
		\end{align}
		From \eqref{MO_c}, one can write
		\begin{align}\label{BC3}
			-&\left\langle\mathcal{C}\big(\u\big)-\mathcal{C}\big(\widetilde{\u}\big),\u-\widetilde{\u}\right\rangle \leq-\frac{1}{2}\||\u|^{\frac{r-1}{2}}(\u-\widetilde{\u})\|_{\H}^2-\frac{1}{2}\||\widetilde{\u}|^{\frac{r-1}{2}}(\u-\widetilde{\u})\|_{\H}^2.
		\end{align}
		Applying H\"older's and Young's inequalities, we infer
		\begin{align}\label{BC4}
			\left|e^{-y(\vartheta_{t}\omega)}(\f(t+\tau)-\f_{\infty},\mathscr{U}^{\tau})\right|\leq\|\f(t+\tau)-\f_{\infty}\|^2_{\L^2(\R^3)}+Ce^{-2y(\vartheta_{t}\omega)}\|\mathscr{U}^{\tau}\|^2_{\H},
		\end{align}
		and
		\begin{align}\label{BC5}
			&e^{y(\vartheta_{t}\omega)}b(\mathscr{U}^{\tau},\mathscr{U}^{\tau},\widetilde{\u})
			\leq \begin{cases}
				\frac{\mu}{2}\|\nabla\mathscr{U}^{\tau}\|_{\H}^2+\frac{\beta}{4}e^{(r-1)y(\vartheta_{t}\omega)}\||\mathscr{U}^{\tau}||\widetilde{\u}|^{\frac{r-1}{2}}\|^2_{\H}+C\|\mathscr{U}^{\tau}\|^2_{\H} ,&\text{ for }  r>3,\\
				\frac{1}{2\beta}\|\nabla\mathscr{U}^{\tau}\|_{\H}^2+\frac{\beta}{2}e^{2y(\vartheta_{t}\omega)}\||\mathscr{U}^{\tau}|\widetilde{\u}\|^2_{\H},&\text{ for } r=3.
			\end{cases}
		\end{align}
		Combining \eqref{BC2}-\eqref{BC5}, we achieve
		\begin{align}\label{BC6}
			\frac{\d }{\d t}\|\mathscr{U}^{\tau}(t)\|^2_{\H}\leq
			S(t)\|\mathscr{U}^{\tau}(t)\|^2_{\H}+\|\f(t+\tau)-\f_{\infty}\|^2_{\L^2(\R^3)},
		\end{align}
		for a.e. $t\in[\tau,\tau+T]$, where $S(t)=e^{-2y(\vartheta_{t}\omega)}+\left|y(\vartheta_{t}\omega)\right|+1$.
		 Making use of Gronwall's inequality  in \eqref{BC6} over $(0,T)$, we obtain
		\begin{align*}
			\|\mathscr{U}^{\tau}(T)\|^2_{\H}\leq \left[\|\mathscr{U}^{\tau}(0)\|^2_{\H}+\int_{0}^{T}\|\f(t+\tau)-\f_{\infty}\|^2_{\L^2(\R^3)} \d t\right]e^{\int_{0}^{T}S(t)\d t}.
		\end{align*}
		Since $y$ is continuous, it implies that $\int_{0}^{T}S(t)\d t$ is bounded. From Assumption \ref{Hypo_f-N} (particularly, \eqref{Hyp1}), we deduce
		\begin{align}\label{BC7}
			\int_{0}^{T}\|\f(t+\tau)-\f_{\infty}\|^2_{\L^2(\R^3)} \d t\leq \int_{-\infty}^{\tau+T}\|\f(t)-\f_{\infty}\|^2_{\L^2(\R^3)} \d t\to 0 \ \text{ as } \ \tau\to -\infty.
		\end{align}
		Using the fact that $\int_{0}^{T}S(t)\d t$ is bounded, \eqref{BC7} and $\lim\limits_{\tau\to\infty}\|\mathscr{U}^{\tau}(0)\|^2_{\H}=0$, we conclude the proof.
	\end{proof}

	\subsection{Increasing random absorbing sets}
	In this subsection, we prove the existence of a pullback $\mathfrak{D}$-random absorbing set for the  system \eqref{SCBF} with $S(\v)=\v$.
	\begin{lemma}\label{Absorbing}
		For both the cases given in Table \ref{Table}, suppose that Assumption \ref{Hypo_f-N} is satisfied. Then, for each $(\tau,\omega,D)\in\R\times\Omega\times\mathfrak{D},$ there exists a time $\mathcal{T}:=\mathcal{T}(\tau,\omega,D)>0$ such that
		\begin{align}\label{AB1}
			&\sup_{s\leq \tau}\sup_{t\geq \mathcal{T}}\sup_{\u_{0}\in D(s-t,\vartheta_{-t}\omega)}\bigg[\|\u(s,s-t,\vartheta_{-s}\omega,\u_{0})\|^2_{\H}\nonumber\\&\quad+\frac{\alpha}{2}\int_{s-t}^{s}e^{\alpha(\zeta-s)-2\sigma\int^{\zeta}_{s}y(\vartheta_{\upeta-s}\omega)\d\upeta}\|\u(\zeta,s-t,\vartheta_{-s}\omega,\u_{0})\|^2_{\H}\d\zeta\nonumber\\&\quad+2\mu\int_{s-t}^{s}e^{\alpha(\zeta-s)-2\sigma\int^{\zeta}_{s}y(\vartheta_{\upeta-s}\omega)\d\upeta}\|\nabla\u(\zeta,s-t,\vartheta_{-s}\omega,\u_{0})\|^2_{\H}\d\zeta\nonumber\\&\quad+2\beta\int_{s-t}^{s}e^{(r-1)y(\vartheta_{\zeta-s}\omega)+\alpha(\zeta-s)-2\sigma\int^{\zeta}_{s}y(\vartheta_{\upeta-s}\omega)\d\upeta}\|\u(\zeta,s-t,\vartheta_{-s}\omega,\u_{0})\|^{r+1}_{\wi\L^{r+1}}\d\zeta\bigg]\nonumber\\&\leq \frac{4}{\alpha}\sup_{s\leq \tau}\int_{-\infty}^{0}e^{\alpha\zeta+2|y(\vartheta_{\zeta}\omega)|+2\sigma\int_{\zeta}^{0}y(\vartheta_{\upeta}\omega)\d\upeta}\|\f(\zeta+s)\|^2_{\L^2(\R^3)}\d\zeta=:\frac{4}{\alpha}\sup_{s\leq \tau}K(s,\omega).
		\end{align}
		Furthermore, for $2<k_1<\infty$ and $k_2\geq0$, we infer
		\begin{align}\label{AB11}
			&\lim_{t\to+\infty}\sup_{s\leq \tau}\sup_{\u_{0}\in D(s-t,\vartheta_{-t}\omega)}\int_{s-t}^{s}e^{k_2|y(\vartheta_{\zeta-s}\omega)|+\alpha(\zeta-s)-2\sigma\int^{\zeta}_{s}y(\vartheta_{\upeta-s}\omega)\d\upeta}\|\u(\zeta,s-t,\vartheta_{-s}\omega,\u_{0})\|^{k_1}_{\H}\d\zeta\nonumber\\&\leq C\int\limits_{-\infty}^{0}e^{k_2\left|y(\vartheta_{\zeta}\omega)\right|+\frac{\alpha}{k_1}\zeta-(k_1-2)\sigma\int_{\zeta}^{0}y(\vartheta_{\upeta}\omega)\d\upeta}\d\zeta\nonumber\\&\quad\times\sup_{s\leq \tau}\bigg[\int\limits_{-\infty}^{0}e^{\frac{2(k_1-1)\alpha}{k_1^2}\zeta+2|y(\vartheta_{\zeta}\omega)|+2\sigma\int_{\zeta}^{0}y(\vartheta_{\upeta}\omega)\d\upeta}\|\f(\zeta+s)\|^2_{\L^2(\R^3)}\d\zeta\bigg]^{\frac{k_1}{2}},
		\end{align}
	and 
	\begin{align}\label{AB12}
		&\lim_{t\to+\infty}\sup_{s\leq \tau}\sup_{\u_{0}\in D(s-t,\vartheta_{-t}\omega)}\int_{s-t}^{s}e^{k_2|y(\vartheta_{\zeta-s}\omega)|+\alpha(\zeta-s)-2\sigma\int^{\zeta}_{s}y(\vartheta_{\upeta-s}\omega)\d\upeta}\|\nabla\u(\zeta,s-t,\vartheta_{-s}\omega,\u_{0})\|^{k_1}_{\H}\d\zeta\nonumber\\&\leq C\int\limits_{-\infty}^{0}e^{k_2|y(\vartheta_{\zeta}\omega)|+\frac{\alpha}{k_1}\zeta-(k_1-2)\sigma\int_{\zeta}^{0}y(\vartheta_{\upeta}\omega)\d\upeta}\d\zeta\nonumber\\&\quad\times\sup_{s\leq \tau}\bigg[\int\limits_{-\infty}^{0}e^{\frac{2(k_1-1)\alpha}{k_1^2}\zeta+2|y(\vartheta_{\zeta}\omega)|+2\sigma\int_{\zeta}^{0}y(\vartheta_{\upeta}\omega)\d\upeta}\|\f(\zeta+s)\|^2_{\L^2(\R^3)}\d\zeta\bigg]^{\frac{k_1}{2}}.
	\end{align}
Finally, we obtain for $s\in\R$ and $t>0$
\begin{align}\label{AB13}
	&\int_{s-t}^{s} (\zeta-s+t) e^{(r-1)y(\vartheta_{\zeta-s}\omega)+\alpha(\zeta-s)-2\sigma\int^{\zeta}_{s}y(\vartheta_{\upeta-s}\omega)\d\upeta}\|\u(\zeta,s-t,\vartheta_{-s}\omega,\u_{0})\|^{r+1}_{\wi\L^{3(r+1)}}\d\zeta \nonumber\\&\leq C(t+1)\bigg[e^{-\alpha t+2\sigma\int_{-t}^{0}y(\vartheta_{\upeta}\omega)\d\upeta}\|\u_{0}\|^2_{\H}+ \int_{-\infty}^{0} e^{\alpha\zeta-2\sigma\int^{\zeta}_{0}y(\vartheta_{\upeta}\omega)\d\upeta+2|y(\vartheta_{\zeta}\omega)|}\|\f(\zeta+s)\|^2_{\L^2(\R^3)}\d\zeta\bigg].
\end{align}
	\end{lemma}
	\begin{proof}
		Let us write the energy inequality \eqref{EI1} for $\u(\zeta)=\u(\zeta,s-t,\vartheta_{-s}\omega,\u_{0})$, that is,
		\begin{align}\label{AB0}
		&\frac{\d}{\d\zeta} \|\u(\zeta)\|^2_{\H}+ \left(\alpha-2\sigma y(\vartheta_{\zeta-s}\omega)\right)\|\u(\zeta)\|^2_{\H}+\frac{\alpha}{2}\|\u(\zeta)\|^2_{\H}+2\mu\|\nabla\u(\zeta)\|^2_{\H}+ 2\beta e^{(r-1)y(\vartheta_{\zeta-s}\omega)} \|\u(\zeta)\|^{r+1}_{\wi \L^{r+1}} \nonumber\\&\leq \frac{2e^{2|y(\vartheta_{\zeta-s}\omega)|}}{\alpha}\|\f(\zeta)\|^2_{\L^2(\R^3)}.
		\end{align}
		In view of the variation of constants formula with respect to $\zeta\in(s-t,\xi)$, we obtain
		\begin{align}\label{AB3}
			&\|\u(\xi,s-t,\vartheta_{-s}\omega,\u_{0})\|^2_{\H}+\frac{\alpha}{2}\int_{s-t}^{\xi}e^{\alpha(\zeta-\xi)-2\sigma\int^{\zeta}_{\xi}y(\vartheta_{\upeta-s}\omega)\d\upeta}\|\u(\zeta,s-t,\vartheta_{-s}\omega,\u_{0})\|^2_{\H}\d\zeta\nonumber\\&+2\mu\int_{s-t}^{\xi}e^{\alpha(\zeta-\xi)-2\sigma\int^{\zeta}_{\xi}y(\vartheta_{\upeta-s}\omega)\d\upeta}\|\nabla\u(\zeta,s-t,\vartheta_{-s}\omega,\u_{0})\|^2_{\H}\d\zeta\nonumber\\&+2\beta\int_{s-t}^{\xi}e^{(r-1)y(\vartheta_{\zeta-s}\omega)+\alpha(\zeta-\xi)-2\sigma\int^{\zeta}_{\xi}y(\vartheta_{\upeta-s}\omega)\d\upeta}\|\u(\zeta,s-t,\vartheta_{-s}\omega,\u_{0})\|^{r+1}_{\wi\L^{r+1}}\d\zeta\nonumber\\&\leq e^{-\alpha(\xi-s+t)+2\sigma\int_{-t}^{\xi-s}y(\vartheta_{\upeta}\omega)\d\upeta}\|\u_{0}\|^2_{\H} + \frac{2}{\alpha}\int_{-t}^{\xi-s}e^{\alpha(\zeta+s-\xi)+2|y(\vartheta_{\zeta}\omega)|+2\sigma\int_{\zeta}^{\xi-s}y(\vartheta_{\upeta}\omega)\d\upeta}\|\f(\zeta+s)\|^2_{\L^2(\R^3)}\d\zeta.
		\end{align}
		Putting $\xi=s$ in \eqref{AB3}, we find
		\begin{align}\label{AB4}
			&\|\u(s,s-t,\vartheta_{-s}\omega,\u_{0})\|^2_{\H}+\frac{\alpha}{2}\int_{s-t}^{s}e^{\alpha(\zeta-s)-2\sigma\int^{\zeta}_{s}y(\vartheta_{\upeta-s}\omega)\d\upeta}\|\u(\zeta,s-t,\vartheta_{-s}\omega,\u_{0})\|^2_{\H}\d\zeta\nonumber\\&+2\mu\int_{s-t}^{s}e^{\alpha(\zeta-s)-2\sigma\int^{\zeta}_{s}y(\vartheta_{\upeta-s}\omega)\d\upeta}\|\nabla\u(\zeta,s-t,\vartheta_{-s}\omega,\u_{0})\|^2_{\H}\d\zeta\nonumber\\&+2\beta\int_{s-t}^{s}e^{(r-1)y(\vartheta_{\zeta-s}\omega)+\alpha(\zeta-s)-2\sigma\int^{\zeta}_{s}y(\vartheta_{\upeta-s}\omega)\d\upeta}\|\u(\zeta,s-t,\vartheta_{-s}\omega,\u_{0})\|^{r+1}_{\wi\L^{r+1}}\d\zeta\nonumber\\&\leq e^{-\alpha t+2\sigma\int_{-t}^{0}y(\vartheta_{\upeta}\omega)\d\upeta}\|\u_{0}\|^2_{\H} +\frac{2}{\alpha}\int_{-\infty}^{0}e^{\alpha\zeta+2|y(\vartheta_{\zeta}\omega)|+2\sigma\int_{\zeta}^{0}y(\vartheta_{\upeta}\omega)\d\upeta}\|\f(\zeta+s)\|^2_{\L^2(\R^3)}\d\zeta,
		\end{align}
		for all $s\leq\tau$. Since $\u_0\in D(s-t,\vartheta_{-t}\omega)$ and $D$ is backward tempered, it implies from \eqref{Z3} and the definition of backward temperedness \eqref{BackTem} that there exists a time $\mathcal{T}=\mathcal{T}(\tau,\omega,D)$ such that for all $t\geq \mathcal{T}$,
		\begin{align}\label{v_0}
			&e^{-\alpha t+2\sigma\int_{-t}^{0}y(\vartheta_{\upeta}\omega)\d\upeta}\sup_{s\leq \tau}\|\u_{0}\|^2_{\H}\nonumber\\&\leq e^{-\frac{\alpha}{3}t}\sup_{s\leq \tau}\|D(s-t,\vartheta_{-t}\omega)\|^2_{\H}\leq\frac{2}{\alpha}\int_{-\infty}^{0}e^{\alpha\zeta+2|y(\vartheta_{\zeta}\omega)|+2\sigma\int_{\zeta}^{0}y(\vartheta_{\upeta}\omega)\d\upeta}\|\f(\zeta+s)\|^2_{\L^2(\R^3)}\d\zeta.
		\end{align}
		Hence, by using \eqref{v_0} and taking supremum on $s\in(-\infty,\tau]$ in \eqref{AB4},  we reach at \eqref{AB1}. Now, using \eqref{AB3}, we estimate for $2<k_1<\infty$ and $k_2\geq0$
		\begin{align}\label{AB-5}
			&\int_{s-t}^{s}e^{k_2|y(\vartheta_{\zeta-s}\omega)|+\alpha(\zeta-s)-2\sigma\int^{\zeta}_{s}y(\vartheta_{\upeta-s}\omega)\d\upeta}\|\u(\zeta,s-t,\vartheta_{-s}\omega,\u_{0})\|^{k_1}_{\H}\d\zeta\nonumber\\&\leq C\int_{s-t}^{s}e^{k_2|y(\vartheta_{\zeta-s}\omega)|+\alpha(\zeta-s)+2\sigma\int_{\zeta-s}^{0}y(\vartheta_{\upeta}\omega)\d\upeta}\bigg[e^{-\frac{k_1}{2}\alpha(\zeta-s+t)+k_1\sigma\int_{-t}^{\zeta-s}y(\vartheta_{\upeta}\omega)\d\upeta}\|\u_{0}\|^{k_1}_{\H} \nonumber\\&\quad+ \bigg(\int\limits_{-t}^{\zeta-s}e^{\alpha(\zeta_1+s-\zeta)+2|y(\vartheta_{\zeta_1}\omega)|+2\sigma\int_{\zeta_1}^{\zeta-s}y(\vartheta_{\upeta}\omega)\d\upeta}\|\f(\zeta_1+s)\|^2_{\L^2(\R^3)}\d\zeta_1\bigg)^\frac{k_1}{2}\bigg]\d\zeta\nonumber\\&\leq C\int_{-\infty}^{0}e^{k_2\left|y(\vartheta_{\zeta}\omega)\right|+\frac{\alpha}{k_1}\zeta-(k_1-2)\sigma\int_{\zeta}^{0}y(\vartheta_{\upeta}\omega)\d\upeta}\d\zeta\times\bigg[ e^{-\frac{(k_1-1)\alpha}{k_1}t+k_1\sigma\int_{-t}^{0}y(\vartheta_{\upeta}\omega)\d\upeta}\|\u_{0}\|^{k_1}_{\H}\nonumber\\&\quad+ \bigg(\int_{-\infty}^{0}e^{\frac{2(k_1-1)\alpha}{k_1^2}\zeta_1+2|y(\vartheta_{\zeta_1}\omega)|+2\sigma\int_{\zeta_1}^{0}y(\vartheta_{\upeta}\omega)\d\upeta}\|\f(\zeta_1+s)\|^2_{\L^2(\R^3)}\d\zeta_1\bigg)^{\frac{k_1}{2}}\bigg].
		\end{align}
		Hence, using \eqref{Z3} and the backward-uniform temperedness property \eqref{BackTem} of $\u_0$ (see \eqref{v_0}), we obtain \eqref{AB11}.
	
		From \eqref{EI4} for $\zeta>s-t$ and $\u(\zeta):=\u(\zeta,s-t,\vartheta_{-s}\omega,\u_{0})$, we have
      	\begin{align}\label{AB12V}
      	&\frac{\d}{\d\zeta} \|\nabla\u(\zeta)\|^2_{\H}+ \left(\alpha-2\sigma y(\vartheta_{\zeta-s}\omega)\right)\|\nabla\u(\zeta)\|^2_{\H} \leq C\|\nabla\u(\zeta)\|^2_{\H}+ C e^{2|y(\vartheta_{\zeta-s}\omega)|}\|\f(\zeta)\|^2_{\L^2(\R^3)}.
      \end{align}
	In view of the variation of constants formula with respect to $\zeta\in(\xi_1,\xi)$ with $s-t<\xi_1\leq\xi\leq s$, we find
  \begin{align}\label{AB13V}
  	& \|\nabla\u(\xi,s-t,\vartheta_{-s}\omega,\u_{0})\|^2_{\H}\nonumber\\&\leq e^{\alpha(\xi_1-\xi)-2\sigma\int^{\xi_1}_{\xi}y(\vartheta_{\upeta-s}\omega)\d\upeta} \|\nabla\u(\xi_1,s-t,\vartheta_{-s}\omega,\u_{0})\|^2_{\H}\nonumber\\&\quad+C\int_{\xi_1}^{\xi}\bigg[e^{\alpha(\zeta-\xi)-2\sigma\int^{\zeta}_{\xi}y(\vartheta_{\upeta-s}\omega)\d\upeta}\|\nabla\u(\zeta,s-t,\vartheta_{-s}\omega,\u_{0})\|^2_{\H}\nonumber\\&\quad+e^{\alpha(\zeta-\xi)-2\sigma\int^{\zeta}_{\xi}y(\vartheta_{\upeta-s}\omega)\d\upeta+2|y(\vartheta_{\zeta-s}\omega)|}\|\f(\zeta)\|^2_{\L^2(\R^3)}\bigg]\d\zeta \nonumber\\&\leq e^{\alpha(\xi_1-\xi)-2\sigma\int^{\xi_1}_{\xi}y(\vartheta_{\upeta-s}\omega)\d\upeta} \|\nabla\u(\xi_1,s-t,\vartheta_{-s}\omega,\u_{0})\|^2_{\H}\nonumber\\&\quad+C\int_{s-t}^{\xi}\bigg[e^{\alpha(\zeta-\xi)-2\sigma\int^{\zeta}_{\xi}y(\vartheta_{\upeta-s}\omega)\d\upeta}\|\nabla\u(\zeta,s-t,\vartheta_{-s}\omega,\u_{0})\|^2_{\H}\nonumber\\&\quad+e^{\alpha(\zeta-\xi)-2\sigma\int^{\zeta}_{\xi}y(\vartheta_{\upeta-s}\omega)\d\upeta+2|y(\vartheta_{\zeta-s}\omega)|}\|\f(\zeta)\|^2_{\L^2(\R^3)}\bigg]\d\zeta
  	\nonumber\\&\leq e^{\alpha(\xi_1-\xi)-2\sigma\int^{\xi_1}_{\xi}y(\vartheta_{\upeta-s}\omega)\d\upeta} \|\nabla\u(\xi_1,s-t,\vartheta_{-s}\omega,\u_{0})\|^2_{\H}+Ce^{-\alpha(\xi-s+t)+2\sigma\int_{s-t}^{\xi}y(\vartheta_{\upeta-s}\omega)\d\upeta}\|\u_{0}\|^2_{\H}\nonumber\\&\quad+C\int_{s-t}^{\xi}e^{\alpha(\zeta-\xi)-2\sigma\int^{\zeta}_{\xi}y(\vartheta_{\upeta-s}\omega)\d\upeta+2|y(\vartheta_{\zeta-s}\omega)|}\|\f(\zeta)\|^2_{\L^2(\R^3)}\d\zeta.
 \end{align}
Integrating \eqref{AB13V} from $s-t$ to $\xi$ with respect to $\xi_1$ and using \eqref{AB3}, we obtain
\begin{align}\label{AB14V}
	 \|\nabla\u(\xi,s-t,\vartheta_{-s}\omega,\u_{0})\|^2_{\H}&\leq C\bigg\{\frac{1}{\xi-s+t}+1\bigg\}\bigg[e^{-\alpha(\xi-s+t)+2\sigma\int_{s-t}^{\xi}y(\vartheta_{\upeta-s}\omega)\d\upeta}\|\u_{0}\|^2_{\H}\nonumber\\&\quad+\int_{s-t}^{\xi}e^{\alpha(\zeta-\xi)-2\sigma\int^{\zeta}_{\xi}y(\vartheta_{\upeta-s}\omega)\d\upeta+2|y(\vartheta_{\zeta-s}\omega)|}\|\f(\zeta)\|^2_{\L^2(\R^3)}\d\zeta\bigg].
\end{align}
  
  Now from \eqref{AB14V}, we estimate for $2<k_1<\infty$ 
  \begin{align}\label{AB15V}
  	&\int_{s-t}^{s}e^{k_2|y(\vartheta_{\xi-s}\omega)|+\alpha(\xi-s)-2\sigma\int^{\xi}_{s}y(\vartheta_{\upeta-s}\omega)\d\upeta}\|\nabla\u(\xi,s-t,\vartheta_{-s}\omega,\u_{0})\|^{k_1}_{\H}\d\xi\nonumber\\&\leq C\int_{s-t}^{s}e^{k_2|y(\vartheta_{\xi-s}\omega)|+\alpha(\xi-s)+2\sigma\int_{\xi-s}^{0}y(\vartheta_{\upeta}\omega)\d\upeta}\bigg\{\frac{1}{\xi-s+t}+1\bigg\}^{\frac{k_1}{2}}\bigg[e^{-\frac{k_1}{2}\alpha(\xi-s+t)+k_1\sigma\int_{-t}^{\xi-s}y(\vartheta_{\upeta}\omega)\d\upeta}\|\u_{0}\|^{k_1}_{\H} \nonumber\\&\quad+ \bigg(\int\limits_{-t}^{\xi-s}e^{\alpha(\zeta+s-\xi)+2|y(\vartheta_{\zeta}\omega)|+2\sigma\int_{\zeta}^{\xi-s}y(\vartheta_{\upeta}\omega)\d\upeta}\|\f(\zeta+s)\|^2_{\L^2(\R^3)}\d\zeta\bigg)^\frac{k_1}{2}\bigg]\d\xi\nonumber\\&\leq C\int_{-t}^{0}e^{k_2|y(\vartheta_{\xi}\omega)|+\frac{\alpha}{k_1}\xi-(k_1-2)\sigma\int_{\xi}^{0}y(\vartheta_{\upeta}\omega)\d\upeta}\bigg\{\frac{1}{\xi+t}+1\bigg\}^{\frac{k_1}{2}}\d\xi\times\bigg[ e^{-\frac{(k_1-1)\alpha}{k_1}t+k_1\sigma\int_{-t}^{0}y(\vartheta_{\upeta}\omega)\d\upeta}\|\u_{0}\|^{k_1}_{\H}\nonumber\\&\quad+ \bigg(\int_{-\infty}^{0}e^{\frac{2(k_1-1)\alpha}{k_1^2}\zeta+2|y(\vartheta_{\zeta}\omega)|+2\sigma\int_{\zeta}^{0}y(\vartheta_{\upeta}\omega)\d\upeta}\|\f(\zeta+s)\|^2_{\L^2(\R^3)}\d\zeta\bigg)^{\frac{k_1}{2}}\bigg].
  \end{align}
  Since $$\int_{-t}^{0}e^{c\xi}\bigg[\frac{1}{\xi+t}+1\bigg]^{\frac{k_1}{2}}\d\xi\to\frac{1}{c}\ \text{ as }\ t\to\infty,$$ for all $c>0$, using \eqref{Z3} and the backward-uniform temperedness property \eqref{BackTem} of $\u_0$ (see \eqref{v_0}), we deduce \eqref{AB12}, as required.
  
  Again from \eqref{EI4}, we write for $\zeta>s-t$
  	\begin{align*}
  	&\frac{\d}{\d\zeta} \|\nabla\u(\zeta)\|^2_{\H}+ \left(\alpha-2\sigma y(\vartheta_{\zeta-s}\omega)\right)\|\nabla\u(\zeta)\|^2_{\H}+C^*e^{(r-1)y(\vartheta_{\zeta-s}\omega)}\|\u(\zeta)\|^{r+1}_{\wi\L^{3(r+1)}} \nonumber\\&\leq C\|\nabla\u(\zeta)\|^2_{\H}+ C e^{2|y(\vartheta_{\zeta-s}\omega)|}\|\f(\zeta)\|^2_{\L^2(\R^3)},
  \end{align*}
  which implies that
 	\begin{align}\label{AB16V}
  	&(\zeta-s+t)\frac{\d}{\d\zeta} \bigg[e^{\alpha(\zeta-s)-2\sigma\int^{\zeta}_{s}y(\vartheta_{\upeta-s}\omega)\d\upeta}\|\nabla\u(\zeta)\|^2_{\H}\bigg]  \nonumber\\&+C^* (\zeta-s+t) e^{(r-1)y(\vartheta_{\zeta-s}\omega)+\alpha(\zeta-s)-2\sigma\int^{\zeta}_{s}y(\vartheta_{\upeta-s}\omega)\d\upeta}\|\u(\zeta)\|^{r+1}_{\wi\L^{3(r+1)}} \nonumber\\&\leq C(\zeta-s+t)\bigg[e^{\alpha(\zeta-s)-2\sigma\int^{\zeta}_{s}y(\vartheta_{\upeta-s}\omega)\d\upeta}\|\nabla\u(\zeta)\|^2_{\H}\nonumber\\&\quad+ e^{\alpha(\zeta-s)-2\sigma\int^{\zeta}_{s}y(\vartheta_{\upeta-s}\omega)\d\upeta+2|y(\vartheta_{\zeta-s}\omega)|}\|\f(\zeta)\|^2_{\L^2(\R^3)}\bigg].
  \end{align}
  We know that
  \begin{align}\label{AB17V}
  	&(\zeta-s+t)\frac{\d}{\d\zeta}\bigg[e^{\alpha(\zeta-s)-2\sigma\int^{\zeta}_{s}y(\vartheta_{\upeta-s}\omega)\d\upeta} \|\nabla\u(\zeta)\|^2_{\H}\bigg]\nonumber\\&=\frac{\d}{\d\zeta}\bigg[(\zeta-s+t) e^{\alpha(\zeta-s)-2\sigma\int^{\zeta}_{s}y(\vartheta_{\upeta-s}\omega)\d\upeta} \|\nabla\u(\zeta)\|^2_{\H}\bigg]- e^{\alpha(\zeta-s)-2\sigma\int^{\zeta}_{s}y(\vartheta_{\upeta-s}\omega)\d\upeta}\|\nabla\u(\zeta)\|^2_{\H}.
  \end{align}
  From \eqref{AB16V} and \eqref{AB17V}, we infer
  \begin{align}\label{AB18V}
  	&\int_{s-t}^{s} (\zeta-s+t) e^{(r-1)y(\vartheta_{\zeta-s}\omega)+\alpha(\zeta-s)-2\sigma\int^{\zeta}_{s}y(\vartheta_{\upeta-s}\omega)\d\upeta}\|\u(\zeta,s-t,\vartheta_{-s}\omega,\u_{0})\|^{r+1}_{\wi\L^{3(r+1)}}\d\zeta\nonumber\\&\leq C\int_{s-t}^{s}(\zeta-s+t+1)\bigg[e^{\alpha(\zeta-s)-2\sigma\int^{\zeta}_{s}y(\vartheta_{\upeta-s}\omega)\d\upeta}\|\nabla\u(\zeta)\|^2_{\H}\nonumber\\&\quad+ e^{\alpha(\zeta-s)-2\sigma\int^{\zeta}_{s}y(\vartheta_{\upeta-s}\omega)\d\upeta+2|y(\vartheta_{\zeta-s}\omega)|}\|\f(\zeta)\|^2_{\L^2(\R^3)}\bigg]\d\zeta\nonumber\\&\leq C(t+1)\int_{s-t}^{s}\bigg[e^{\alpha(\zeta-s)-2\sigma\int^{\zeta}_{s}y(\vartheta_{\upeta-s}\omega)\d\upeta}\|\nabla\u(\zeta)\|^2_{\H}\nonumber\\&\quad+ e^{\alpha(\zeta-s)-2\sigma\int^{\zeta}_{s}y(\vartheta_{\upeta-s}\omega)\d\upeta+2|y(\vartheta_{\zeta-s}\omega)|}\|\f(\zeta)\|^2_{\L^2(\R^3)}\bigg]\d\zeta \nonumber\\&\leq C(t+1)\bigg[e^{-\alpha t+2\sigma\int_{-t}^{0}y(\vartheta_{\upeta}\omega)\d\upeta}\|\u_{0}\|^2_{\H}\nonumber\\&\quad+ \int_{-\infty}^{0} e^{\alpha\zeta-2\sigma\int^{\zeta}_{0}y(\vartheta_{\upeta}\omega)\d\upeta+2|y(\vartheta_{\zeta}\omega)|}\|\f(\zeta+s)\|^2_{\L^2(\R^3)}\d\zeta\bigg],
  \end{align}
where we have used \eqref{AB3} also. This completes the proof.
	\end{proof}

	\begin{proposition}\label{IRAS}
		For both the cases given in Table \ref{Table}, suppose that Assumption \ref{Hypo_f-N} is satisfied. For $K(\tau,\omega)$ same as in \eqref{AB1}, we have
		\vskip 2mm
		\noindent
		\emph{(i)} There is an increasing pullback $\mathfrak{D}$-random absorbing set $\mathcal{K}$ given by
		\begin{align}\label{IRAS1}
			\mathcal{K}(\tau,\omega):=\left\{\v\in\H:\|\v\|^2_{\H}\leq \frac{4e^{2y(\omega)}}{\alpha}\sup_{s\leq \tau}K(s,\omega)\right\}, \ \text{ for all } \ \tau\in\R \text{ and }
			\omega\in\Omega.
		\end{align}
		Moreover, $\mathcal{K}$ is backward-uniformly tempered with arbitrary rate, that is, $\mathcal{K}\in{\mathfrak{D}}$.
		\vskip 2mm
		\noindent
		\emph{(ii)} There is a $\mathfrak{B}$-pullback \textbf{random} absorbing set $\widetilde{\mathcal{K}}$ given by
		
		\begin{align}\label{IRAS11}
			\widetilde{\mathcal{K}}(\tau,\omega):=\left\{\v\in\H:\|\v\|^2_{\H}\leq \frac{4e^{2y(\omega)}}{\alpha}K(\tau,\omega)\right\}\in\mathfrak{B}, \ \text{ for all } \ \tau\in\R \text{ and }
			\omega\in\Omega.
		\end{align}
	\end{proposition}
	\begin{proof}
		See the proof \cite[Proposition 4.6]{RKM}.
	\end{proof}

	\subsection{Backward uniform tail-estimates and backward flattening-property}\label{BUTE-BFP}
	In this subsection, we show that the solution of the system \eqref{2} satisfies the \emph{backward uniform tail-estimates} and \emph{backward flattening-property} for $r\in(3,\infty)$ with any $\beta,\mu>0$ and $r=3$ with $2\beta\mu\geq1$. These estimates help us to obtain the backward uniform pullback $\mathfrak{D}$-asymptotic compactness of $\Phi$. We use a cut-off function technique to obtain \emph{backward uniform tail-estimates} and \emph{backward flattening-property}. The following lemma provides the backward uniform tail-estimates for the solutions of the system \eqref{2}.
	\begin{lemma}\label{largeradius}
		For both the cases given in Table \ref{Table}, suppose that Assumption \ref{Hypo_f-N} holds. Then, for any $(\tau,\omega,D)\in\R\times\Omega\times\mathfrak{D},$ the solution of \eqref{2}  satisfies
		\begin{align}\label{ep}
			&\lim_{k,t\to+\infty}\sup_{s\leq \tau}\sup_{\u_{0}\in D(s-t,\vartheta_{-t}\omega)}\|\u(s,s-t,\vartheta_{-s}\omega,\u_{0})\|^2_{\mathbb{L}^2(\mathcal{O}^{c}_{k})}=0,
		\end{align}
		where $\mathcal{O}_{k}=\{x\in\R^3:|x|\leq k\},$ $k\in\mathbb{N}$.
	\end{lemma}
	\begin{proof}
		Let $\uprho$ be a smooth function such that $0\leq\uprho(\xi)\leq 1,$ for $\xi\in\R^+$ and
		\begin{align*}
			\uprho(\xi)=\begin{cases*}
				0, \text{ for }0\leq \xi\leq 1,\\
				1, \text{ for } \xi\geq2.
			\end{cases*}
		\end{align*}
		Then, there exists a positive constant $C$ such that $|\uprho'(\xi)|\leq C,$ for all $\xi\in\R^+$. Taking divergence to the first equation of \eqref{2} formally, we obtain   in the weak sense
		\begin{align*}
			-e^{-y(\vartheta_{t}\omega)}\Delta p&=e^{y(\vartheta_{t}\omega)}\nabla\cdot\left[\big(\u\cdot\nabla\big)\u\right]+\beta e^{(r-1)y(\vartheta_{t}\omega)}\nabla\cdot\left[|\u|^{r-1}\u\right]- e^{-y(\vartheta_{t}\omega)}\nabla\cdot\f\\
			&=e^{y(\vartheta_{t}\omega)}\nabla\cdot\left[\nabla\cdot\big(\u\otimes\u\big)\right]+\beta e^{(r-1)y(\vartheta_{t}\omega)}\nabla\cdot\left[|\u|^{r-1}\u\right]- e^{-y(\vartheta_{t}\omega)}\nabla\cdot\f\\
			&=e^{y(\vartheta_{t}\omega)}\sum_{i,j=1}^{3}\frac{\partial^2}{\partial x_i\partial x_j}\big(u_iu_j\big)+\beta e^{(r-1)y(\vartheta_{t}\omega)}\nabla\cdot\left[|\u|^{r-1}\u\right]-  e^{-y(\vartheta_{t}\omega)}\nabla\cdot\f.
		\end{align*}
	This  implies
		\begin{align}\label{p-value}
			p=(-\Delta)^{-1}\left[e^{2y(\vartheta_{t}\omega)}\sum_{i,j=1}^{3}\frac{\partial^2}{\partial x_i\partial x_j}\big(u_iu_j\big)+\beta e^{ry(\vartheta_{t}\omega)}\nabla\cdot\left[|\u|^{r-1}\u\right]- \nabla\cdot\f\right]
		\end{align}
		in the weak sense.	Taking the inner product to the first equation of \eqref{2} with $\uprho\left(\frac{|x|^2}{k^2}\right)\u$, we have
		\begin{align}\label{ep1}
			&\frac{1}{2} \frac{\d}{\d t} \int_{\R^3}\uprho\left(\frac{|x|^2}{k^2}\right)|\u|^2\d x\nonumber\\&= \mu \int_{\R^3}(\Delta\u) \uprho\left(\frac{|x|^2}{k^2}\right) \u \d x-\alpha \int_{\R^3}\uprho\left(\frac{|x|^2}{k^2}\right)|\u|^2\d x-e^{y(\vartheta_{t}\omega)}b\left(\u,\u,\uprho\left(\frac{|x|^2}{k^2}\right)\u\right)\nonumber\\&\quad-\beta e^{(r-1)y(\vartheta_{t}\omega)} \int_{\R^3}\left|\u\right|^{r+1}\uprho\left(\frac{|x|^2}{k^2}\right)\d x-e^{-y(\vartheta_{t}\omega)}\int_{\R^3}(\nabla p)\uprho\left(\frac{|x|^2}{k^2}\right)\u\d x\nonumber\\&\quad+ e^{-y(\vartheta_{t}\omega)} \int_{\R^3}\f\uprho\left(\frac{|x|^2}{k^2}\right)\u\d x +\sigma y(\vartheta_{t}\omega)\int_{\R^3}\uprho\left(\frac{|x|^2}{k^2}\right)|\u|^2\d x.
		\end{align}
		Let us now estimate each term on the right hand side of \eqref{ep1}. Integration by parts and divergence free condition of $\u(\cdot)$ help us to obtain
		\begin{align}\label{ep2}
			&\mu \int_{\R^3}(\Delta\u) \uprho\left(\frac{|x|^2}{k^2}\right) \u \d x+\mu \int_{\R^3}|\nabla\u|^2 \uprho\left(\frac{|x|^2}{k^2}\right)  \d x\nonumber\\&= -\mu \int_{\R^3} \uprho'\left(\frac{|x|^2}{k^2}\right)\frac{2}{k^2}(x\cdot\nabla) \u\cdot\u \d x\nonumber\\&\leq  \frac{2\sqrt{2}\mu}{k} \int\limits_{k\leq|x|\leq \sqrt{2}k}\left|\u\right| \left|\uprho'\left(\frac{|x|^2}{k^2}\right)\right|\left|\nabla \u\right| \d x\leq \frac{C}{k} \int_{\R^3}\left|\u\right| \left|\nabla \u\right| \d x\leq \frac{C}{k} \left[\|\u\|^2_{\H}+\|\nabla\u\|^2_{\H}\right],
		\end{align}
		and
		\begin{align}\label{ep3}
			-e^{y(\vartheta_{t}\omega)}b\left(\u,\u,\uprho\left(\frac{|x|^2}{k^2}\right)\u\right)&=e^{y(\vartheta_{t}\omega)}\int_{\R^3} \uprho'\left(\frac{|x|^2}{k^2}\right)\frac{x}{k^2}\cdot\u |\u|^2 \d x\nonumber\\&\leq \frac{\sqrt{2}e^{|y(\vartheta_{t}\omega)|}}{k} \int\limits_{k\leq|x|\leq \sqrt{2}k} \left|\uprho'\left(\frac{|x|^2}{k^2}\right)\right| |\u|^3 \d x\leq \frac{C}{k} e^{|y(\vartheta_{t}\omega)|}\|\u\|^2_{\wi\L^4}\|\u\|_{\H}\nonumber\\&\leq \frac{C}{k} e^{|y(\vartheta_{t}\omega)|} \|\u\|^{\frac{3}{2}}_{\H}\|\nabla\u\|_{\H}^{\frac{3}{2}}\leq \frac{C}{k}\bigg[\|\nabla\u\|^{2}_{\H}+e^{4|y(\vartheta_{t}\omega)|}\|\u\|_{\H}^{4}\bigg],		
		\end{align}
		where we have used Ladyzhenskaya's  and Young's inequalities in the penultimate and final inequalities, respectively. Using integration by parts, divergence free condition and \eqref{p-value}, we obtain
		\begin{align}\label{361}
			-&e^{-y(\vartheta_{t}\omega)}\int_{\R^3}(\nabla p)\uprho\left(\frac{|x|^2}{k^2}\right)\u\d x=e^{-y(\vartheta_{t}\omega)}\int_{\R^3}p\uprho'\left(\frac{|x|^2}{k^2}\right)\frac{2}{k^2}(x\cdot\u)\d x\nonumber\\& \leq \frac{Ce^{|y(\vartheta_{t}\omega)|}}{k}\int\limits_{\R^3}\left|(-\Delta)^{-1}\left[\nabla\cdot\left[\nabla\cdot\big(\u\otimes\u\big)\right]\right]\right|\cdot\left|\u\right|\d x \nonumber\\&\quad+ \frac{Ce^{(r-1)y(\vartheta_{t}\omega)}}{k}\int\limits_{\R^3}\left|(-\Delta)^{-1}\left[\nabla\cdot\left[|\u|^{r-1}\u\right]\right]\right|\cdot\left|\u\right|\d x+\frac{Ce^{|y(\vartheta_{t}\omega)|}}{k}\int_{\R^3}|(-\Delta)^{-1}[\nabla\cdot\f]|\cdot|\u|\d  x\\& =: \frac{C}{k}\left[S_1+e^{(r-1)y(\vartheta_{t}\omega)}S_2+S_3\right],\nonumber
		\end{align}
	where $S_1$, $S_2$ and $S_3$ are the terms appearing in the right hand side of \eqref{361}. 
		\vskip2mm
		\noindent
		\textbf{Estimate of $S_1$}: Using H\"older's inequality,  Plancherel's theorem, Ladyzhenskaya's and Young's inequalities, respectively, we get 
		\begin{align}
			|S_1| &\leq e^{|y(\vartheta_{t}\omega)|}\left\|(-\Delta)^{-1}\left[\nabla\cdot\left[\nabla\cdot\big(\u\otimes\u\big)\right]\right]\right\|_{\L^2(\R^3)}\|\u\|_{\H}\leq e^{|y(\vartheta_{t}\omega)|}\|\u\|^2_{\wi\L^4}\|\u\|_{\H}\nonumber\\&\leq C e^{|y(\vartheta_{t}\omega)|} \|\u\|^{\frac{3}{2}}_{\H}\|\nabla\u\|_{\H}^{\frac{3}{2}}\leq C\bigg[\|\nabla\u\|^{2}_{\H}+e^{4|y(\vartheta_{t}\omega)|}\|\u\|_{\H}^{6}\bigg].
		\end{align}
		\vskip2mm
		\noindent
		\textbf{Estimate of $S_2$:} Applying H\"older's inequality, Plancherel's theorem, \cite[Theorem 1.38]{BCD} (see Remark \ref{Hdot} above), interpolation and Young's inequalities, we obtain
		\begin{align}\label{S_2(d,r)}
			|S_2|&\leq 
				\|(-\Delta)^{-1}\left[\nabla\cdot\left[|\u|^{r-1}\u\right]\right]\|_{\L^{2}(\R^3)}\|\u\|_{\H} \leq \||\u|^{r-1}\u\|_{\dot{\mathbb{H}}^{-1}(\R^3)}\|\u\|_{\H}\nonumber\\&\leq \||\u|^{r-1}\u\|_{\mathbb{L}^{\frac{6}{5}}(\R^3)}\|\u\|_{\H} = \|\u\|^r_{\mathbb{L}^{\frac{6r}{5}}(\R^3)}\|\u\|_{\H}
			\nonumber\\&\leq \|\u\|^{\frac{6(r+1)}{3r+1}}_{\H}\|\u\|_{\wi\L^{3(r+1)}}^{\frac{(3r-5)(r+1)}{3r+1}}
			\leq \|\u\|^{2(r+1)}_{\H}+\|\u\|_{\wi\L^{3(r+1)}}^{\frac{(3r-5)(r+1)}{3r-4}}.
		\end{align}
		\vskip2mm
	\noindent
	\textbf{Estimate of $S_3$}: Applying H\"older's inequality, Plancherel's theorem, and Young's inequalities, we find 
	\begin{align}\label{S_3(d,r)}
		\left|S_3\right|&\leq 
			Ce^{|y(\vartheta_{t}\omega)|}\|(-\Delta)^{-1}\left[\nabla\cdot \f\right]\|_{\L^{2}(\R^3)}\|\u\|_{\H} \nonumber\\&\leq 
		Ce^{|y(\vartheta_{t}\omega)|}\|\f\|_{\dot{\mathbb{H}}^{-1}(\R^3)}\|\u\|_{\H}  \leq 
	Ce^{2|y(\vartheta_{t}\omega)|}\|\f\|^{2}_{\dot{\mathbb{H}}^{-1}(\R^3)}+C\|\u\|^2_{\H}.
	\end{align}

		Finally, we estimate the penultimate term on right hand side of \eqref{ep1} by using H\"older's and Young's inequalities as follows:
		\begin{align}
			e^{-y(\vartheta_{t}\omega)}\int_{\R^3}\f(x)\uprho\left(\frac{|x|^2}{k^2}\right)\u \d x\leq \frac{\alpha}{4} \int_{\R^3}\uprho\left(\frac{|x|^2}{k^2}\right)|\u|^2\d x +\frac{e^{2|y(\vartheta_{t}\omega)|}}{\alpha} \int_{\R^3}\uprho\left(\frac{|x|^2}{k^2}\right)|\f(x)|^2\d x.\label{ep4}
		\end{align}

		Combining \eqref{ep1}-\eqref{ep4}, we get
		\begin{align}\label{ep5}
			&	\frac{\d}{\d t} \|\u\|^2_{\mathbb{L}^2(\mathcal{O}_k^{c})}+ \left(\alpha-2\sigma y(\vartheta_{t}\omega)\right) \|\u\|^2_{\mathbb{L}^2(\mathcal{O}_k^c)} \nonumber\\ &\leq\frac{C}{k} \left[\|\u\|^2_{\V}+e^{4|y(\vartheta_{t}\omega)|}\|\u\|_{\H}^{6}+e^{(r-1)y(\vartheta_{t}\omega)}\|\u\|^{2(r+1)}_{\H}+e^{(r-1)y(\vartheta_{t}\omega)}\|\u\|_{\wi\L^{3(r+1)}}^{\frac{(3r-5)(r+1)}{3r-4}}\right.\nonumber\\&\quad\left.+e^{2|y(\vartheta_{t}\omega)|}\|\f\|^2_{\dot{\mathbb{H}}^{-1}(\R^3)}\right]+\frac{2e^{2|y(\vartheta_{t}\omega)|}}{\alpha} \int_{|x|\geq k}|\f(x)|^2\d x.
		\end{align}
	
		Applying the variation of constants formula to the above equation \eqref{ep5} on $(s-t,s)$ and replacing $\omega$ by $\vartheta_{-s}\omega$, for $s\leq\tau,$ $ t\geq 0$ and $\omega\in\Omega$, we find
			\begin{align}
			&\|\u(s,s-t,\vartheta_{-s}\omega,\u_{0})\|^2_{\mathbb{L}^2(\mathcal{O}_k^{c})} \nonumber\\& \leq e^{-\alpha t+2\sigma\int_{-t}^{0}y(\vartheta_{\upeta}\omega)\d\upeta}\|\u_{0}\|^2_{\H}+\frac{C}{k}\bigg[\int_{s-t}^{s}e^{\alpha(\zeta-s)-2\sigma\int^{\zeta}_{s}y(\vartheta_{\upeta-s}\omega)\d\upeta}\|\u(\zeta,s-t,\vartheta_{-s}\omega,\u_{0})\|^2_{\V}\d\zeta\nonumber\\&\quad+\int_{s-t}^{s}e^{4|y(\vartheta_{\zeta-s}\omega)|+\alpha(\zeta-s)-2\sigma\int^{\zeta}_{s}y(\vartheta_{\upeta-s}\omega)\d\upeta}\|\u(\zeta,s-t,\vartheta_{-s}\omega,\u_{0})\|^{6}_{\H}\d\zeta\nonumber\\&\quad+\int_{s-t}^{s}e^{(r-1)y(\vartheta_{\zeta-s}\omega)+\alpha(\zeta-s)-2\sigma\int^{\zeta}_{s}y(\vartheta_{\upeta-s}\omega)\d\upeta}\|\u(\zeta,s-t,\vartheta_{-s}\omega,\u_{0})\|^{2(r+1)}_{\H}\nonumber\\&\quad+\int_{s-t}^{s}e^{(r-1)y(\vartheta_{\zeta-s}\omega)+\alpha(\zeta-s)-2\sigma\int^{\zeta}_{s}y(\vartheta_{\upeta-s}\omega)\d\upeta}\|\u(\zeta,s-t,\vartheta_{-s}\omega,\u_{0})\|^{\frac{(3r-5)(r+1)}{3r-4}}_{\wi\L^{3(r+1)}}\d\zeta\nonumber\\&\quad+\int_{-t}^{0}e^{\alpha\zeta+2|y(\vartheta_{\zeta}\omega)|+2\sigma\int_{\zeta}^{0}y(\vartheta_{\upeta}\omega)\d\upeta}\|\f(\zeta+s)\|^2_{\dot{\mathbb{H}}^{-1}(\R^3)}\d\zeta\bigg]\nonumber\\&\quad+C\int_{-t}^{0}e^{\alpha\zeta+2|y(\vartheta_{\zeta}\omega)|+2\sigma\int_{\zeta}^{0}y(\vartheta_{\upeta}\omega)\d\upeta} \int_{|x|\geq k}|\f(x,\zeta+s)|^2\d x\d\zeta.\label{ep6}
		\end{align}
	
		Now, we obtain from \eqref{AB13}
	\begin{align*}
		&\int_{s-t}^{s}e^{(r-1)y(\vartheta_{\zeta-s}\omega)+\alpha(\zeta-s)-2\sigma\int^{\zeta}_{s}y(\vartheta_{\upeta-s}\omega)\d\upeta}\|\u(\zeta,s-t,\vartheta_{-s}\omega,\u_{0})\|^{\frac{(3r-5)(r+1)}{3r-4}}_{\wi\L^{3(r+1)}}\d\zeta
		\nonumber\\&\leq\left[\int_{s-t}^{s}\left\{\frac{1}{\zeta-s+t}\right\}^{3r-5}e^{(r-1)y(\vartheta_{\zeta-s}\omega)+\alpha(\zeta-s)-2\sigma\int^{\zeta}_{s}y(\vartheta_{\upeta-s}\omega)\d\upeta}\d\zeta\right]^{\frac{1}{3r-4}}\nonumber\\&\quad\times\left[\int_{s-t}^{s}(\zeta-s+t)e^{(r-1)y(\vartheta_{\zeta-s}\omega)+\alpha(\zeta-s)-2\sigma\int^{\zeta}_{s}y(\vartheta_{\upeta-s}\omega)\d\upeta}\|\u(\zeta,s-t,\vartheta_{-s}\omega,\u_{0})\|^{r+1}_{\wi\L^{3(r+1)}}\d\zeta\right]^{\frac{3r-5}{3r-4}}\nonumber\\&\leq C\left[(t+1)^{3r-5}\int_{-t}^{0}\left\{\frac{1}{\zeta+t}\right\}^{3r-5}e^{(r-1)y(\vartheta_{\zeta}\omega)+\alpha\zeta-2\sigma\int^{\zeta}_{0}y(\vartheta_{\upeta}\omega)\d\upeta}\d\zeta\right]^{\frac{1}{3r-4}}\nonumber\\&\quad\times\left[e^{-\alpha t+2\sigma\int_{-t}^{0}y(\vartheta_{\upeta}\omega)\d\upeta}\|\u_{0}\|^2_{\H}+ \int_{-\infty}^{0} e^{\alpha\zeta-2\sigma\int^{\zeta}_{0}y(\vartheta_{\upeta}\omega)\d\upeta+2|y(\vartheta_{\zeta}\omega)|}\|\f(\zeta+s)\|^2_{\L^2(\R^3)}\d\zeta\right]^{\frac{3r-5}{3r-4}}.
	\end{align*}
	
	We also have
	\begin{align*}
		\lim_{t\to +\infty}\left[(t+1)^{3r-5}\int_{-t}^{0}\left\{\frac{1}{\zeta+t}\right\}^{3r-5}e^{\frac{\alpha}{2}\zeta}\d\zeta\right]=\frac{2}{\alpha},
	\end{align*}
	which gives  
	\begin{align}\label{ep7}
		&\lim_{t\to +\infty}\sup_{s\leq \tau}\int_{s-t}^{s}e^{(r-1)y(\vartheta_{\zeta-s}\omega)+\alpha(\zeta-s)-2\sigma\int^{\zeta}_{s}y(\vartheta_{\upeta-s}\omega)\d\upeta}\|\u(\zeta,s-t,\vartheta_{-s}\omega,\u_{0})\|^{\frac{3(r-1)(r+1)}{3r-2}}_{\wi\L^{3(r+1)}}\d\zeta
		\nonumber\\&\leq C\left[ \sup_{s\leq \tau}\int_{-\infty}^{0} e^{\alpha\zeta-2\sigma\int^{\zeta}_{0}y(\vartheta_{\upeta}\omega)\d\upeta+2|y(\vartheta_{\zeta}\omega)|}\|\f(\zeta+s)\|^2_{\L^2(\R^3)}\d\zeta\right]^{\frac{3(r-1)}{3r-2}}.
	\end{align}
	
		Now using \eqref{Z3}, the definition of backward-uniform temperedness \eqref{BackTem} (for the first term on right hand side  of \eqref{ep6}), Lemma \ref{Absorbing}, convergence in \eqref{ep7} and Hypothesis \ref{Hypo_f-N} (for the second term on right hand side of \eqref{ep6}), and \eqref{f3-N} (for the final term on right hand side of \eqref{ep6}), we immediately complete the proof.
	\end{proof}
	

	The following lemma provides the backward flattening-property for the solution of the system \eqref{2}. For each $k\geq1$, we let
	\begin{align*}
		\varrho_k(x):= 1-\uprho\left(\frac{|x|^2}{k^2}\right),  \ \ x\in\R^3.
	\end{align*}
	Let $\bar{\u}:=\varrho_k\u$ for $\u:=\u(s,s-t,\omega,\u_{\tau})\in\H$. Then $\bar{\u}\in\L^2(\mathcal{O}_{\sqrt{2}k})$, which has the orthogonal decomposition:
	\begin{align}
		\bar{\u}=\P_{i}\bar{\u}\oplus(\I-\P_{i})\bar{\u}=:\bar{\u}_{i,1}+\bar{\u}_{i,2},  \ \ \text{ for eah } i\in\N,
	\end{align}
	where, $\P_i:\L^2(\mathcal{O}_{\sqrt{2}k})\to\H_{i}:=\mathrm{span}\{e_1,e_2,\cdots,e_i\}\subset\L^2(\mathcal{O}_{\sqrt{2}k})$ is a canonical projection and $\{e_j\}_{j=1}^{\infty}$ is the family of eigenfunctions for $-\Delta$ in $\L^2(\mathcal{O}_{\sqrt{2}k})$ with corresponding positive eigenvalues $\lambda_1\leq\lambda_2\leq\cdots\leq\lambda_j\to\infty$ as $j\to\infty$. We also have that $$\varrho_k \Delta\u=\Delta\bar{\u}-\u\Delta\varrho_k-2\nabla\varrho_k\cdot\nabla\u.$$ Furthermore, for $\boldsymbol{\psi}\in\H_0^1(\mathcal{O}_{\sqrt{2}k})$, we have
	\begin{align}\label{poin-i}
		\mathrm{P}_i\boldsymbol{\psi}&=\sum_{j=1}^{i}(\boldsymbol{\psi},e_j)e_j,\  \nabla\mathrm{P}_i\boldsymbol{\psi}=\A^{1/2}\mathrm{P}_i\boldsymbol{\psi}=\sum_{j=1}^{i}\lambda^{1/2}_j(\boldsymbol{\psi},e_j)e_j,\nonumber\\ (\I-\mathrm{P}_i)\boldsymbol{\psi}&=\sum_{j=i+1}^{\infty}(\boldsymbol{\psi},e_j)e_j,\ \nabla(\I-\P_i)\boldsymbol{\psi}=\A^{1/2}(\I-\P_i)\boldsymbol{\psi}=\sum_{j=i+1}^{\infty}\lambda^{1/2}_j(\boldsymbol{\psi},e_j)e_j, \nonumber\\
		\|\nabla(\I-\P_i)\boldsymbol{\psi}\|_{\L^2(\mathcal{O}_{\sqrt{2}k})}^2&=\sum_{j=i+1}^{\infty}\lambda_j|(\boldsymbol{\psi},e_j)|^2\geq \lambda_{i+1}\sum_{j=i+1}^{\infty}|(\boldsymbol{\psi},e_j)|^2=\lambda_{i+1}\|(\I-\P_i)\boldsymbol{\psi}\|_{\L^2(\mathcal{O}_{\sqrt{2}k})}^2.
	\end{align}

	\begin{lemma}\label{Flattening}
		For  both the cases given in Table \ref{Table}, suppose that Assumption \ref{Hypo_f-N} is satisfied. Let $(\tau,\omega,D)\in\R\times\Omega\times\mathfrak{D}$ and $k\geq1$ be fixed. Then
		\begin{align}\label{FL-P}
			\lim_{i,t\to+\infty}\sup_{s\leq \tau}\sup_{\u_{0}\in D(s-t,\vartheta_{-t}\omega)}\|(\I-\P_{i})\bar{\u}(s,s-t,\vartheta_{-s}\omega,\bar{\u}_{0,2})\|^2_{\L^2(\mathcal{O}_{\sqrt{2}k})}=0,
		\end{align}
		where $\bar{\u}_{0,2}=(\I-\P_{i})(\varrho_k\u_{0})$.
	\end{lemma}
	\begin{proof}
		Multiplying by $\varrho_k$ in the first equation of \eqref{2}, we rewrite the equation as:
		\begin{align}\label{FL1}
			&\frac{\d\bar{\u}}{\d t}-\mu \Delta\bar{\u}+e^{y(\vartheta_{t}\omega)}\varrho_k(\u\cdot\nabla)\u+\alpha\bar{\u}+\beta e^{(r-1)y(\vartheta_{t}\omega)}\varrho_k|\u|^{r-1}\u\nonumber\\&=-e^{-y(\vartheta_{t}\omega)}\varrho_k\nabla p+e^{-y(\vartheta_{t}\omega)}\varrho_k\f +\sigma y(\vartheta_t\omega)\bar{\u}-\mu\u\Delta\varrho_k-2\mu\nabla\varrho_k\cdot\nabla\u.
		\end{align}
		Applying $(\I-\P_i)$ to the equation \eqref{FL1} and taking the inner product of the resulting equation with $\bar{\u}_{i,2}$ in $\L^2(\mathcal{O}_{\sqrt{2}k})$ gives
		\begin{eqnarray}\label{FL2}
			&&\frac{1}{2}\frac{\d}{\d t}\|\bar{\u}_{i,2}\|^2_{\L^2(\mathcal{O}_{\sqrt{2}k})}+\mu\|\nabla\bar{\u}_{i,2}\|^2_{\L^2(\mathcal{O}_{\sqrt{2}k})} +\left(\alpha-\sigma y(\vartheta_t\omega)\right)\|\bar{\u}_{i,2}\|^2_{\L^2(\mathcal{O}_{\sqrt{2}k})}\nonumber\\&&\quad+\beta e^{(r-1)y(\vartheta_{t}\omega)}\|\left|\u\right|^{\frac{r-1}{2}}\bar{\u}_{i,2}\|^2_{\L^2(\mathcal{O}_{\sqrt{2}k})}\nonumber\\&&=-\underbrace{e^{y(\vartheta_{t}\omega)}\sum_{q,q'=1}^{3}\int_{\mathcal{O}_{\sqrt{2}k}}\left(\I-\P_i\right)\bigg[u_{q}\frac{\partial u_{q'}}{\partial x_q}\left\{\varrho_k(x)\right\}^2 u_{q'}\bigg]\d x}_{=:J_1}-\underbrace{\big(e^{-y(\vartheta_{t}\omega)}\varrho_k\nabla p,\bar{\u}_{i,2}\big)}_{=:J_2}\nonumber\\&&\qquad+\underbrace{\left\{\big(e^{-y(\vartheta_{t}\omega)}\varrho_k\f,\bar{\u}_{i,2}\big)-\mu\big(\u\Delta\varrho_k,\bar{\u}_{i,2}\big) - \mu\big(2\nabla\varrho_k\cdot\nabla\u,\bar{\u}_{i,2}\big)\right\}}_{=:J_3}\nonumber\\&&\qquad-\underbrace{\beta e^{(r-1)y(\vartheta_{t}\omega)}\int_{\mathcal{O}_{\sqrt{2}k}}|\u|^{r-1}(\mathrm{P}_{i}\bar{\u})\cdot\bar{\u}_{i,2}\d x}_{=:J_4}.
		\end{eqnarray}
		Next, we estimate each terms of \eqref{FL2} as follows: Using integration by parts, divergence free condition of $\u(\cdot)$, \eqref{poin-i} (without loss of generality (WLOG), one may assume that $\lambda_{i}\geq1$), H\"older's and Young's inequalities,  we deduce 
		\begin{align}
			\left|J_1\right|&=e^{y(\vartheta_{t}\omega)}\left|\int_{\mathcal{O}_{\sqrt{2}k}}\left(\I-\P_i\right)\bigg[\uprho'\left(\frac{|x|^2}{k^2}\right)\frac{x}{k^2}\cdot\varrho_k(x)\u|\u|^2\bigg]\d x\right|\nonumber\\&\leq Ce^{y(\vartheta_{t}\omega)}\|\bar{\u}_{i,2}\|_{\L^2(\mathcal{O}_{\sqrt{2}k})}\|\u\|^2_{\wi\L^4}\leq C\lambda_{i+1}^{-\frac{1}{2}}e^{y(\vartheta_{t}\omega)}\|\nabla\bar{\u}_{i,2}\|_{\L^2(\mathcal{O}_{\sqrt{2}k})}\|\u\|^{\frac{r-3}{r-1}}_{\H}\|\u\|^{\frac{r+1}{r-1}}_{\wi\L^{r+1}}\nonumber\\&\leq\frac{\mu}{8}\|\nabla\bar{\u}_{i,2}\|^2_{\L^2(\mathcal{O}_{\sqrt{2}k})}+C\lambda_{i+1}^{-1}\bigg[\|\u\|^2_{\H}+e^{(r-1)y(\vartheta_{t}\omega)}\|\u\|^{r+1}_{\wi\L^{r+1}}\bigg],\label{FL3}
			\\
			\left|J_3\right|&\leq C\bigg[e^{|y(\vartheta_{t}\omega)|}\|\f\|_{\L^2(\R^3)}+\|\u\|_{\H}+ \|\nabla\u\|_{\H}\bigg]\|\bar{\u}_{i,2}\|_{\L^2(\mathcal{O}_{\sqrt{2}k})}\nonumber\\&\leq C\lambda^{-\frac{1}{2}}_{i+1}\bigg[\|\u\|_{\H}+\|\nabla\u\|_{\H}+e^{|y(\vartheta_{t}\omega)|}\|\f\|_{\L^2(\R^3)}\bigg]\|\nabla\bar{\u}_{i,2}\|_{\L^2(\mathcal{O}_{\sqrt{2}k})}\nonumber\\&\leq\frac{\mu}{8}\|\nabla\bar{\u}_{i,2}\|^2_{\L^2(\mathcal{O}_{\sqrt{2}k})}+ C\lambda^{-1}_{i+1}\bigg[\|\u\|^2_{\H}+ \|\nabla\u\|^2_{\H}+e^{2|y(\vartheta_{t}\omega)|}\|\f\|^2_{\L^2(\R^3)}\bigg].\label{FL4}
		\end{align}
		An integration by parts, divergence free condition and \eqref{p-value} results to 
		\begin{align}
			\quad|J_2|&=\left|e^{-y(\vartheta_{t}\omega)}\int_{\mathcal{O}_{\sqrt{2}k}}(\I-\P_i)p\uprho'\left(\frac{|x|^2}{k^2}\right)\frac{4}{k^2}(x\cdot\bar{\u})\d x\right|\nonumber\\& \leq Ce^{y(\vartheta_{t}\omega)}\int_{\mathcal{O}_{\sqrt{2}k}}\left|(-\Delta)^{-1}\left[\nabla\cdot\left[\nabla\cdot\big(\u\otimes\u\big)\right]\right]\right|\cdot\left|\bar{\u}_{i,2}\right|\d x \nonumber\\&\quad+ Ce^{(r-1)y(\vartheta_{t}\omega)}\int_{\mathcal{O}_{\sqrt{2}k}}\left|(-\Delta)^{-1}\left[\nabla\cdot\left[|\u|^{r-1}\u\right]\right]\right|\cdot\left|\bar{\u}_{i,2}\right|\d x \nonumber\\&\quad+Ce^{|y(\vartheta_{t}\omega)|}\int_{\mathcal{O}_{\sqrt{2}k}}|(-\Delta)^{-1}[\nabla\cdot\f]|\cdot|\bar{\u}_{i,2}|\d  x \nonumber\\&=: C\left[\widetilde{S}_1+\widetilde{S}_2+\widetilde{S}_3\right].
		\end{align}
		\vskip1mm
		\noindent
		\textbf{Estimate of $\widetilde{S}_1$}: Applying H\"older's inequality, Plancherel's theorem, \eqref{poin-i}, interpolation and Young's inequalities, we get (similar to \eqref{FL3})
		\begin{align}
			|\widetilde{S}_1| &\leq e^{y(\vartheta_{t}\omega)}\left\|(-\Delta)^{-1}\left[\nabla\cdot\left[\nabla\cdot\big(\u\otimes\u\big)\right]\right]\right\|_{\L^2(\R^3)}\|\bar{\u}_{i,2}\|_{\L^2(\mathcal{O}_{\sqrt{2}k})}\nonumber\\&\leq  Ce^{y(\vartheta_{t}\omega)}\|\bar{\u}_{i,2}\|_{\L^2(\mathcal{O}_{\sqrt{2}k})}\|\u\|^2_{\wi\L^4}\nonumber\\&\leq C\lambda_{i+1}^{-\frac{1}{2}}e^{y(\vartheta_{t}\omega)}\|\nabla\bar{\u}_{i,2}\|_{\L^2(\mathcal{O}_{\sqrt{2}k})}\|\u\|^{\frac{r-3}{r-1}}_{\H}\|\u\|^{\frac{r+1}{r-1}}_{\wi\L^{r+1}}\nonumber\\&\leq\frac{\mu}{8}\|\nabla\bar{\u}_{i,2}\|^2_{\L^2(\mathcal{O}_{\sqrt{2}k})}+C\lambda_{i+1}^{-1}\bigg[\|\u\|^2_{\H}+e^{(r-1)y(\vartheta_{t}\omega)}\|\u\|^{r+1}_{\wi\L^{r+1}}\bigg].
		\end{align}
		\vskip1mm
		\noindent
		\textbf{Estimate of $\widetilde{S}_2$:} Applying H\"older's inequality, Plancherel's theorem, \eqref{poin-i}, \cite[Theorem 1.38]{BCD} (see Remark \ref{Hdot} above), interpolation and Young's inequalities, we obtain
		\begin{align}\label{S_2(d,r)p}
			|\widetilde{S}_2|&\leq e^{(r-1)y(\vartheta_{t}\omega)}
			\|(-\Delta)^{-1}\left[\nabla\cdot\left[|\u|^{r-1}\u\right]\right]\|_{\L^{2}(\R^3)}\|\bar{\u}_{i,2}\|_{\L^2(\mathcal{O}_{\sqrt{2}k})}\nonumber\\& \leq Ce^{(r-1)y(\vartheta_{t}\omega)} \lambda_{i+1}^{-\frac{1}{2}} \||\u|^{r-1}\u\|_{\dot{\mathbb{H}}^{-1}(\R^3)}\|\nabla\bar{\u}_{i,2}\|_{\L^2(\mathcal{O}_{\sqrt{2}k})}\nonumber\\&\leq Ce^{(r-1)y(\vartheta_{t}\omega)}\lambda_{i+1}^{-\frac{1}{2}} \||\u|^{r-1}\u\|_{\mathbb{L}^{\frac{6}{5}}(\R^3)}\|\nabla\bar{\u}_{i,2}\|_{\L^2(\mathcal{O}_{\sqrt{2}k})}\nonumber\\& \leq Ce^{(r-1)y(\vartheta_{t}\omega)}\lambda_{i+1}^{-\frac{1}{2}} \|\u\|^r_{\mathbb{L}^{\frac{6r}{5}}(\R^3)}\|\u\|_{\V}
			\nonumber\\&\leq Ce^{(r-1)y(\vartheta_{t}\omega)}\lambda_{i+1}^{-\frac{1}{2}} \|\u\|^{\frac{3r+5}{3r+1}}_{\H}\|\u\|_{\wi\L^{3(r+1)}}^{\frac{(3r-5)(r+1)}{3r+1}}\|\u\|_{\V}\nonumber \\&\leq Ce^{(r-1)y(\vartheta_{t}\omega)} \lambda_{i+1}^{-\frac{1}{2}} \|\u\|^{\frac{6(r+1)}{3r+1}}_{\V}\|\u\|_{\wi\L^{3(r+1)}}^{\frac{(3r-5)(r+1)}{3r+1}}
			\nonumber\\& \leq C e^{(r-1)y(\vartheta_{t}\omega)}\lambda_{i+1}^{-\frac{1}{2}} \bigg[\|\u\|^{2(r+1)}_{\V}+\|\u\|_{\wi\L^{3(r+1)}}^{\frac{(3r-5)(r+1)}{3r-4}} \bigg].
		\end{align}
			\vskip1mm
		\noindent
		\textbf{Estimate of $\widetilde{S}_3$}: Similar to \eqref{S_3(d,r)}, we find 
		\begin{align}\label{S_3(d,r)p}
			|\widetilde{S}_3|&\leq Ce^{|y(\vartheta_{t}\omega)|}\|(-\Delta)^{-1}\left[\nabla\cdot \f\right]\|_{\L^{2}(\R^3)}\|\bar{\u}_{i,2}\|_{\L^2(\mathcal{O}_{\sqrt{2}k})} \nonumber\\&\leq 
			C\lambda_{i+1}^{-\frac{1}{2}}e^{|y(\vartheta_{t}\omega)|}\|\f\|_{\dot{\mathbb{H}}^{-1}(\R^3)}\|\nabla\bar{\u}_{i,2}\|_{\L^2(\mathcal{O}_{\sqrt{2}k})} \nonumber\\&  \leq  \frac{\mu}{8} \|\nabla\bar{\u}_{i,2}\|^2_{\L^2(\mathcal{O}_{\sqrt{2}k})}+
			C\lambda_{i+1}^{-1}e^{2|y(\vartheta_{t}\omega)|}\|\f\|^{2}_{\dot{\mathbb{H}}^{-1}(\R^3)}.
		\end{align}
	
		Using H\"older's, \eqref{poin-i}, interpolation and Young's inequalities, we obtain
	\begin{align}\label{J4}
		|J_4|&\leq C e^{(r-1)y(\vartheta_{t}\omega)}\int_{\mathcal{O}_{\sqrt{2}k}}|\u|^{\frac{r-1}{4}}|\bar{\u}_{i,2}|^{\frac{1}{2}}|\bar{\u}_{i,2}|^{\frac{1}{2}}|\u|^{\frac{3(r-1)}{4}}|\mathrm{P}_{i}\bar{\u}|\d x \nonumber\\&\leq C e^{(r-1)y(\vartheta_{t}\omega)} \|\left|\u\right|^{\frac{r-1}{2}}\bar{\u}_{i,2}\|^{\frac{1}{2}}_{\L^2(\mathcal{O}_{\sqrt{2}k})}  \|\bar{\u}_{i,2}\|^{\frac{1}{2}}_{\L^2(\mathcal{O}_{\sqrt{2}k})}\|\u\|^{\frac{3(r-1)}{4}}_{\wi\L^{3(r+1)}}\|\u\|_{\wi\L^{\frac{4(r+1)}{r+3}}}  \nonumber\\&\leq C \lambda_{i+1}^{-\frac{1}{4}}e^{(r-1)y(\vartheta_{t}\omega)} \|\left|\u\right|^{\frac{r-1}{2}}\bar{\u}_{i,2}\|^{\frac{1}{2}}_{\L^2(\mathcal{O}_{\sqrt{2}k})}  \|\nabla\bar{\u}_{i,2}\|^{\frac{1}{2}}_{\L^2(\mathcal{O}_{\sqrt{2}k})}\|\u\|^{\frac{3(r-1)}{4}}_{\wi\L^{3(r+1)}}\|\u\|_{\wi\L^{\frac{4(r+1)}{r+3}}}   \nonumber\\&\leq C \lambda_{i+1}^{-\frac{1}{4}}e^{(r-1)y(\vartheta_{t}\omega)} \|\left|\u\right|^{\frac{r-1}{2}}\bar{\u}_{i,2}\|^{\frac{1}{2}}_{\L^2(\mathcal{O}_{\sqrt{2}k})}  \|\u\|^{\frac{1}{2}}_{\V}\|\u\|^{\frac{3(r-1)}{4}}_{\wi\L^{3(r+1)}}\|\u\|^{\frac{1}{2}}_{\H} \|\u\|^{\frac{1}{2}}_{\wi\L^{r+1}}\nonumber\\& \leq \frac{\beta}{2} e^{(r-1)y(\vartheta_{t}\omega)}\|\left|\u\right|^{\frac{r-1}{2}}\bar{\u}_{i,2}\|^2_{\L^2(\mathcal{O}_{\sqrt{2}k})} +C \lambda_{i+1}^{-\frac{1}{3}}e^{(r-1)y(\vartheta_{t}\omega)}\bigg[\|\u\|^{r+1}_{\V}+\|\u\|^{\frac{3(r-1)(r+1)}{3r-2}}_{\wi\L^{3(r+1)}}\nonumber\\&\quad +\|\u\|^{2(r+1)}_{\H} + \|\u\|^{r+1}_{\wi\L^{r+1}}\bigg].
	\end{align}

		Now, combining \eqref{FL2}-\eqref{J4}, we arrive at
		\begin{align}\label{FL7}
			&\frac{\d}{\d t}\|\bar{\u}_{i,2}\|^2_{\L^2(\mathcal{O}_{\sqrt{2}k})} +\left(\alpha-2\sigma y(\vartheta_t\omega)\right)\|\bar{\u}_{i,2}\|^2_{\L^2(\mathcal{O}_{\sqrt{2}k})}\nonumber\\&\leq C\lambda^{-\frac{1}{3}}_{i+1}\bigg[\|\u\|^2_{\V}+e^{(r-1)|y(\vartheta_{t}\omega)|}\|\u\|^{r+1}_{\V}+e^{(r-1)|y(\vartheta_{t}\omega)|}\|\u\|^{2(r+1)}_{\V}+e^{(r-1)y(\vartheta_{t}\omega)}\|\u\|^{r+1}_{\wi\L^{r+1}}\nonumber\\&\quad+e^{(r-1)y(\vartheta_{t}\omega)}\|\u\|^{\frac{(3r-5)(r+1)}{3r-4}}_{\wi\L^{3(r+1)}}+e^{(r-1)y(\vartheta_{t}\omega)}\|\u\|^{\frac{3(r-1)(r+1)}{3r-2}}_{\wi\L^{3(r+1)}}+e^{2|y(\vartheta_{t}\omega)|}\|\f\|^2_{\L^2(\R^3)}\nonumber\\&\quad+e^{2|y(\vartheta_{t}\omega)|}\|\f\|^2_{\dot{\mathbb{H}}^{-1}(\R^3)}\bigg].
		\end{align}
		 In view of the variation of constant formula, we find
		\begin{align}\label{FL8}
			&\|(\I-\P_{i})\bar{\u}(s,s-t,\vartheta_{-s}\omega,\bar{\u}_{0,2})\|^2_{\L^2(\mathcal{O}_{\sqrt{2}k})}\nonumber\\&\leq e^{-\alpha t+2\sigma\int_{-t}^{0}y(\vartheta_{\upeta}\omega)\d\upeta}\|(\I-\P_i)(\varrho_k\u_{0})\|^2_{\L^2(\mathcal{O}_{\sqrt{2}k})}\nonumber\\&\quad+C\lambda^{-\frac{1}{3}}_{i+1}\bigg[\int_{s-t}^{s}e^{\alpha(\zeta-s)-2\sigma\int^{\zeta}_{s}y(\vartheta_{\upeta-s}\omega)\d\upeta}\|\u(\zeta,s-t,\vartheta_{-s}\omega,\u_{0})\|^2_{\V}\d\zeta\nonumber\\&\quad+\int_{s-t}^{s}e^{(r-1)|y(\vartheta_{\zeta-s}\omega)|+\alpha(\zeta-s)-2\sigma\int^{\zeta}_{s}y(\vartheta_{\upeta-s}\omega)\d\upeta}\|\u(\zeta,s-t,\vartheta_{-s}\omega,\u_{0})\|^{r+1}_{\V}\d\zeta\nonumber\\&\quad+\int_{s-t}^{s}e^{(r-1)|y(\vartheta_{\zeta-s}\omega)|+\alpha(\zeta-s)-2\sigma\int^{\zeta}_{s}y(\vartheta_{\upeta-s}\omega)\d\upeta}\|\u(\zeta,s-t,\vartheta_{-s}\omega,\u_{0})\|^{2(r+1)}_{\V}\d\zeta\nonumber\\&\quad+\int_{s-t}^{s}e^{(r-1)|y(\vartheta_{\zeta-s}\omega)|+\alpha(\zeta-s)-2\sigma\int^{\zeta}_{s}y(\vartheta_{\upeta-s}\omega)\d\upeta}\|\u(\zeta,s-t,\vartheta_{-s}\omega,\u_{0})\|^{r+1}_{\wi\L^{r+1}}\d\zeta\nonumber\\&\quad+\int_{s-t}^{s}e^{(r-1)y(\vartheta_{\zeta-s}\omega)+\alpha(\zeta-s)-2\sigma\int^{\zeta}_{s}y(\vartheta_{\upeta-s}\omega)\d\upeta}\|\u(\zeta,s-t,\vartheta_{-s}\omega,\u_{0})\|^{\frac{(3r-5)(r+1)}{3r-4}}_{\wi\L^{3(r+1)}}\d\zeta\nonumber\\&\quad+\int_{s-t}^{s}e^{(r-1)y(\vartheta_{\zeta-s}\omega)+\alpha(\zeta-s)-2\sigma\int^{\zeta}_{s}y(\vartheta_{\upeta-s}\omega)\d\upeta}\|\u(\zeta,s-t,\vartheta_{-s}\omega,\u_{0})\|^{\frac{3(r-1)(r+1)}{3r-2}}_{\wi\L^{3(r+1)}}\d\zeta\nonumber\\&\quad+\int_{-t}^{0}e^{\alpha\zeta+2|y(\vartheta_{\zeta}\omega)|+2\sigma\int_{\zeta}^{0}y(\vartheta_{\upeta}\omega)\d\upeta}\big\{\|\f(\zeta+s)\|^2_{\L^2(\R^3)}+\|\f(\zeta+s)\|^2_{\dot{\mathbb{H}}^{-1}(\R^3)}\big\}\d\zeta\bigg].
		\end{align}

		Since, $\|(\I-\P_i)(\varrho_k\u_{0})\|^2_{\L^2(\mathcal{O}_{\sqrt{2}k})}\leq C\|\u_{0}\|^2_{\H},$ for all $\u_{0}\in D(s-t,\vartheta_{-t}\omega)$ and $s\leq\tau$. Now, similar to \eqref{ep6}, using the definition of backward temperedness \eqref{BackTem}, \eqref{Z3}, \eqref{G3}, Lemma \ref{Absorbing}, Hypothesis \ref{Hypo_f-N} and the fact that $\lambda_{i}\to\infty$ as $i\to\infty$ in \eqref{FL8}, we obtain \eqref{FL-P}, as desired, which completes the proof.
	\end{proof}

	\subsection{Proof of Theorem \ref{MT1-N}}
	This subsection is devoted to the main result of this section, that is, the existence of pullback $\mathfrak{D}$-random attractors and their asymptotic autonomy for the solution of the system \eqref{SCBF} with $S(\v)=\v$. For both the cases given in Table \ref{Table}, the existence of pullback random attractors for non-autonomous 3D stochastic CBF equations driven by multiplicative noise on the whole space is established in \cite{KM6}. For both the cases given in Table \ref{Table}, as the existence of a unique pullback random attractor is known for each $\tau$, one can obtain the existence of a unique random attractor for autonomous 3D stochastic CBF equations driven by multiplicative noise on the whole space (cf. \cite{KM6}).
	
	In view of Propositions \ref{Back_conver} and \ref{IRAS}, and Lemmas \ref{largeradius}-\ref{Flattening}, the proof of Theorem \ref{MT1-N} can be completed by applying similar arguments as in the proof of \cite[Theorem 1.6]{RKM} (\cite[Subsection 3.5]{RKM}) and \cite[Theorem 5.2]{CGTW}.

	\section{3D stochastic CBF equations: Additive noise}\label{sec4}\setcounter{equation}{0}
	In this section, we consider 3D stochastic CBF equations driven by additive white noise, that is, $S(\v)$ is independent of $\v$ and establish the asymptotic autonomy of pullback random attractors. Let us consider the following 3D stochastic CBF equations:
	\begin{equation}\label{SCBF-A}
		\left\{
		\begin{aligned}
			\frac{\d\v(t)}{\d t}+\mu \A\v(t)+\B(\v(t))+\alpha\v(t) +\beta\mathcal{C}(\v(t))&=\mathscr{P}\f(t) +\g\frac{\d \W(t)}{\d t} , \ \ t>\tau, \ \tau\in\R, \\
			\v(x)|_{t=\tau}&=\v_{\tau}(x),	\hspace{28mm} x\in \R^3,
		\end{aligned}
		\right.
	\end{equation}
	where $\g\in\D(\A)$ and $\W(t,\omega)$ is the standard scalar Wiener process on the probability space $(\Omega, \mathscr{F}, \mathbb{P})$ (see Section \ref{sec3} above). Let us define $\u(t,\tau,\omega,\u_{\tau}):=\v(t,\tau,\omega,\v_{\tau})-\g y(\vartheta_{t}\omega)$, where $y$ is given in \eqref{OU1} and satisfies \eqref{OU2}, and $\v$ is the solution of \eqref{1} with $S(\v)=\g$.  Then $\u$ satisfies:
	\begin{equation}\label{2-A}
		\left\{
		\begin{aligned}
			\frac{\d\u}{\d t}-\mu \Delta\u&+\big((\u+\g y(\vartheta_{t}\omega))\cdot\nabla\big)(\u+\g y(\vartheta_{t}\omega))+\alpha\u+\beta \left|\u+\g y(\vartheta_{t}\omega)\right|^{r-1}(\u+\g y(\vartheta_{t}\omega))\\&=-\nabla p+\f +(\sigma-\alpha)\g y(\vartheta_{t}\omega)+\mu y(\vartheta_{t}\omega)\Delta\g,  \ \ \ \ \ \ \  \text{ in }\  \R^3\times(\tau,\infty), \\ \nabla\cdot\u&=0, \hspace{77mm}  \text{ in } \ \ \R^3\times[\tau,\infty), \\
			\u(x)|_{t=\tau}&=\u_{\tau}(x)=\v_{\tau}(x)-\g(x)y(\vartheta_{\tau}\omega),  \ \ \hspace{30mm} x\in \R^3 \ \text{ and }\ \tau\in\R,\\
			|\u(x)|&\to 0,\hspace{76.5mm}  \text{ as }\ |x|\to\infty,
		\end{aligned}
		\right.
	\end{equation}
	as well as (projected form)
	\begin{equation}\label{CCBF-A}
		\left\{
		\begin{aligned}
			\frac{\d\u}{\d t}  +\mu \A\u&+ \B(\u+\g y(\vartheta_{t}\omega))+\alpha\u + \beta \mathcal{C}(\u +\g y(\vartheta_{t}\omega)) \\&= \mathscr{P}\boldsymbol{f} + (\sigma-\alpha)\g y(\vartheta_{t}\omega) +\mu y(\vartheta_{t}\omega)\Delta\g ,\ \  \quad t> \tau, \ \ \tau\in\R ,\\
			\u(x)|_{t=\tau}&=\u_{\tau}(x)=\v_{0}(x)-\g(x)y(\vartheta_{\tau}\omega), \hspace{20mm} x\in\R^3,
		\end{aligned}
		\right.
	\end{equation}
	in $\V'+\widetilde{\L}^{\frac{r+1}{r}}$, where $r\geq1$.

	\subsection{NRDS}
	The following lemma will be frequently used.

	\begin{lemma}\label{EI}
		For both the cases given in Table \ref{Table}, the solution of \eqref{CCBF-A} satisfies the following inequality:
		\begin{align}\label{EI1-A}
			&\frac{\d}{\d t}\|\u(t)\|^2_{\H}+\alpha\|\u(t)\|^2_{\H}+\frac{\alpha}{2}\|\u(t)\|^2_{\H}+\mu\|\nabla\u(t)\|^2_{\H}+\beta\|\u(t)+\g y(\vartheta_{t}\omega)\|^{r+1}_{\wi\L^{r+1}}\nonumber\\&\leq R\left[\|\f(t)\|^{2}_{\L^2(\R^3)}+y_1(\vartheta_{t}\omega)\right],
		\end{align}
		for a.e. $t\geq0$, where $R>0$ is some constant and $y_1(\omega)=\left|y(\omega)\right|^2+\left|y(\omega)\right|^{r+1}+\left|y(\omega)\right|^{\frac{2(r+1)}{r-1}}$, and 
			\begin{align}\label{EI2-A}
			&\frac{\d}{\d t}\|\nabla\u(t)\|^2_{\H}+\alpha\|\nabla\u(t)\|^2_{\H}+\frac{C^*}{2}\|\u(t)+\g y(\vartheta_{t}\omega)\|^{r+1}_{\wi\L^{3(r+1)}}\nonumber\\&\leq \widetilde{R}\bigg[\|\f(t)\|^{2}_{\L^2(\R^3)}+\|\u(t)\|^2_{\V}+\|\u(t)+\g y(\vartheta_{t}\omega)\|^{r+1}_{\wi\L^{r+1}}+y_2(\vartheta_{t}\omega)\bigg],
		\end{align}
	for a.e. $t>0$, where $\widetilde{R}>0$ is some constant and $y_2(\omega)=|y(\omega)|^{2}+ |y(\omega)|^{r+1}$.
	\end{lemma}
	\begin{proof}
		We find from \eqref{CCBF-A} that
		\begin{align}\label{ue17}
			\frac{1}{2}\frac{\d}{\d t}\|\u\|^2_{\H}=&-\mu\|\nabla\u\|^2_{\H}-\alpha\|\u\|^2_{\H}-\beta\|\u+\g y\|^{r+1}_{\wi\L^{r+1}}+b(\u+\g y,\u+\g y,\g y)\nonumber\\&+\beta\left\langle\mathcal{C}(\u+\g y),\g y\right\rangle+\left(\f,\u\right)+y\left(\left(\sigma-\alpha\right)\g-\mu \A\g,\u\right),
		\end{align}
		for a.e. $t\in[\tau,\tau+T]$ with $T>0$. Using $\g\in\D(\A)$, H\"older's and Young's inequalities, there exist constants $R_1,R_2,R_3, R_4>0$ such that
		\begin{align}
				\beta|\left\langle\mathcal{C}(\u+\g y),\g y\right\rangle|&\leq\beta \left|y\right|\|\u+\g y\|^{r}_{\wi\L^{r+1}}\|\g\|_{\wi\L^{r+1}}\leq\frac{\beta}{4}\|\u+\g y\|^{r+1}_{\wi\L^{r+1}} + R_1\left|y\right|^{r+1},\label{ue18}\\
				|\left(\f,\u\right)|&\leq\|\f\|_{\L^2(\R^3)}\|\u\|_{\H}\leq\frac{\alpha}{12}\|\u\|^2_{\H}+R_2\|\f\|^{2}_{\L^2(\R^3)},\label{ue19}\\
				|y\big(\left(\sigma-\alpha\right)\g-\mu \A\g,\u\big)|&\leq\frac{\alpha}{12}\|\u\|^2_{\H}+R_3\left|y\right|^2,\label{ue20}\\
			\left|b(\u+\g y,\u+\g y,\g y)\right|&= \left|b(\u+\g y,\u,\g y)\right|\nonumber\\&\leq \left|y\right|\|\u+\g y\|_{\wi\L^{r+1}}\|\nabla\u\|_{\H}\|\g\|_{\wi\L^{\frac{2(r+1)}{r-1}}}\nonumber\\&\leq\frac{\beta}{4}\|\u+\g y\|^{r+1}_{\wi\L^{r+1}} +\frac{\mu}{2}\|\nabla\u\|^{2}_{\H}+R_4\left|y\right|^{\frac{2(r+1)}{r-1}}.\label{ue21}
		\end{align}
		Combining \eqref{ue17}-\eqref{ue21}, we reach at \eqref{EI1-A} with $R=\max\{2R_1,2R_2,2R_3,2R_4\}$, as required.
		
		Again from \eqref{CCBF-A}, we have 
		\begin{align}\label{ue22}
			&\frac{1}{2}\frac{\d}{\d t}\|\nabla\u\|^2_{\H}  +\mu \|\A\u\|^2_{\H}+\alpha\|\nabla\u\|^2_{\H} + \beta \big(\mathcal{C}(\u +\g y),\A(\u+\g y)\big) \nonumber\\&= -b(\u+\g y, \u+\g y, \A\u)-\beta\big(\mathcal{C}(\u +\g y),\A\g y\big)+(\boldsymbol{f}, \A\u) + (\sigma-\alpha) y(\g , \A\u) +\mu y(\Delta\g, \A\u)
		\end{align}
	
	From \eqref{CuAu} and \eqref{3(r+1)}, we write
		\begin{align}\label{ue23}
		\big(\mathcal{C}(\u+\g y),\A\u+\A\g y\big)&=\int_{\mathbb{R}^3}|\nabla\u(x)+\nabla\g(x) y|^2|\u(x)+\g(x) y|^{r-1}\d x\nonumber\\&\quad+4\left[\frac{r-1}{(r+1)^2}\right]\int_{\mathbb{R}^3}|\nabla|\u(x)+\g(x) y|^{\frac{r+1}{2}}|^2\d x.
	\end{align} 
and 
\begin{align}\label{3(r+1)-A}
	 \|\u+\g y\|^{r+1}_{\wi\L^{3(r+1)}} \leq \frac{8\beta}{C^*}\left[\frac{r-1}{(r+1)^2}\right] \int_{\mathbb{R}^3}|\nabla|\u(x)+\g(x)y|^{\frac{r+1}{2}}|^2\d x, 
\end{align} 
respectively. Using $\g\in\D(\A)$, H\"older's, interpolation and Young's inequalities, there exist constants $R_5,R_6,R_7, R_8>0$ such that
\begin{align}\label{ue24}
	|\left(\f,\A\u\right)|&\leq\|\f\|_{\L^2(\R^3)}\|\A\u\|_{\H}\leq\frac{\mu}{6}\|\A\u\|^2_{\H}+R_5\|\f\|^{2}_{\L^2(\R^3)},\\
	|y\big(\left(\sigma-\alpha\right)\g-\mu \A\g,\A\u\big)|&\leq\frac{\mu}{6}\|\A\u\|^2_{\H}+R_6\left|y\right|^2,
\end{align}
\begin{align}\label{ue25}
&|b(\u+\g y, \u+\g y, \A\u)| 
\nonumber\\&\leq
	\||\u+\g y|\nabla(\u+\g y)\|_{\H}\|\A\u\|_{\H}
\nonumber\\&\leq 
\begin{cases}
	\frac{\mu}{6}\|\A\u\|^2_{\H}+\frac{\beta}{2}\||\u+\g y|^{\frac{r-1}{2}}\nabla(\u+\g y)\|^2_{\H}+C\|\nabla\u+\nabla\g y\|^2_{\H},   & \hspace{3mm}\text{ for } r>3,\\
	\frac{\mu}{6}\|\A\u\|^2_{\H}+\frac{1}{2\mu}\||\u+\g y|\nabla(\u+\g y)\|^2_{\H},  & \hspace{3mm}\text{ for } r=3,
\end{cases}
\nonumber\\&\leq 
\begin{cases}
	\frac{\mu}{6}\|\A\u\|^2_{\H}+\frac{\beta}{2}\||\u+\g y|^{\frac{r-1}{2}}\nabla(\u+\g y)\|^2_{\H}+R_7\big[\|\nabla\u\|^2_{\H}+ |y|^2\big],   & \text{ for } r>3,\\
	\frac{\mu}{6}\|\A\u\|^2_{\H}+\frac{1}{2\mu}\||\u+\g y|\nabla(\u+\g y)\|^2_{\H},  & \text{ for } r=3,
\end{cases}
\end{align}
and
\begin{align}\label{ue26}
	\beta |y(\mathcal{C}(\u+\g y),\A\g )|&\leq \beta|y|\|\u+\g y\|^{r}_{\wi\L^{2r}}\|\A\g\|_{\H}\nonumber\\& \leq C|y|\|\u+\g y\|^{\frac{r+3}{4}}_{\wi\L^{r+1}}\|\u+\g y\|^{\frac{3(r-1)}{4}}_{\wi\L^{3(r+1)}} \nonumber\\&\leq \frac{C^*}{4}\|\u+\g y\|^{r+1}_{\wi\L^{3(r+1)}}+R_8\bigg[\|\u+\g y\|^{r+1}_{\wi\L^{r+1}} + |y|^{r+1}\bigg].
\end{align}

Combining \eqref{ue22}-\eqref{ue26}, we arrive at \eqref{EI2-A} with $\widetilde{R}=\max\{2R_5,2R_6,2R_7,2R_8\}$, as desired.

	\end{proof}

	\begin{lemma}\label{Soln-A}
		Assume that $\f\in\mathrm{L}^2_{\mathrm{loc}}(\R;\L^2(\R^3))$. For each $(\tau,\omega,\u_{\tau})\in\R\times\Omega\times\H$, the system \eqref{CCBF-A} has a unique solution $\u(\cdot,\tau,\omega,\u_{\tau})\in\mathrm{C}([\tau,+\infty);\H)\cap\mathrm{L}^2_{\mathrm{loc}}(\tau,+\infty;\V)\cap\mathrm{L}^{r+1}_{\mathrm{loc}}(\tau,+\infty;\widetilde{\L}^{r+1})$ for both the cases given in Table \ref{Table} such that $\u(\cdot)$ is continuous with respect to the initial data.
	\end{lemma}
	\begin{proof}
		One can prove the existence and uniqueness of solution by a standard Faedo-Galerkin approximation method, see the works \cite{HR,KM,PAM}, etc. For the continuity with respect to the initial data $\u_{\tau}$, see the proof of Theorem 3.9 in \cite{KM}.
	\end{proof}

	Next result shows the Lusin continuity of the mapping of the solution to the system \eqref{CCBF-A} in sample points.

	\begin{proposition}\label{LusinC-A}
		For both the cases given in Table \ref{Table}, suppose that $\f\in\mathrm{L}^2_{\mathrm{loc}}(\R;\L^2(\R^3))$. For each $N\in\N$, the mapping $\omega\mapsto\u(t,\tau,\omega,\u_{\tau})$ (solution of \eqref{CCBF-A}) is continuous from $(\Omega_{N},d_{\Omega_N})$ to $\H$, uniformly in $t\in[\tau,\tau+T]$ with $T>0.$
	\end{proposition}
	\begin{proof}
		Let us assume $\omega_k,\omega_0\in\Omega_N$ be such that $d_{\Omega_N}(\omega_k,\omega_0)\to0$ as $k\to\infty$. Let $\mathscr{U}^k(\cdot):=\u^k(\cdot)-\u^0(\cdot),$ where $\u^k(\cdot)=\u(\cdot,\tau,\omega_k,\u_{\tau})$ and $\u_0(\cdot)=\u(\cdot,\tau,\omega_0,\u_{\tau})$. Then, $\mathscr{U}^k(\cdot)$ satisfies:
		\begin{align}\label{LC1-A}
			\frac{\d\mathscr{U}^k}{\d t}&=-\mu \A\mathscr{U}^k-\alpha\mathscr{U}^k-\left[\B\big(\u^k+y(\vartheta_{t}\omega_k)\g\big)-\B\big(\u^0+y(\vartheta_{t}\omega_0)\g\big)\right] -\big[\beta \mathcal{C}\big(\u^k+y(\vartheta_{t}\omega_k)\g\big)\nonumber\\&\quad-\beta \mathcal{C}\big(\u^0+y(\vartheta_{t}\omega_0)\g\big)\big] +\left\{(\sigma-\alpha)\g+\mu\Delta\g\right\}\left[y(\vartheta_t\omega_k)-y(\vartheta_t\omega_0)\right],
		\end{align}
		in $\V'+\widetilde{\L}^{\frac{r+1}{r}}$.  Taking the inner product with $\mathscr{U}^k(\cdot)$ in \eqref{LC1-A}, using \eqref{441} and rearranging the terms, we obtain
		\begin{align}\label{LC2-A}
			&\frac{1}{2}\frac{\d }{\d t}\|\mathscr{U}^k\|^2_{\H}+\mu\|\nabla\mathscr{U}^k\|^2_{\H}+\alpha\|\mathscr{U}^k\|^2_{\H}\nonumber\\&=-b\big(\mathscr{U}^k+\left[y(\vartheta_t\omega_k)-y(\vartheta_t\omega_0)\right]\g,\mathscr{U}^k+\left[y(\vartheta_t\omega_k)-y(\vartheta_t\omega_0)\right]\g,\u^0+y(\vartheta_t\omega_0)\g\big)\nonumber\\&\quad+\left[y(\vartheta_{t}\omega_k)-y(\vartheta_{t}\omega_0)\right]\left\{b\big(\u^k+y(\vartheta_t\omega_k)\g,\u^k,\g\big)-b\big(\u^0+y(\vartheta_t\omega_0)\g,\u^0,\g\big)\right\}\nonumber\\&\quad-\beta\big\langle\mathcal{C}\big(\u^k+y(\vartheta_t\omega_k)\g\big)-\mathcal{C}\big(\u^0+y(\vartheta_t\omega_0)\g\big),\left(\u^k+y(\vartheta_t\omega_k)\g\right)-\left(\u^0+y(\vartheta_t\omega_0)\g\right)\big\rangle \nonumber\\&\quad+\beta\left[y(\vartheta_{t}\omega_k)-y(\vartheta_{t}\omega_0)\right]\left\langle\mathcal{C}\big(\u^k+y(\vartheta_t\omega_k)\g\big)-\mathcal{C}\big(\u^0+y(\vartheta_t\omega_0)\g\big),\g\right\rangle\nonumber\\&\quad+\left[y(\vartheta_t\omega_k)-y(\vartheta_t\omega_0)\right]\big((\sigma-\alpha)\g+\mu\Delta\g,\mathscr{U}^k\big).
		\end{align}
		From \eqref{MO_c}, we get
		
		\begin{align}\label{LC3-A}
			-&\beta \left\langle\mathcal{C}\big(\u^k+y(\vartheta_t\omega_k)\g\big)-\mathcal{C}\big(\u^0+y(\vartheta_t\omega_0)\g\big),\left(\u^k+y(\vartheta_t\omega_k)\g\right)-\left(\u^0+y(\vartheta_t\omega_0)\g\right)\right\rangle \nonumber\\&\leq-\frac{\beta}{2}\left\|\left|\left(\u^k+y(\vartheta_t\omega_k)\g\right)\right|^{\frac{r-1}{2}}\left[\mathscr{U}^k+\left(y(\vartheta_t\omega_k)-y(\vartheta_t\omega_0)\right)\g\right]\right\|_{\H}^2\nonumber\\&\quad-\frac{\beta}{2}\left\|\left|\left(\u^0+y(\vartheta_t\omega_0)\g\right)\right|^{\frac{r-1}{2}}\left[\mathscr{U}^k+\left(y(\vartheta_t\omega_k)-y(\vartheta_t\omega_0)\right)\g\right]\right\|_{\H}^2.
		\end{align}
		For $\g\in\D(\A)$, in view of H\"older's and Young's inequalities, we obtain
		\begin{align}
			&\left\{b\big(\u^k+y(\vartheta_t\omega_k)\g,\u^k,\g\big)-b\big(\u^0+y(\vartheta_t\omega_0)\g,\u^0,\g\big)\right\}\nonumber\\&\leq \bigg\{\|\u^k+y(\vartheta_t\omega_k)\g\|_{\wi\L^{r+1}}\|\nabla\u^k\|_{\H}+\|\u^0+y(\vartheta_t\omega_0)\g\|_{\wi\L^{r+1}}\|\nabla\u^0\|_{\H}\bigg\}\|\g\|_{\wi\L^{\frac{2(r+1)}{r-1}}}\nonumber\\&\leq C\bigg\{\|\u^k+y(\vartheta_t\omega_k)\g\|^{r+1}_{\wi\L^{r+1}}+\|\nabla\u^k\|^2_{\H}+\|\u^0+y(\vartheta_t\omega_0)\g\|^{r+1}_{\wi\L^{r+1}}+\|\nabla\u^0\|^2_{\H}+1\bigg\}.\label{LC4-A}
		\end{align}
		Moreover, we have
		\begin{align}
			&\beta\left\langle\mathcal{C}\big(\u^k+y(\vartheta_t\omega_k)\g\big)-\mathcal{C}\big(\u^0+y(\vartheta_t\omega_0)\g\big),\g\right\rangle\nonumber\\&\leq
			\bigg\{\|\u^k+y(\vartheta_t\omega_k)\g\|^{r}_{\wi\L^{r+1}}+\|\u^0+y(\vartheta_t\omega_0)\g\|^{r}_{\wi\L^{r+1}}\bigg\}\|\g\|_{\wi\L^{r+1}}\nonumber\\&\leq C\bigg\{\|\u^k+y(\vartheta_t\omega_k)\g\|^{r+1}_{\wi\L^{r+1}}+\|\u^0+y(\vartheta_t\omega_0)\g\|^{r+1}_{\wi\L^{r+1}}+1\bigg\},\label{LC5-A}
		\end{align}
	and
	\begin{align}
			&\big((\sigma-\alpha)\g+\mu\Delta\g,\mathscr{U}^k\big) \leq C\|\u^k-\u^0\|_{\H}\leq C\bigg\{\|\u^k\|^2_{\H}+\|\u^0\|^2_{\H}+1\bigg\}\label{LC6-A}.
		\end{align}
		Next, we estimate the remaining terms of \eqref{LC2-A} separately.
		\vskip 2mm
		\noindent
		\textbf{Case I:} \textit{When $r>3$.} Using H\"older's and Young's inequalities, we infer
		\begin{align}\label{LC8-A}
			&\left|b\big(\mathscr{U}^k+\left[y(\vartheta_t\omega_k)-y(\vartheta_t\omega_0)\right]\g,\mathscr{U}^k+\left[y(\vartheta_t\omega_k)-y(\vartheta_t\omega_0)\right]\g,\u^0+y(\vartheta_t\omega_0)\g\big)\right|
			\nonumber\\&\leq\frac{\mu}{2}\|\nabla\mathscr{U}^k\|^2_{\H}+C\|\mathscr{U}^k\|^2_{\H}+C\left|y(\vartheta_t\omega_k)-y(\vartheta_t\omega_0)\right|^2\nonumber\\&\quad+\frac{\beta}{4}\left\|\left|\u^0+y(\vartheta_{t}\omega_0)\g\right|^{\frac{r-1}{2}}\left[\mathscr{U}^k+\left(y(\vartheta_t\omega_k)-y(\vartheta_t\omega_0)\right)\g\right]\right\|^2_{\H}.
		\end{align}
		\vskip 2mm
		\noindent
		\textbf{Case II:} \textit{When $r=3$ with $2\beta\mu\geq1$.} Applying \eqref{b0}, H\"older's and Young's inequalities, we obtain
		
		\begin{align}
			&\left|b\big(\mathscr{U}^k+\left[y(\vartheta_t\omega_k)-y(\vartheta_t\omega_0)\right]\g,\mathscr{U}^k+\left[y(\vartheta_t\omega_k)-y(\vartheta_t\omega_0)\right]\g,\u^0+y(\vartheta_t\omega_0)\g\big)\right|\nonumber\\&\leq\left|b\big(\mathscr{U}^k,\mathscr{U}^k+\left[y(\vartheta_t\omega_k)-y(\vartheta_t\omega_0)\right]\g,\u^0+y(\vartheta_t\omega_0)\g\big)\right|\nonumber\\&\quad+\left|y(\vartheta_t\omega_k)-y(\vartheta_t\omega_0)\right|\left|b\big(\g,\mathscr{U}^k+\left[y(\vartheta_t\omega_k)-y(\vartheta_t\omega_0)\right]\g,\u^k+y(\vartheta_t\omega_k)\g\big)\right|\nonumber\\&\leq\frac{1}{2\beta}\|\nabla\mathscr{U}^k\|_{\H}+\frac{\beta}{2}\left\|\left|\u^0+y(\vartheta_{t}\omega_0)\g\right|\left[\mathscr{U}^k+\left(y(\vartheta_t\omega_k)-y(\vartheta_t\omega_0)\right)\g\right]\right\|^2_{\H}\nonumber\\&\quad+C\left|y(\vartheta_t\omega_k)-y(\vartheta_t\omega_0)\right|^2+\frac{\beta}{2}\left\|\left|\u^k+y(\vartheta_{t}\omega_k)\g\right|\left[\mathscr{U}^k+\left(y(\vartheta_t\omega_k)-y(\vartheta_t\omega_0)\right)\g\right]\right\|^2_{\H}.\label{LC9-A}
		\end{align}
		Combining \eqref{LC2-A}-\eqref{LC9-A}, we arrive at
		\begin{align}\label{LC10-A}
			\frac{\d }{\d t}\|\mathscr{U}^k(t)\|^2_{\H}\leq C
			\left[\|\mathscr{U}^k(t)\|^2_{\H}+\widetilde{Q}(t)\right],
		\end{align}
		for a.e. $t\in[\tau,\tau+T]$ with $T>0$, and where
		\begin{align*}
				\widetilde{Q}&=\left|y(\vartheta_{t}\omega_k)-y(\vartheta_{t}\omega_0)\right|\bigg\{\|\u^k+y(\vartheta_t\omega_k)\g\|^{r+1}_{\wi\L^{r+1}}+\|\u^k\|^2_{\V}+\|\u^0+y(\vartheta_t\omega_0)\g\|^{r+1}_{\wi\L^{r+1}}+\|\u^0\|^2_{\V}+1\bigg\}\nonumber\\&\quad+\left|y(\vartheta_t\omega_k)-y(\vartheta_t\omega_0)\right|^2.
		\end{align*}
		We infer from \eqref{EI1-A} that
		\begin{align*}
			&\int_{\tau}^{\tau+T}\bigg\{\|\u^k(t)+y(\vartheta_t\omega_k)\g\|^{r+1}_{\wi\L^{r+1}}+\|\u^k(t)\|^2_{\H}+\|\nabla\u^k(t)\|^2_{\H}\bigg\}\d t\nonumber\\&\leq\|\u_{\tau}\|^2_{\L^2(\R^3)}+C\int_{\tau}^{\tau+T}\left[\|\f(t)\|^{2}_{\H}+\left|y(\vartheta_{t}\omega_k)\right|^2+\left|y(\vartheta_{t}\omega_k)\right|^{r+1}+\left|y(\vartheta_{t}\omega_k)\right|^{\frac{2(r+1)}{r-1}}\right]\d t,
		\end{align*}
		which gives
		\begin{align}\label{LC11-A}
			\sup_{k\in\N}\int_{\tau}^{\tau+T}\bigg\{\|\u^k(t)+y(\vartheta_t\omega_k)\g\|^{r+1}_{\wi\L^{r+1}}+\|\u^k(t)\|^2_{\H}+\|\nabla\u^k(t)\|^2_{\H}\bigg\}\d t\leq C(\tau,T,\omega_0),
		\end{align}
		where we have used \eqref{conv_z2} and the fact that $\f\in\mathrm{L}^2_{\text{loc}}(\R;\L^2(\R^3))$.
		Now, from $\f\in\mathrm{L}^2_{\text{loc}}(\R;\L^2(\R^3))$, $\u^0\in\mathrm{C}([\tau,+\infty);\H)\cap\mathrm{L}^2_{\mathrm{loc}}(\tau,+\infty;\V)\cap\mathrm{L}^{r+1}_{\mathrm{loc}}(\tau,+\infty;\widetilde{\L}^{r+1})$, Lemma \ref{conv_z} and \eqref{LC11-A}, we conclude that
		\begin{align}\label{LC13-A}
			\lim_{k\to+\infty}\int_{\tau}^{\tau+T}\widetilde{Q}(t)\d t=0.
		\end{align}
		In view of the Gronwall inequality in \eqref{LC10-A} and making use of \eqref{LC13-A}, one can complete the proof.
	\end{proof}
	
	Lemma \ref{Soln-A} ensures us that we can define a mapping $\widetilde{\Phi}:\R^+\times\R\times\Omega\times\H\to\H$ by
	\begin{align}\label{Phi-A}
		\widetilde{\Phi}(t,\tau,\omega,\v_{\tau}):=\v(t+\tau,\tau,\vartheta_{-\tau}\omega,\v_{\tau})=\u(t+\tau,\tau,\vartheta_{-\tau}\omega,\u_{\tau})+\g y(\vartheta_{t}\omega).
	\end{align}
	The Lusin continuity in Proposition \ref{LusinC-A} provides the $\mathscr{F}$-measurability of $\widetilde{\Phi}$. Consequently, $\widetilde{\Phi}$ defined by \eqref{Phi-A} is a NRDS on $\H$.

	\subsection{Backward convergence of NRDS}  Consider the following autonomous 3D stochastic CBF equations driven by the additive white noise:
	\begin{equation}\label{A-SCBF-A}
		\left\{
		\begin{aligned}
			\frac{\d\widetilde{\v}(t)}{\d t}+\mu \A\widetilde{\v}(t)+\B(\widetilde{\v}(t))+\alpha\widetilde{\v}(t) +\beta\mathcal{C}(\widetilde{\v}(t))&=\mathscr{P}\f_{\infty} +\g\frac{\d \W(t)}{\d t},\ \ t>0, \\
			\widetilde{\v}(x,0)&=\widetilde{\v}_{0}(x),\hspace{27mm}	x\in \R^3.
		\end{aligned}
		\right.
	\end{equation}
	Let $\widetilde{\u}(t,\omega)=\widetilde{\v}(t,\omega)-\g y(\vartheta_{t}\omega)$. Then, the system \eqref{A-SCBF-A} can be written in the following pathwise deterministic system:
	\begin{equation}\label{A-CCBF-A}
		\left\{
		\begin{aligned}
			\frac{\d\widetilde{\u}(t)}{\d t} & +\mu \A\widetilde{\u}(t)+ \B(\widetilde{\u}(t)+\g y(\vartheta_{t}\omega))+\alpha\widetilde{\u}(t) + \beta \mathcal{C}(\widetilde{\u}(t) +\g y(\vartheta_{t}\omega)) \\&\quad= \mathscr{P} \boldsymbol{f}_{\infty} + (\sigma-\alpha)\g y(\vartheta_{t}\omega) -\mu y(\vartheta_{t}\omega)\A\g , \quad t> 0 ,\\
			\widetilde{\u}(x,0)&=\widetilde{\u}_{0}(x)=\widetilde{\v}_{0}(x)-\g(x)y(\omega), \hspace{25mm} \ \ x\in\R^3,
		\end{aligned}
		\right.
	\end{equation}
	in $\V'+\widetilde{\L}^{\frac{r+1}{r}}$.
	\begin{proposition}\label{Back_conver-A}
		For both the cases given in Table \ref{Table}, let Assumption \ref{Hypo_f-N} hold and $\lim\limits_{\tau\to-\infty}\|\u_{\tau}-\widetilde{\u}_0\|_{\H}=0$. Then, the solution $\u$ of the system \eqref{CCBF-A} backward converges to the solution $\widetilde{\u}$ of the system \eqref{A-CCBF-A}, that is,
		\begin{align}
			\lim_{\tau\to -\infty}\|\u(T+\tau,\tau,\vartheta_{-\tau}\omega,\u_{\tau})-\widetilde{\u}(t,\omega,\widetilde{\u}_0)\|_{\H}=0, \ \ \text{ for all } \ T>0 \text{ and } \omega\in\Omega.
		\end{align}
	\end{proposition}
	\begin{proof}
		Let $\mathscr{U}^{\tau}(\cdot):=\u(\cdot+\tau,\tau,\vartheta_{-\tau}\omega,\u_{\tau})-\widetilde{\u}(\cdot,\omega,\widetilde{\u}_0)$. From \eqref{CCBF-A} and \eqref{A-CCBF-A}, we get
		\begin{align}\label{BC1-A}
			\frac{\d\mathscr{U}^{\tau}}{\d t}&=-\mu \A\mathscr{U}^{\tau}-\alpha\mathscr{U}^{\tau}-\left[\B\big(\u+\g y(\vartheta_{t}\omega)\big)-\B\big(\widetilde{\u}+\g y(\vartheta_{t}\omega)\big)\right]\nonumber\\&\quad -\beta\left[\mathcal{C}\big(\u+\g y(\vartheta_{t}\omega)\big)-\mathcal{C}\big(\widetilde{\u}+\g y(\vartheta_{t}\omega)\big)\right]+\left[\mathscr{P}\f(t+\tau)-\mathscr{P}\f_{\infty}\right],
		\end{align}
		in $\V'+\widetilde{\L}^{\frac{r+1}{r}}$.  In view of \eqref{BC1-A}, we obtain
		\begin{align}\label{BC2-A}
			\frac{\d }{\d t}\|\mathscr{U}^{\tau}\|^2_{\H}&=-\mu\|\nabla\mathscr{U}^{\tau}\|^2_{\H}-\alpha\|\mathscr{U}^{\tau}\|^2_{\H}-\left\langle\B\big(\u+\g y(\vartheta_{t}\omega)\big)-\B\big(\widetilde{\u}+\g y(\vartheta_{t}\omega)\big),\u-\widetilde{\u}\right\rangle\nonumber\\&\quad-\beta \left\langle\mathcal{C}\big(\u+\g y(\vartheta_{t}\omega)\big)-\mathcal{C}\big(\widetilde{\u}+\g y(\vartheta_{t}\omega)\big),\u-\widetilde{\u}\right\rangle+(\f(t+\tau)-\f_{\infty},\mathscr{U}^{\tau}).
		\end{align}
		From \eqref{MO_c}, one can  rewrite
		\begin{align}\label{BC3-A}
			-&\beta \left\langle\mathcal{C}\big(\u+\g y(\vartheta_{t}\omega)\big)-\mathcal{C}\big(\widetilde{\u}+\g y(\vartheta_{t}\omega)\big),(\u+\g y(\vartheta_{t}\omega))-(\widetilde{\u}+\g y(\vartheta_{t}\omega))\right\rangle \nonumber\\&\leq-\frac{\beta}{2}\||\u+\g y(\vartheta_{t}\omega)|^{\frac{r-1}{2}}\mathscr{U}^{\tau}\|_{\H}^2-\frac{\beta}{2}\||\widetilde{\u}+\g y(\vartheta_{t}\omega)|^{\frac{r-1}{2}}\mathscr{U}^{\tau}\|_{\H}^2
		\end{align}
		Applying \eqref{b0}, \eqref{441}, H\"older's and Young's inequalities, we infer
		\begin{align}\label{BC4-A}
			&\left|\left\langle\B\big(\u+\g y(\vartheta_{t}\omega)\big)-\B\big(\widetilde{\u}+\g y(\vartheta_{t}\omega)\big),(\u+\g y(\vartheta_{t}\omega))-(\widetilde{\u}+\g y(\vartheta_{t}\omega))\right\rangle\right|\nonumber\\&=\left|b(\mathscr{U}^{\tau},\mathscr{U}^{\tau},\widetilde{\u}+\g y(\vartheta_{t}\omega))\right|\nonumber\\&\leq \begin{cases}
				\frac{\mu}{2}\|\nabla\mathscr{U}^{\tau}\|_{\H}^2+\frac{\beta}{4}\||\widetilde{\u}+\g y(\vartheta_{t}\omega)|^{\frac{r-1}{2}}|\mathscr{U}^{\tau}|\|^2_{\H}+C\|\mathscr{U}^{\tau}\|^2_{\H} ,&\text{ for } r>3,\\
				\frac{1}{2\beta}\|\nabla\mathscr{U}^{\tau}\|_{\H}^2+\frac{\beta}{2}\||\widetilde{\u}+\g y(\vartheta_{t}\omega)||\mathscr{U}^{\tau}|\|^2_{\H},&\text{ for } r=3 \text{ and } 2\beta\mu\geq1,
			\end{cases}
		\end{align}
		and
		\begin{align}\label{BC5-A}
			\left|(\f(t+\tau)-\f_{\infty},\mathscr{U}^{\tau})\right|\leq C\|\f(t+\tau)-\f_{\infty}\|^2_{\L^2(\R^3)}+\frac{\alpha}{2}\|\mathscr{U}^{\tau}\|^2_{\H}.
		\end{align}
		Combining \eqref{BC2-A}-\eqref{BC5-A}, we achieve
		\begin{align}\label{BC6-A}
			&\frac{\d }{\d t}\|\mathscr{U}^{\tau}\|^2_{\H}\leq C
				\|\mathscr{U}^{\tau}\|^2_{\H}+\|\f(t+\tau)-\f_{\infty}\|^2_{\L^2(\R^3)}.
		\end{align}
	 Applying similar steps as in Proposition \ref{Back_conver}, we complete the proof.
	\end{proof}
	\subsection{Increasing random absorbing sets}
	This subsection provides the existence of increasing $\mathfrak{D}$-random absorbing set for non-autonomous 3D stochastic CBF equations \eqref{SCBF-A}.
	\begin{lemma}\label{Absorbing-A}
		For both the cases given in Table \ref{Table} and for each $(\tau,\omega,D)\in\R\times\Omega\times\mathfrak{D},$ there exists a time $\widetilde{\mathcal{T}}:=\widetilde{\mathcal{T}}(\tau,\omega,D)>0$ such that
		\begin{align}\label{AB1-A}
			&\sup_{s\leq \tau}\sup_{t\geq \widetilde{\mathcal{T}}}\sup_{\u_{0}\in D(s-t,\vartheta_{-t}\omega)}\bigg[\|\u(s,s-t,\vartheta_{-s}\omega,\u_{0})\|^2_{\H}+\frac{\alpha}{2}\int_{s-t}^{s}e^{\alpha(\zeta-s)}\|\u(\zeta,s-t,\vartheta_{-s}\omega,\u_{0})\|^2_{\H}\d\zeta\nonumber\\&\quad+\mu\int_{s-t}^{s}e^{\alpha(\zeta-s)}\|\nabla\u(\zeta,s-t,\vartheta_{-s}\omega,\u_{0})\|^2_{\H}\d\zeta\nonumber\\&\quad+\beta\int_{s-t}^{s}e^{\alpha(\zeta-s)}\|\u(\zeta,s-t,\vartheta_{-s}\omega,\u_{0})+\g y(\vartheta_{\zeta-s}\omega)\|^{r+1}_{\wi\L^{r+1}}\d\zeta\bigg]\leq 2R\sup_{s\leq \tau}\widetilde{K}(s,\omega),
		\end{align}
		where $R$ is the  same as in \eqref{EI1-A} and $\widetilde{K}(s,\omega)$ is given by
		\begin{align}\label{AB2-A}
			\widetilde{K}(s,\omega):=\int_{-\infty}^{0}e^{\alpha\zeta}\left[\|\f(\zeta+s)\|^2_{\L^2(\R^3)}+y_1(\vartheta_{\zeta}\omega)\right]\d\zeta.
		\end{align}
	 Furthermore, for $2<k_1<\infty$, we infer that
		\begin{align}\label{AB11-A}
			&\lim_{t\to +\infty}\sup_{s\leq \tau}\sup_{\u_{0}\in D(s-t,\vartheta_{-t}\omega)}\int_{s-t}^{s}e^{\alpha(\zeta-s)}\|\u(\zeta,s-t,\vartheta_{-s}\omega,\u_{0})\|^{k_1}_{\H}\d\zeta\nonumber\\&\leq C\bigg(\int_{-\infty}^{0}e^{\frac{2(k_1-1)\alpha}{k_1^2}\zeta}\bigg[\|\f(\zeta+s)\|^2_{\L^2(\R^3)}+y_1(\vartheta_{\zeta}\omega)\bigg]\d\zeta\bigg)^{\frac{k_1}{2}},
		\end{align}
	and 
		\begin{align}\label{AB12-A}
		&\lim_{t\to +\infty}\sup_{s\leq \tau}\sup_{\u_{0}\in D(s-t,\vartheta_{-t}\omega)}\int_{s-t}^{s}e^{\alpha(\xi-s)}\|\nabla\u(\xi,s-t,\vartheta_{-s}\omega,\u_{0})\|^{k_1}_{\H}\d\xi\nonumber\\&\leq C\bigg[\int_{-\infty}^{0}e^{\frac{\alpha}{2k_1}\zeta}\bigg(\|\f(\zeta+s)\|^2_{\L^2(\R^3)}+y_1(\vartheta_{\zeta}\omega)+y_2(\vartheta_{\zeta}\omega)\bigg)\d\zeta\bigg]^{\frac{k_1}{2}}   \nonumber\\&\quad + C\bigg[ \int_{-\infty}^{0}e^{\frac{\alpha}{8k_1}\zeta_1}\left[\|\f(\zeta_1+s)\|^2_{\L^2(\R^3)}+y_1(\vartheta_{\zeta_1}\omega)\right]\d\zeta_1\bigg]^{k_1}.
	\end{align}
Finally, we obtain for $s\in\R$ and $t>0$
\begin{align}\label{AB18V-A}
	&\int_{s-t}^{s} (\zeta-s+t) e^{\alpha(\zeta-s)}\|\u(\zeta,s-t,\vartheta_{-s}\omega,\u_{0})+\g y(\vartheta_{\zeta-s}\omega)\|^{r+1}_{\wi\L^{3(r+1)}}\d\zeta
	\nonumber\\&\leq C(t+1)\bigg[ e^{-\frac{\alpha}{4}t}\|\u_{0}\|^4_{\H}+ \int_{-\infty}^{0}e^{\frac{\alpha}{2}\zeta}\bigg(\|\f(\zeta+s)\|^2_{\L^2(\R^3)}+y_1(\vartheta_{\zeta}\omega)+y_2(\vartheta_{\zeta}\omega)\bigg)\d\zeta\ \nonumber\\&\quad  + \bigg(\int_{-\infty}^{0}e^{\frac{\alpha}{8}\zeta}\left[\|\f(\zeta+s)\|^2_{\L^2(\R^3)}+y_1(\vartheta_{\zeta}\omega)\right]\d\zeta\bigg)^2\bigg].
\end{align}
	\end{lemma}

	\begin{proof}
		Let us consider the energy inequality \eqref{EI1-A} for $\u(\zeta):=\u(\zeta,s-t,\vartheta_{-s}\omega,\u_{0})$, that is,
		\begin{align*}
			&\frac{\d}{\d \zeta}\|\u(\zeta)\|^2_{\H}+\alpha\|\u(\zeta)\|^2_{\H}+\frac{\alpha}{2}\|\u(\zeta)\|^2_{\H}+\mu\|\nabla\u(\zeta)\|^2_{\H}+\beta\|\u(\zeta)+\g y(\vartheta_{\zeta-s}\omega)\|^{r+1}_{\wi\L^{r+1}}\nonumber\\&\leq R\left[\|\f(\zeta)\|^{2}_{\L^2(\R^3)}+y_1(\vartheta_{\zeta-s}\omega)\right],
		\end{align*}
		In view of the variation of constants formula with respect to $\zeta\in(s-t,\xi)$, 	\begin{align}\label{AB3-A}
			&\|\u(\xi,s-t,\vartheta_{-s}\omega,\u_{0})\|^2_{\H}+\frac{\alpha}{2}\int_{s-t}^{\xi}e^{\alpha(\zeta-\xi)}\|\u(\zeta,s-t,\vartheta_{-s}\omega,\u_{0})\|^2_{\H}\d\zeta\nonumber\\&+\mu\int_{s-t}^{\xi}e^{\alpha(\zeta-\xi)}\|\nabla\u(\zeta,s-t,\vartheta_{-s}\omega,\u_{0})\|^2_{\H}\d\zeta\nonumber\\&+\beta\int_{s-t}^{\xi}e^{\alpha(\zeta-\xi)}\|\u(\zeta,s-t,\vartheta_{-s}\omega,\u_{0})+\g y(\vartheta_{\zeta-s}\omega)\|^{r+1}_{\wi\L^{r+1}}\d\zeta\nonumber\\&\leq e^{-\alpha(\xi-s+t)}\|\u_{0}\|^2_{\H} + R\int_{-t}^{\xi-s}e^{\alpha(\zeta+s-\xi)}\left[\|\f(\zeta+s)\|^2_{\L^2(\R^3)}+y_1(\vartheta_{\zeta}\omega)\right]\d\zeta.
		\end{align}
		Putting $\xi=s$ in \eqref{AB3-A}, we find
		\begin{align}\label{AB4-A}
			&\|\u(s,s-t,\vartheta_{-s}\omega,\u_{0})\|^2_{\H}+\frac{\alpha}{2}\int_{s-t}^{s}e^{\alpha(\zeta-s)}\|\u(\zeta,s-t,\vartheta_{-s}\omega,\u_{0})\|^2_{\H}\d\zeta\nonumber\\&+\mu\int_{s-t}^{s}e^{\alpha(\zeta-s)}\|\nabla\u(\zeta,s-t,\vartheta_{-s}\omega,\u_{0})\|^2_{\H}\d\zeta\nonumber\\&+\beta\int_{s-t}^{s}e^{\alpha(\zeta-s)}\|\u(\zeta,s-t,\vartheta_{-s}\omega,\u_{0})+\g y(\vartheta_{\zeta-s}\omega)\|^{r+1}_{\wi\L^{r+1}}\d\zeta\nonumber\\&\quad\leq e^{-\alpha t}\|\u_{0}\|^2_{\H} +R\int_{-\infty}^{0}e^{\alpha\zeta}\left[\|\f(\zeta+s)\|^2_{\L^2(\R^3)}+y_1(\vartheta_{\zeta}\omega)\right]\d\zeta,
		\end{align}
		for all $s\leq\tau$. Since $\u_0\in D(s-t,\vartheta_{-t}\omega)$ and $D$ is backward tempered, the definition of backward temperedness \eqref{BackTem} ensures that there exists a time $\widetilde{\mathcal{T}}=\widetilde{\mathcal{T}}(\tau,\omega,D)$ such that for all $t\geq \widetilde{\mathcal{T}}$,
		\begin{align}\label{v_0-A}
			e^{-\alpha t}\sup_{s\leq \tau}\|\u_{0}\|^2_{\H}\leq R\int_{-\infty}^{0}e^{\alpha\zeta}\left[\|\f(\zeta+s)\|^2_{\L^2(\R^3)}+y_1(\vartheta_{\zeta}\omega)\right]\d\zeta.
		\end{align}
		Hence, Using \eqref{v_0-A} and taking supremum on $s$ over $(-\infty,\tau]$ in \eqref{AB4-A}, we arrive at \eqref{AB1-A}. Furthermore, the inequality \eqref{AB11-A} can be obtained by using \eqref{AB3-A} and following the similar arguments as in \eqref{AB-5}.

			From \eqref{EI2-A} for $\zeta>s-t$ and $\u(\zeta):=\u(\zeta,s-t,\vartheta_{-s}\omega,\u_{0})$, we have
	\begin{align}\label{AB12V-A}
		&\frac{\d}{\d \zeta}\|\nabla\u(\zeta)\|^2_{\H}+\alpha\|\nabla\u(\zeta)\|^2_{\H}\nonumber\\&\leq \widetilde{R}\bigg[\|\f(\zeta)\|^{2}_{\L^2(\R^3)}+\|\u(\zeta)+\g y(\vartheta_{\zeta-s}\omega)\|^{r+1}_{\wi\L^{r+1}}+\|\u(\zeta)\|^2_{\V}+\|\u(\zeta)\|^4_{\H}+y_2(\vartheta_{\zeta-s}\omega)\bigg],
	\end{align}

		In view of the variation of constants formula with respect to $\zeta\in(\xi_1,\xi)$ with $s-t<\xi_1\leq\xi\leq s$, we find
		 \begin{align}\label{AB13V-A}
			 &\|\nabla\u(\xi,s-t,\vartheta_{-s}\omega,\u_{0})\|^2_{\H}\nonumber\\&\leq e^{\alpha(\xi_1-\xi)} \|\nabla\u(\xi_1,s-t,\vartheta_{-s}\omega,\u_{0})\|^2_{\H}+C\int_{s-t}^{\xi}e^{\alpha(\zeta-\xi)}\bigg[\|\u(\zeta,s-t,\vartheta_{-s}\omega,\u_{0})\|^2_{\V}\nonumber\\&\quad+\|\u(\zeta,s-t,\vartheta_{-s}\omega,\u_{0})+\g y(\vartheta_{\zeta-s}\omega)\|^{r+1}_{\wi\L^{r+1}}+ \|\u(\zeta,s-t,\vartheta_{-s}\omega,\u_{0})\|^4_{\H}\bigg]\d\zeta \nonumber\\&\quad+C\int_{s-t}^{\xi}e^{\alpha(\zeta-\xi)}\bigg[\|\f(\zeta)\|^2_{\L^2(\R^3)}+y_2(\vartheta_{\zeta-s}\omega)\bigg]\d\zeta.
		\end{align}
		Integrating \eqref{AB13V-A} from $s-t$ to $\xi$ with respect to $\xi_1$ and using \eqref{AB3-A}, we obtain for $2<k_1<\infty$ 
		{\small \begin{align}\label{AB14V-A}
			&\|\nabla\u(\xi,s-t,\vartheta_{-s}\omega,\u_{0})\|^2_{\H}\nonumber\\&\leq C\bigg\{\frac{1}{\xi-s+t}+1\bigg\}\bigg[e^{-\alpha(\xi-s+t)}\|\u_{0}\|^2_{\H} + \int_{s-t}^{\xi}e^{\alpha(\zeta-\xi)}\bigg(\|\f(\zeta)\|^2_{\L^2(\R^3)}+y_1(\vartheta_{\zeta-s}\omega)+y_2(\vartheta_{\zeta-s}\omega)\bigg)\d\zeta\bigg]     \nonumber\\&\quad+C\int_{s-t}^{\xi}e^{\alpha(\zeta-\xi)} \|\u(\zeta,s-t,\vartheta_{-s}\omega,\u_{0})\|^4_{\H}\d\zeta
			\nonumber\\&\leq C\bigg\{\frac{1}{\xi-s+t}+1\bigg\}\bigg[e^{-\alpha(\xi-s+t)}\|\u_{0}\|^2_{\H} + \int_{s-t}^{\xi}e^{\alpha(\zeta-\xi)}\bigg(\|\f(\zeta)\|^2_{\L^2(\R^3)}+y_1(\vartheta_{\zeta-s}\omega)+y_2(\vartheta_{\zeta-s}\omega)\bigg)\d\zeta\bigg]    \nonumber\\&\quad+C\int_{s-t}^{\xi}e^{\alpha(\zeta-\xi)}\bigg[ e^{-2\alpha(\zeta-s+t)}\|\u_{0}\|^4_{\H} + \bigg(\int_{-t}^{\zeta-s}e^{\alpha(\zeta_1+s-\zeta)}\left[\|\f(\zeta_1+s)\|^2_{\L^2(\R^3)}+y_1(\vartheta_{\zeta_1}\omega)\right]\d\zeta_1\bigg)^2\bigg]\d\zeta\nonumber\\&\leq C\bigg\{\frac{1}{\xi-s+t}+1\bigg\}\bigg[e^{-\frac{\alpha}{2k_1}(\xi-s+t)}\|\u_{0}\|^2_{\H} + \int_{-t}^{\xi-s}e^{\frac{\alpha}{2k_1}(\zeta+s-\xi)}\bigg(\|\f(\zeta+s)\|^2_{\L^2(\R^3)}+y_1(\vartheta_{\zeta}\omega)+y_2(\vartheta_{\zeta}\omega)\bigg)\d\zeta\bigg]    \nonumber\\&\quad+C\int_{s-t}^{\xi}e^{\frac{\alpha}{2k_1}(\zeta-\xi)}\bigg[ e^{-\frac{\alpha}{4k_1}(\zeta-s+t)}\|\u_{0}\|^4_{\H} + \bigg(\int_{-t}^{\zeta-s}e^{\frac{\alpha}{8k_1}(\zeta_1+s-\zeta)}\left[\|\f(\zeta_1+s)\|^2_{\L^2(\R^3)}+y_1(\vartheta_{\zeta_1}\omega)\right]\d\zeta_1\bigg)^2\bigg]\d\zeta
			\nonumber\\&\leq Ce^{-\frac{\alpha}{2k_1}(\xi-s)}\bigg\{\frac{1}{\xi-s+t}+1\bigg\}\bigg[e^{-\frac{\alpha}{2k_1}t}\|\u_{0}\|^2_{\H} + \int_{-\infty}^{0}e^{\frac{\alpha}{2k_1}\zeta}\bigg(\|\f(\zeta+s)\|^2_{\L^2(\R^3)}+y_1(\vartheta_{\zeta}\omega)+y_2(\vartheta_{\zeta}\omega)\bigg)\d\zeta\bigg]    \nonumber\\&\quad+C\int_{s-t}^{\xi}e^{\frac{\alpha}{2k_1}(\zeta-\xi)}e^{-\frac{\alpha}{4k_1}(\zeta-s)}\d\zeta\bigg[ e^{-\frac{\alpha}{4k_1}t}\|\u_{0}\|^4_{\H} + \bigg(\int_{-\infty}^{0}e^{\frac{\alpha}{8k_1}\zeta_1}\left[\|\f(\zeta_1+s)\|^2_{\L^2(\R^3)}+y_1(\vartheta_{\zeta_1}\omega)\right]\d\zeta_1\bigg)^2\bigg]
			\nonumber\\&\leq Ce^{-\frac{\alpha}{2k_1}(\xi-s)}\bigg\{\frac{1}{\xi-s+t}+1\bigg\}\bigg[e^{-\frac{\alpha}{2k_1}t}\|\u_{0}\|^2_{\H} + \int_{-\infty}^{0}e^{\frac{\alpha}{2k_1}\zeta}\bigg(\|\f(\zeta+s)\|^2_{\L^2(\R^3)}+y_1(\vartheta_{\zeta}\omega)+y_2(\vartheta_{\zeta}\omega)\bigg)\d\zeta\bigg]    \nonumber\\&\quad+Ce^{-\frac{\alpha}{4k_1}(\xi-s)}\bigg[ e^{-\frac{\alpha}{4k_1}t}\|\u_{0}\|^4_{\H} + \bigg(\int_{-\infty}^{0}e^{\frac{\alpha}{8k_1}\zeta_1}\left[\|\f(\zeta_1+s)\|^2_{\L^2(\R^3)}+y_1(\vartheta_{\zeta_1}\omega)\right]\d\zeta_1\bigg)^2\bigg].
		\end{align}}

		Now from \eqref{AB14V-A}, we estimate for $2<k_1<\infty$ 
		\begin{align}\label{AB15V-A}
			&\int_{s-t}^{s}e^{\alpha(\xi-s)}\|\nabla\u(\xi,s-t,\vartheta_{-s}\omega,\u_{0})\|^{k_1}_{\H}\d\xi\nonumber\\&\leq C\int_{s-t}^{s}e^{\frac{3\alpha}{4}(\xi-s)}\bigg\{\frac{1}{\xi-s+t}+1\bigg\}^{\frac{k_1}{2}}\d\xi \nonumber\\&\qquad\times\bigg[e^{-\frac{\alpha}{2k_1}t}\|\u_{0}\|^2_{\H} + \int_{-\infty}^{0}e^{\frac{\alpha}{2k_1}\zeta}\bigg(\|\f(\zeta+s)\|^2_{\L^2(\R^3)}+y_1(\vartheta_{\zeta}\omega)+y_2(\vartheta_{\zeta}\omega)\bigg)\d\zeta\bigg]^{\frac{k_1}{2}}   \nonumber\\&\quad + C\int_{s-t}^{s}e^{\frac{7\alpha}{8}(\xi-s)}\d\xi\bigg[ e^{-\frac{\alpha}{4k_1}t}\|\u_{0}\|^4_{\H} + \bigg(\int_{-\infty}^{0}e^{\frac{\alpha}{8k_1}\zeta_1}\left[\|\f(\zeta_1+s)\|^2_{\L^2(\R^3)}+y_1(\vartheta_{\zeta_1}\omega)\right]\d\zeta_1\bigg)^2\bigg]^{\frac{k_1}{2}}.
		\end{align}
		Since $$\int_{-t}^{0}e^{c\xi}\bigg[\frac{1}{\xi+t}+1\bigg]^{\frac{k_1}{2}}\d\xi\to\frac{1}{c}\ \text{ as }\ t\to\infty,$$ using \eqref{Z3} and the backward-uniform temperedness property \eqref{BackTem} of $\u_0$ (see \eqref{v_0} above), we obtain \eqref{AB12-A}, as desired.
		
		Again from \eqref{EI2-A}, we write for $\zeta>s-t$
		\begin{align*}
		&\frac{\d}{\d \zeta}\|\nabla\u(\zeta)\|^2_{\H}+\alpha\|\nabla\u(\zeta)\|^2_{\H}+\frac{C^*}{2}\|\u(\zeta)+\g y(\vartheta_{\zeta-s}\omega)\|^{r+1}_{\wi\L^{3(r+1)}}\nonumber\\&\leq \widetilde{R}\bigg[\|\f(\zeta)\|^{2}_{\L^2(\R^3)}+\|\u(\zeta)+\g y(\vartheta_{\zeta-s}\omega)\|^{r+1}_{\wi\L^{r+1}}+\|\u(\zeta)\|^2_{\V}+\|\u(\zeta)\|^4_{\H}+y_2(\vartheta_{\zeta-s}\omega)\bigg],
	\end{align*}
		which implies that
		\begin{align}\label{AB16V-A}
			&(\zeta-s+t)\frac{\d}{\d\zeta} \bigg[e^{\alpha(\zeta-s)}\|\nabla\u(\zeta)\|^2_{\H}\bigg]  +\frac{C^*}{2} (\zeta-s+t) e^{\alpha(\zeta-s)}\|\u(\zeta)+\g y(\vartheta_{\zeta-s}\omega)\|^{r+1}_{\wi\L^{3(r+1)}} \nonumber\\&\leq Ce^{\alpha(\zeta-s)}(\zeta-s+t)\bigg[\|\f(\zeta)\|^{2}_{\L^2(\R^3)}+\|\u(\zeta)+\g y(\vartheta_{\zeta-s}\omega)\|^{r+1}_{\wi\L^{r+1}}+\|\u(\zeta)\|^2_{\V}\nonumber\\&\quad+\|\u(\zeta)\|^4_{\H}+y_2(\vartheta_{\zeta-s}\omega)\bigg].
		\end{align}
	
		We know that
		\begin{align}\label{AB17V-A}
			&(\zeta-s+t)\frac{\d}{\d\zeta}\bigg[e^{\alpha(\zeta-s)} \|\nabla\u(\zeta)\|^2_{\H}\bigg]=\frac{\d}{\d\zeta}\bigg[(\zeta-s+t) e^{\alpha(\zeta-s)} \|\nabla\u(\zeta)\|^2_{\H}\bigg]- e^{\alpha(\zeta-s)}\|\nabla\u(\zeta)\|^2_{\H}.
		\end{align}
		From \eqref{AB16V-A} and \eqref{AB17V-A}, we infer
		\begin{align*}
			&\int_{s-t}^{s} (\zeta-s+t) e^{\alpha(\zeta-s)}\|\u(\zeta,s-t,\vartheta_{-s}\omega,\u_{0})+\g y(\vartheta_{\zeta-s}\omega)\|^{r+1}_{\wi\L^{3(r+1)}}\d\zeta\nonumber\\&
			\leq C\int_{s-t}^{s}(\zeta-s+t+1)e^{\alpha(\zeta-s)}\bigg[\|\f(\zeta)\|^{2}_{\L^2(\R^3)}+\|\u(\zeta)+\g y(\vartheta_{\zeta-s}\omega)\|^{r+1}_{\wi\L^{r+1}}+\|\u(\zeta)\|^2_{\V}+\|\u(\zeta)\|^4_{\H}\nonumber\\&\quad+y_2(\vartheta_{\zeta-s}\omega)\bigg]\d\zeta \nonumber\\&\leq C(t+1)\bigg[e^{-\frac{\alpha}{2}t}\|\u_{0}\|^2_{\H} + e^{-\frac{\alpha}{4}t}\|\u_{0}\|^4_{\H}+ \int_{-\infty}^{0}e^{\frac{\alpha}{2}\zeta}\bigg(\|\f(\zeta+s)\|^2_{\L^2(\R^3)}+y_1(\vartheta_{\zeta}\omega)+y_2(\vartheta_{\zeta}\omega)\bigg)\d\zeta\ \nonumber\\&\quad  + \bigg(\int_{-\infty}^{0}e^{\frac{\alpha}{8}\zeta}\left[\|\f(\zeta+s)\|^2_{\L^2(\R^3)}+y_1(\vartheta_{\zeta}\omega)\right]\d\zeta\bigg)^2\bigg],
		\end{align*}
		where we have used \eqref{AB4-A}, which completes the proof.
		
	\end{proof}
	\begin{proposition}\label{IRAS-A}
		For both the cases given in Table \ref{Table}, suppose that Assumption \ref{Hypo_f-N} holds. For $R$ and $\widetilde{K}(s,\omega), $ the same as in \eqref{EI1-A} and \eqref{AB2-A}, respectively, we have
		\vskip 2mm
		\noindent
		\emph{(i)} There is an increasing pullback $\mathfrak{D}$-random absorbing set $\mathcal{R}$ given by
		\begin{align}\label{IRAS1-N}
			\mathcal{R}(\tau,\omega):=\left\{\v\in\H:\|\v\|^2_{\H}\leq 4R\sup_{s\leq \tau} \widetilde{K}(s,\omega)+2\|\g\|^2_{\H}\left|y(\omega)\right|^2\right\}, \text{ for all } \tau\in\R \text{ and } \omega\in\Omega.
		\end{align}
		Moreover, $\mathcal{R}$ is backward-uniformly tempered with arbitrary rate, that is, $\mathcal{R}\in{\mathfrak{D}}$.
		\vskip 2mm
		\noindent
		\emph{(ii)} There is a $\mathfrak{B}$-pullback \textbf{random} absorbing set $\widetilde{\mathcal{R}}$ given by
		\begin{align}\label{IRAS11-N}
			\widetilde{\mathcal{R}}(\tau,\omega):=\left\{\v\in\H:\|\v\|^2_{\H}\leq 4R \widetilde{K}(\tau,\omega)+2\|\g\|^2_{\H}\left|y(\omega)\right|^2\right\}\in{\mathfrak{B}}, \text{ for all } \tau\in\R \text{ and } \omega\in\Omega.
		\end{align}
	\end{proposition}
	\begin{proof}
		See the proof of \cite[Proposition 3.6]{RKM}.
	\end{proof}
	
	\subsection{Backward uniform tail-estimates and backward flattening-property}
	In this subsection, we prove the backward tail-estimates and backward flattening-property for the solution of \eqref{2-A} for both the cases given in Table \ref{Table}. These estimates help us to prove the backward uniform pullback $\mathfrak{D}$-asymptotic compactness of the solution of \eqref{CCBF-A}. We will use the cut-off function (same as in Lemma \ref{largeradius}) to obtain these estimates. The following lemma provides the backward uniform tail-estimates for the solution of the system \eqref{2-A}.

	\begin{lemma}\label{largeradius-A}
		For both the cases given in Table \ref{Table}, suppose that Assumption \ref{Hypo_f-N} is satisfied. Then, for any $(\tau,\omega,D)\in\R\times\Omega\times\mathfrak{D},$ the solution of \eqref{SCBF-A} satisfies
		\begin{align}\label{ep-A}
			&\lim_{k,t\to+\infty}\sup_{s\leq \tau}\sup_{\u_{0}\in D(s-t,\vartheta_{-t}\omega)}\|\u(s,s-t,\vartheta_{-s}\omega,\u_{0})\|^2_{\mathbb{L}^2(\mathcal{O}^{c}_{k})}=0,
		\end{align}
		where $\mathcal{O}_{k}=\{x\in\R^3:|x|\leq k\}.$
	\end{lemma}
	\begin{proof}
		Let $\uprho$ be a smooth function defined in Lemma \ref{largeradius}. Taking divergence to the first equation in \eqref{2-A}, formally we obtain (see the proof of Lemma \ref{largeradius} the detailed calculations)
		\begin{align}\label{p-value-A}
			p=(-\Delta)^{-1}\left[\nabla\cdot\left[\nabla\cdot\big((\u+\g y) \otimes(\u+\g y) \big)\right] +\beta\nabla\cdot\left[|\u+\g y|^{r-1}(\u+\g y)\right]-\nabla\cdot\f \right].
		\end{align}
		Taking the inner product to the first equation of \eqref{2-A} with $\uprho\left(\frac{|x|^2}{k^2}\right)\u$, we have
		\begin{align}\label{ep1-A}
			&\frac{1}{2} \frac{\d}{\d t} \int_{\R^3}\uprho\left(\frac{|x|^2}{k^2}\right)|\u|^2\d x\nonumber \\&= \mu \int_{\R^3}(\Delta\u) \uprho\left(\frac{|x|^2}{k^2}\right) \u \d x-\alpha \int_{\R^3}\uprho\left(\frac{|x|^2}{k^2}\right)|\u|^2\d x-b\left(\u+\g y,\u+\g y,\uprho\left(\frac{|x|^2}{k^2}\right)(\u+\g y)\right)\nonumber\\&\quad+b\left(\u+\g y,\u+\g y,\uprho\left(\frac{|x|^2}{k^2}\right)\g y\right)-\beta \int_{\R^3}\uprho\left(\frac{|x|^2}{k^2}\right)\left|\u+\g y\right|^{r+1}\d x\nonumber\\&\quad+\beta \int_{\R^3}\left|\u+\g y\right|^{r-1}(\u+\g y)\uprho\left(\frac{|x|^2}{k^2}\right)\g y\d x-\int_{\R^3}(\nabla p)\uprho\left(\frac{|x|^2}{k^2}\right)\u\d x+ \int_{\R^3}\f\uprho\left(\frac{|x|^2}{k^2}\right)\u\d x \nonumber\\&\quad+(\sigma-\alpha)y\int_{\R^3}\g\uprho\left(\frac{|x|^2}{k^2}\right)\u\d x+\mu y\int_{\R^3}(\Delta\g)\uprho\left(\frac{|x|^2}{k^2}\right)\u\d x.
		\end{align}
	
		Let us now estimate each terms on right hand side of \eqref{ep1-A}. Using integration by parts, divergence free condition of $\u(\cdot)$ and $\g\in\D(\A)$, we infer (see inequalities \eqref{ep2}-\eqref{ep4})
		\begin{align}
				\mu \int_{\R^3}(\Delta\u) \uprho\left(\frac{|x|^2}{k^2}\right) \u \d x&\leq -\mu \int_{\R^3}|\nabla\u|^2 \uprho\left(\frac{|x|^2}{k^2}\right)  \d x +\frac{C}{k} \left[\|\u\|^2_{\H}+\|\nabla\u\|^2_{\H}\right],\label{ep2-A}\\
				y^2\ b\left(\u+\g y,\g,\uprho\left(\frac{|x|^2}{k^2}\right)\g\right)&\leq\frac{C}{k}\left[\|\u+\g y\|^2_{\H}+\left|y\right|^4\|\g\|^{4}_{\wi \L^{4}}\right]\leq\frac{C}{k}\left[\|\u\|^2_{\H}+\left|y\right|^2+\left|y\right|^4\right],\\
			-b\left(\u+\g y,\u+\g y,\uprho\left(\frac{|x|^2}{k^2}\right)(\u+\g y)\right)&\leq \frac{C}{k}\|\u+\g y\|^3_{\wi \L^3}\leq\frac{C}{k}\bigg[\|\u+\g y\|^2_{\H}+\|\u+\g y\|^{r+1}_{\wi \L^{r+1}}\bigg]\nonumber\\&\leq\frac{C}{k}\bigg[\|\u\|^2_{\H}+\left|y\right|^2+\|\u+\g y\|^{r+1}_{\wi \L^{r+1}}\bigg], \label{ep3-A}
		\end{align}
		where we have used interpolation and Young's inequalities. Using integration by parts, divergence free condition and \eqref{p-value-A}, we obtain 
		\begin{align}
			-&\int_{\R^3}(\nabla p)\uprho\left(\frac{|x|^2}{k^2}\right)\u\d x=\int_{\R^3}p\uprho'\left(\frac{|x|^2}{k^2}\right)\frac{2}{k^2}(x\cdot\u)\d x\nonumber\\& \leq \frac{C}{k}\int\limits_{\R^3}\left|(-\Delta)^{-1}\left[\nabla\cdot\left[\nabla\cdot\big((\u+\g y)\otimes(\u+\g y)\big)\right]\right]\right|\cdot\left|\u\right|\d x \nonumber\\&\quad+ \frac{C}{k}\int\limits_{\R^3}\left|(-\Delta)^{-1}\left[\nabla\cdot\left[|\u+\g y|^{r-1}(\u+\g y)\right]\right]\right|\cdot\left|\u\right|\d x+\frac{C}{k}\int_{\R^3}|(-\Delta)^{-1}[\nabla\cdot\f]|\cdot|\u|\d  x \nonumber\\&=: \frac{C}{k}\left[Q_1+Q_2+Q_3\right].
		\end{align}
		\vskip2mm
		\noindent
		\textbf{Estimate of $Q_1$}: Using $\g\in\D(\A)$, H\"older's inequality, Plancherel's theorem, and  interpolation and Young's inequalities, we get
		\begin{align}
			|Q_1| &\leq \left\|(-\Delta)^{-1}\left[\nabla\cdot\left[\nabla\cdot\big((\u+\g y)\otimes(\u+\g y)\big)\right]\right]\right\|_{\L^2(\R^3)}\|\u\|_{\H}\nonumber\\&\leq \|\u+\g y\|^2_{\wi\L^4}\|\u\|_{\H}\leq \|\u+\g y\|^{\frac{r-3}{r-1}}_{\H}\|\u+\g y\|^{\frac{r+1}{r-1}}_{\wi\L^{r+1}}\|\u\|_{\H} \nonumber\\&\leq C\bigg[\|\u\|^2_{\H}+\|\u+\g y\|^{r+1}_{\wi\L^{r+1}}+|y|^2\bigg].
		\end{align}
		\vskip2mm
		\noindent
		\textbf{Estimate of $Q_2$:} Similar to \eqref{S_2(d,r)}, applying H\"older's inequality, Plancherel's theorem, \cite[Theorem 1.38]{BCD} (Remark \ref{Hdot}), and interpolation and Young's inequalities, we obtain
		\begin{align}\label{Q_2(d,r)p1}
			|Q_2|
			&\leq C 
				\|\u+\g y\|^r_{\wi\L^{\frac{6r}{5}}}\|\u\|_{\H} 
		\leq C \|\u+\g y\|^{\frac{3r+5}{3r+1}}_{\H}\|\u+\g y\|_{\wi\L^{3(r+1)}}^{\frac{(3r-5)(r+1)}{3r+1}}\|\u\|_{\H}
			\nonumber\\&\leq C \bigg[\|\u+\g y\|^{\frac{3r+5}{3}}_{\H} +\|\u+\g y\|_{\wi\L^{3(r+1)}}^{\frac{(3r-5)(r+1)}{3r-4}}+\|\u\|_{\H}^{\frac{3r+1}{2}}\bigg]
			\nonumber\\&\leq C \bigg[\|\u\|_{\H}^2+\|\u\|^{2(r+1)}_{\H} +\|\u+\g y\|_{\wi\L^{3(r+1)}}^{\frac{(3r-5)(r+1)}{3r-4}}+|y|^2+|y|^{2(r+1)}\bigg],
		\end{align}
		where we have used interpolation and Young's inequalities.
			\vskip2mm
		\noindent
		\textbf{Estimate of $Q_3$}: Similar to \eqref{S_3(d,r)}, we find 
		\begin{align}\label{wiS_3(d,r)}
			\left|Q_3\right|&\leq 
			C\|\f\|^{2}_{\dot{\mathbb{H}}^{-1}(\R^3)}+C\|\u\|^2_{\H}.
		\end{align}
		
		Finally, we estimate the remaining terms of \eqref{ep1-A} by using H\"older's and Young's inequalities as follows,
		\begin{align}
			&y b\left(\u+\g y,\u,\uprho\left(\frac{|x|^2}{k^2}\right)\g\right)+\beta y\int_{\R^3}\left|\u+\g y\right|^{r-1}(\u+\g y)\uprho\left(\frac{|x|^2}{k^2}\right)\g \d x  \nonumber\\&+\int_{\R^3}\f(x)\uprho\left(\frac{|x|^2}{k^2}\right)\u \d x+(\ell-\alpha)y\int_{\R^3}\g \uprho\left(\frac{|x|^2}{k^2}\right)\u\d x+\mu y\int_{\R^3}(\Delta\g )\uprho\left(\frac{|x|^2}{k^2}\right)\u\d x \nonumber\\&\leq\frac{\beta}{2}\int_{\R^3}\uprho\left(\frac{|x|^2}{k^2}\right)\left|\u+\g y\right|^{r+1}\d x+\frac{\mu}{2}\int_{\R^3} \uprho\left(\frac{|x|^2}{k^2}\right)|\nabla\u|^2 \d x+\frac{\alpha}{2}\int_{\R^3} \uprho\left(\frac{|x|^2}{k^2}\right)|\u|^2 \d x \nonumber\\&\quad+C\int_{\R^3}\uprho\left(\frac{|x|^2}{k^2}\right)\bigg[\left|y\right|^{\frac{2(r+1)}{r-1}}\left|\g\right|^{\frac{2(r+1)}{r-1}}+\left|y\right|^{r+1}|\g|^{r+1}+|\f|^2+\left|y\right|^2 |\g|^2+\left|y\right|^2|\Delta\g|^2\bigg]\d x.	\label{ep4-A}
		\end{align}
		Combining \eqref{ep1-A}-\eqref{ep4-A}, we get
		\begin{align}\label{ep5-A}
			&\frac{\d}{\d t}\|\u\|^2_{\mathbb{L}^2(\mathcal{O}_k^{c})}+\alpha\|\u\|^2_{\mathbb{L}^2(\mathcal{O}_k^c)}\nonumber\\& \leq \frac{C}{k} \bigg[\|\u\|^2_{\V}+\|\u+\g y\|^{r+1}_{\wi\L^{r+1}}+\|\u+\g y\|_{\wi\L^{3(r+1)}}^{\frac{(3r-5)(r+1)}{3r-4}}+\|\u\|_{\H}^{2(r+1)}+\|\f\|^2_{\dot{\mathbb{H}}^{-1}(\R^3)}+\left|y\right|^{2}+\left|y\right|^{2(r+1)}\bigg]\nonumber\\&\quad+C\left|y\right|^{\frac{2(r+1)}{r-1}} \int_{|x|\geq k}\left|\g(x)\right|^{\frac{2(r+1)}{r-1}}\d x+C\left|y\right|^{r+1} \int_{|x|\geq k}|\g (x)|^{r+1}\d x+C \int_{|x|\geq k}|\f(x)|^2\d x\nonumber\\&\quad+C\left|y\right|^2 \int_{|x|\geq k}|\g (x)|^2\d x+C\left|y\right|^2 \int_{|x|\geq k}|\Delta\g (x)|^2\d x.
		\end{align}
		Making use of the variation of constant formula to the above equation \eqref{ep5-A} on $(s-t,s)$ and replacing $\omega$ by $\vartheta_{-s}\omega$, we find that, for $s\leq\tau, t\geq 0$ and $\omega\in\Omega$,
		\begin{eqnarray}\label{ep6-A}
			&&\|\u(s,s-t,\vartheta_{-s}\omega,\u_{0})\|^2_{\mathbb{L}^2(\mathcal{O}_k^{c})} \nonumber\\&& \leq e^{-\alpha t}\|\u_{0}\|^2_{\H}+\frac{C}{k}\bigg[\int_{s-t}^{s}e^{\alpha(\zeta-s)}\bigg\{\|\u(\zeta,s-t,\vartheta_{-s}\omega,\u_{0})\|^2_{\V}+\|\u(\zeta,s-t,\vartheta_{-s}\omega,\u_{0})\|^{2(r+1)}_{\H}\nonumber\\&&\quad+\|\u(\zeta,s-t,\vartheta_{-s}\omega,\u_{0})+\g y(\vartheta_{\zeta-s}\omega)\|^{r+1}_{\wi\L^{r+1}}+\|\u(\zeta,s-t,\vartheta_{-s}\omega,\u_{0})+\g y(\vartheta_{\zeta-s}\omega)\|^{\frac{(3r-5)(r+1)}{3r-4}}_{\wi\L^{3(r+1)}}\bigg\}\d\zeta\nonumber\\&&\quad+\int_{-\infty}^{0}e^{\alpha\zeta}\bigg\{\|\f(\zeta+s)\|^2_{\dot{\mathbb{H}}^{-1}(\R^3)}+\left|y(\vartheta_{\zeta}\omega)\right|^{2}+\left|y(\vartheta_{\zeta}\omega)\right|^{2(r+1)}\bigg\}\d\zeta\bigg]\nonumber\\&&\quad+C\int_{-\infty}^{0}e^{\alpha\zeta}\left|y(\vartheta_{\zeta}\omega)\right|^{\frac{2(r+1)}{r-1}}\d\zeta\int\limits_{|x|\geq k}\left|\g(x)\right|^{\frac{2(r+1)}{r-1}}\d x+C\int_{-\infty}^{0}e^{\alpha\zeta}\left|y(\vartheta_{\zeta}\omega)\right|^{r+1}\d\zeta\int\limits_{|x|\geq k}|\g (x)|^{r+1}\d x \nonumber\\&&\quad+C\int_{-\infty}^{0}e^{\alpha\zeta}\left|y(\vartheta_{\zeta}\omega)\right|^2\d\zeta\int\limits_{|x|\geq k}|\g (x)|^2\d x+C\int_{-\infty}^{0}e^{\alpha\zeta}\left|y(\vartheta_{\zeta}\omega)\right|^2\d\zeta \int\limits_{|x|\geq k}|\Delta\g (x)|^2\d x\nonumber\\&&\quad+C\int_{-\infty}^{0}e^{\alpha\zeta} \int\limits_{|x|\geq k}|\f(x,\zeta+s)|^2\d x\d\zeta.
		\end{eqnarray}
		Now, using the definition of backward temperedness \eqref{BackTem}, \eqref{Z3}, \eqref{f3-N}, $\g\in\D(\A)$ and Lemma \ref{Absorbing-A} (\eqref{AB1-A} and \eqref{AB11-A}-\eqref{AB18V-A}), one can complete the proof.
	\end{proof}

	\begin{lemma}\label{Flattening-N}
		For both the cases given in Table \ref{Table}, suppose that Assumption \ref{Hypo_f-N} holds. Let $(\tau,\omega,D)\in\R\times\Omega\times\mathfrak{D}$ and $k\geq1$ be fixed. Then
		\begin{align}\label{FL-P-A}
			\lim_{i,t\to+\infty}\sup_{s\leq \tau}\sup_{\u_{0}\in D(s-t,\vartheta_{-t}\omega)}\|(\I-\P_{i})\bar{\u}(s,s-t,\vartheta_{-s}\omega,\bar{\u}_{0,2})\|^2_{\L^2(\mathcal{O}_{\sqrt{2}k})}=0,
		\end{align}
		where $\bar{\u}_{0,2}=(\I-\P_{i})(\varrho_k\u_{0})$.
	\end{lemma}

	\begin{proof}
		The first equation of \eqref{2-A} can be rewritten as (multiplying by $\varrho_k$):
		\begin{align}\label{FL1-A}
			&\frac{\d\bar{\u}}{\d t}-\mu\Delta\bar{\u}+\varrho_k\big((\u+\g y)\cdot\nabla\big)(\u+\g y)+\alpha\bar{\u}+\varrho_k|\u+\g y|^{r-1}(\u+\g y)+\varrho_k\nabla p\nonumber\\&=-\mu\u\Delta\varrho_k-2\mu\nabla\varrho_k\cdot\nabla\u+\varrho_k\f +(\sigma-\alpha)\varrho_k\g y+\mu y\varrho_k\Delta\g.
		\end{align}
		Applying the projection $(\I-\P_i)$ and taking the inner product with $\bar{\u}_{i,2}$ in $\L^2(\mathcal{O}_{\sqrt{2}k})$ to the equation \eqref{FL1-A}, we get
		\begin{align}\label{FL2-A}
			&\frac{1}{2}\frac{\d}{\d t}\|\bar{\u}_{i,2}\|^2_{\L^2(\mathcal{O}_{\sqrt{2}k})} +\mu\|\nabla\bar{\u}_{i,2}\|^2_{\L^2(\mathcal{O}_{\sqrt{2}k})}+\alpha\|\bar{\u}_{i,2}\|^2_{\L^2(\mathcal{O}_{\sqrt{2}k})}+\beta\||\u+\g y|^{\frac{r-1}{2}}(\bar{\u}_{i,2}+\bar{\g}_{i,2}y)\|^2_{\L^2(\mathcal{O}_{\sqrt{2}k})}\nonumber\\&=-\underbrace{\sum_{q,m=1}^{3}\int_{\mathcal{O}_{\sqrt{2}k}}\left(\I-\P_i\right)\bigg[(u_{q}+g_{q}y)\frac{\partial(u_{m}+g_{m}y)}{\partial x_q}\left\{\varrho_k(x)\right\}^2(u_{m}+g_{m}y)\bigg]\d x}_{:=L_1}\nonumber\\&\quad+\underbrace{y\sum_{q,m=1}^{3}\int_{\mathcal{O}_{\sqrt{2}k}}\left(\I-\P_i\right)\bigg[(u_{q}+g_{q}y)\frac{\partial(u_{m}+g_{m}y)}{\partial x_q}\left\{\varrho_k(x)\right\}^2g_{m}\bigg]\d x}_{:=L_2}\nonumber\\& \quad + \underbrace{\beta \int_{\mathcal{O}_{\sqrt{2}k}}|\u+\g y|^{r-1}(\mathrm{P}_{i}\bar{\u}+\mathrm{P}_{i}\bar{\g} y)\cdot(\bar{\u}_{i,2}+\bar{\g}_{i,2} y)\d x}_{=:L_3} \nonumber\\&\quad+\underbrace{y\int_{\mathcal{O}_{\sqrt{2}k}}\bigg[|\u+\g y|^{r-1}\varrho_k(\u+\g y)\bar{\g}_{i,2} \bigg]\d x}_{:=L_4} - \underbrace{\big(\varrho_k(x)\nabla p, \bar{\u}_{i,2}\big)}_{:=L_5} \nonumber\\&\quad-\underbrace{\left\{\mu\big(\u\Delta\varrho_k+2\nabla\varrho_k\cdot\nabla\u,\bar{\u}_{i,2}\big)-\big(\varrho_k\f,\bar{\u}_{i,2}\big) -(\sigma-\alpha) y\big(\varrho_k\g,\bar{\u}_{i,2}\big)-\mu y\big(\varrho_k\Delta\g,\bar{\u}_{i,2}\big)\right\}}_{:=L_6}.
		\end{align}
		Next, we estimate each terms on the right hand side of \eqref{FL2-A} as follows:
		\vskip 2mm
		\noindent
		\textbf{Estimate of $L_1$:} Using integration by parts, divergence free condition of $\u(\cdot)$, \eqref{poin-i} (WLOG we assume that $\lambda_{i}\geq1$), H\"older's, Ladyzhenskaya's and Young's inequalities, we find
		\begin{align}\label{FL3-A}
			\left|L_1\right|&=\left|\int_{\mathcal{O}_{\sqrt{2}k}}\left(\I-\P_i\right)\bigg[\uprho'\left(\frac{|x|^2}{k^2}\right)\frac{x}{k^2}\cdot\left\{\bar{\u}+\bar{\g}y\right\}|\u+\boldsymbol{g}y|^2\bigg]\d x\right| \nonumber\\&\leq C\|\bar{\u}_{i,2}+\bar{\g}_{i,2}y\|_{\L^2(\mathcal{O}_{\sqrt{2}k})}\|\u+\g y\|^2_{\wi\L^4}\nonumber\\&\leq C\lambda_{i+1}^{-\frac{1}{2}}\|\nabla(\bar{\u}_{i,2}+\bar{\g}_{i,2}y)\|_{\L^2(\mathcal{O}_{\sqrt{2}k})}\|\u+\g y\|^{\frac{r-3}{r-1}}_{\H} \|\u+\g y\|^{\frac{r+1}{r-1}}_{\wi\L^{r+1}}
			\nonumber\\&\leq  C\lambda_{i+1}^{-\frac{1}{2}}\|\u+\g y\|_{\V}\|\u+\g y\|^{\frac{r-3}{r-1}}_{\H} \|\u+\g y\|^{\frac{r+1}{r-1}}_{\wi\L^{r+1}}
			\nonumber\\&\leq C\lambda_{i+1}^{-\frac{1}{2}}\bigg[\|\u\|^2_{\V} +\|\u+\g y\|^{r+1}_{\wi\L^{r+1}}+|y|^2\bigg].
		\end{align}
		\vskip 2mm
		\noindent
		\textbf{Estimate of $L_2$ and $L_4$:} Using $\g\in\D(\A)$, H\"older's, Agmon's, \eqref{poin-i}, interpolation and Young's inequalities, respectively, we get 
		\begin{align}\label{FL4-A}
				&\left|L_2+L_4\right|\nonumber\\&\leq|y|\|\bar{\g}_{i,2}\|_{\L^{\infty}(\mathcal{O}_{\sqrt{2}k})}\bigg[\|\u+\g y\|_{\H} \|\nabla(\u+\g y)\|_{\H}+\|\u+\g y\|^{r}_{\wi\L^{r}}\bigg]\nonumber\\
		&\leq C\|\bar{\g}_{i,2}\|^{\frac{1}{4}}_{\L^{2}(\mathcal{O}_{\sqrt{2}k})}\|\bar{\g}_{i,2}\|^{\frac{3}{4}}_{\H^{2}(\mathcal{O}_{\sqrt{2}k})}|y|\bigg[\|\u+\g y\|_{\H} \|\nabla(\u+\g y)\|_{\H}+
		\|\u+\g y\|^{\frac{2}{r-1}}_{\H}\|\u+\g y\|^{\frac{(r+1)(r-2)}{r-1}}_{\wi\L^{r+1}}\bigg]
		\nonumber\\&\leq C\lambda^{-\frac{1}{8}}_{i+1}\bigg[\|\u+\g y\|^4_{\H}+ \|\nabla(\u+\g y)\|^2_{\H}+\|\u+\g y\|^{r+1}_{\H}+\|\u+\g y\|^{r+1}_{\wi\L^{r+1}}+|y|^4+|y|^{r+1}\bigg]
		\nonumber\\&\leq C\lambda^{-\frac{1}{8}}_{i+1}\bigg[\|\u+\g y\|^2_{\V}+ \|\u+\g y\|^{r+1}_{\H}+\|\u+\g y\|^{r+1}_{\wi\L^{r+1}}+|y|^4+|y|^{r+1}\bigg]
		\nonumber\\&\leq C\lambda^{-\frac{1}{8}}_{i+1}\bigg[\|\u\|^2_{\V}+\|\u\|^{r+1}_{\H}+\|\u+\g y\|^{r+1}_{\wi\L^{r+1}}+|y|^2+|y|^{r+1}\bigg].
		\end{align}
		\vskip 2mm
		\noindent
		\textbf{Estimate of $L_5$:} Using integration by parts, divergence free condition for $\u(\cdot)$ and \eqref{p-value-A}, we obtain
		\begin{align}
			|L_5|&=\left|\int_{\mathcal{O}_{\sqrt{2}k}}(\I-\P_i)\bigg[p\uprho'\left(\frac{|x|^2}{k^2}\right)\frac{4}{k^2}(x\cdot\bar{\u})\bigg]\d x\right|\nonumber\\& \leq C\int_{\mathcal{O}_{\sqrt{2}k}}\left|(-\Delta)^{-1}\left[\nabla\cdot\left[\nabla\cdot\big((\u+\g y)\otimes(\u+\g y)\big)\right]\right]\right|\cdot\left|\bar{\u}_{i,2}\right|\d x \nonumber\\&\quad+ C\int_{\mathcal{O}_{\sqrt{2}k}}\left|(-\Delta)^{-1}\left[\nabla\cdot\left[|\u+\g y|^{r-1}(\u+\g y)\right]\right]\right|\cdot\left|\bar{\u}_{i,2}\right|\d x +C\int_{\mathcal{O}_{\sqrt{2}k}}|(-\Delta)^{-1}[\nabla\cdot\f]|\cdot|\bar{\u}_{i,2}|\d  x \nonumber\\&=: C\left[\widetilde{Q}_1+\widetilde{Q}_2+\widetilde{Q}_3\right].
		\end{align}
		\vskip2mm
		\noindent
		\textit{Estimate of $\widetilde{Q}_1$}: Using H\"older's inequality, Plancherel's theorem, interpolation and Young's inequalities, we get 
		\begin{align}
			|\widetilde{Q}_1|&\leq \left\|(-\Delta)^{-1}\left[\nabla\cdot\left[\nabla\cdot\big((\u+\g y)\otimes(\u+\g y)\big)\right]\right]\right\|_{\L^2(\R^3)}\|\bar{\u}_{i,2}\|_{\L^2(\mathcal{O}_{\sqrt{2}k})}\nonumber\\&\leq \|\u+\g y\|^2_{\wi\L^4} \|\bar{\u}_{i,2}\|_{\L^2(\mathcal{O}_{\sqrt{2}k})} \leq C\lambda^{-\frac{1}{2}}_{i+1} \bigg[\|\u\|^2_{\V}+\|\u +\g y\|^{r+1}_{\wi\L^{r+1}}+|y|^2\bigg]
		\end{align}
		\vskip2mm
		\noindent
		\textit{Estimate of $\widetilde{Q}_2$:} Making use of H\"older's inequality, Plancherel's theorem, \cite[Theorem 1.38]{BCD} (Remark \ref{Hdot}), \eqref{poin-i}, interpolation and Young's inequalities (see \eqref{S_2(d,r)p}), we find
		\begin{align}\label{Q_2(d,r)p}
				|\widetilde{Q}_2|&\leq C
				\|\u+\g y\|^r_{\wi\L^{\frac{6r}{5}}}\|\bar{\u}_{i,2}\|_{\L^2(\mathcal{O}_{\sqrt{2}k})},
			\nonumber\\&\leq C \lambda_{i+1}^{-\frac{1}{2}} \|\u+\g y\|^{\frac{3r+5}{3r+1}}_{\H}\|\u+\g y\|^{\frac{(3r-5)(r+1)}{3r+1}}_{\wi\L^{3(r+1)}}\|\nabla\bar{\u}_{i,2}\|_{\L^2(\mathcal{O}_{\sqrt{2}k})}\nonumber\\& \leq C\lambda_{i+1}^{-\frac{1}{2}}  \bigg[\|\u+\g y\|^{\frac{3r+5}{3}}_{\H} +\|\u+\g y\|_{\wi\L^{3(r+1)}}^{\frac{(3r-5)(r+1)}{3r-4}}+\|\u\|_{\V}^{\frac{3r+1}{2}}\bigg]
			\nonumber\\&\leq C \lambda_{i+1}^{-\frac{1}{2}}\bigg[\|\u\|_{\V}^2+\|\u\|^{2(r+1)}_{\V} +\|\u+\g y\|_{\wi\L^{3(r+1)}}^{\frac{(3r-5)(r+1)}{3r-4}}+|y|^2+|y|^{2(r+1)}\bigg]
		\end{align}
	\vskip 2mm
	\noindent
		\textit{Estimate of $\widetilde{Q}_3$}: Similar to \eqref{S_3(d,r)p}, we find 
		\begin{align}\label{Q_3(d,r)p}
			|\widetilde{Q}_3|&\leq 
			 \frac{\mu}{4}\|\nabla\bar{\u}_{i,2}\|^2_{\L^2(\mathcal{O}_{\sqrt{2}k})}+ C\lambda^{-1}_{i+1}\|\f\|^2_{\dot{\mathbb{H}}^{-1}(\R^3)}.
		\end{align}
		\vskip 2mm
		\noindent
		\textbf{Estimate of $L_6$:} Applying H\"older's and Young's inequalities, we deduce
		\begin{align}
			\left|L_6\right|&\leq C\bigg[\|\u\|_{\H}+ \|\nabla\u\|_{\H}+\|\f\|_{\L^2(\R^3)}+\left|y\right|\bigg]\|\bar{\u}_{i,2}\|_{\L^2(\mathcal{O}_{\sqrt{2}k})}\nonumber\\&\leq C\lambda^{-1/2}_{i+1}\bigg[\|\u\|_{\H}+\|\nabla\u\|_{\H}+\|\f\|_{\L^2(\R^3)}+\left|y\right|\bigg]\|\nabla\bar{\u}_{i,2}\|_{\L^2(\mathcal{O}_{\sqrt{2}k})}\nonumber\\&\leq\frac{\mu}{4} \|\nabla\bar{\u}_{i,2}\|^2_{\L^2(\mathcal{O}_{\sqrt{2}k})}+ C\lambda^{-1}_{i+1}\bigg[\|\u\|^2_{\H}+\|\nabla\u\|^2_{\H}+\|\f\|^2_{\L^2(\R^3)}+\left|y\right|^2\bigg].\label{FL5-A}
		\end{align}
				\vskip 2mm
			\noindent
			\textbf{Estimate of $L_3$:} Using H\"older's, \eqref{poin-i}, interpolation and Young's inequalities, we obtain
	\begin{eqnarray}\label{L3}
		&&|L_3|\nonumber\\&&\leq C \int_{\mathcal{O}_{\sqrt{2}k}}|\u+\g y|^{\frac{r-1}{4}} |\bar{\u}_{i,2}+\bar{\g}_{i,2}y|^{\frac{1}{2}}|\bar{\u}_{i,2}+\bar{\g}_{i,2}y|^{\frac{1}{2}}|\u+\g y|^{\frac{3(r-1)}{4}} |\mathrm{P}_{i}(\bar{\u}+\bar{\g} y)|\d x \nonumber\\&&\leq C  \|\left|\u+\g y\right|^{\frac{r-1}{2}}(\bar{\u}_{i,2}+\bar{\g}_{i,2} y)\|^{\frac{1}{2}}_{\L^2(\mathcal{O}_{\sqrt{2}k})}  \|\bar{\u}_{i,2}+\bar{\g}_{i,2} y\|^{\frac{1}{2}}_{\L^2(\mathcal{O}_{\sqrt{2}k})} \|\u+\g y\|^{\frac{3(r-1)}{4}}_{\wi\L^{3(r+1)}} \|\u+\g y\|_{\wi\L^{\frac{4(r+1)}{r+3}}}  \nonumber\\&&\leq C \lambda_{i+1}^{-\frac{1}{4}} \|\left|\u+\g y\right|^{\frac{r-1}{2}} (\bar{\u}_{i,2}+\bar{\g}_{i,2}y)\|^{\frac{1}{2}}_{\L^2(\mathcal{O}_{\sqrt{2}k})}  \|\nabla(\bar{\u}_{i,2}+\bar{\g}_{i,2}y)\|^{\frac{1}{2}}_{\L^2(\mathcal{O}_{\sqrt{2}k})}\|\u+\g y\|^{\frac{3(r-1)}{4}}_{\wi\L^{3(r+1)}}\nonumber\\&&\quad\times\|\u+\g y\|_{\wi\L^{\frac{4(r+1)}{r+3}}}   \nonumber\\&&\leq C \lambda_{i+1}^{-\frac{1}{4}} \|\left|\u+\g y\right|^{\frac{r-1}{2}}(\bar{\u}_{i,2}+\bar{\g}_{i,2} y)\|^{\frac{1}{2}}_{\L^2(\mathcal{O}_{\sqrt{2}k})}  \|\u+\g y\|^{\frac{1}{2}}_{\V} \|\u+\g y\|^{\frac{3(r-1)}{4}}_{\wi\L^{3(r+1)}} \|\u+\g y\|^{\frac{1}{2}}_{\H} \|\u+\g y\|^{\frac{1}{2}}_{\wi\L^{r+1}}\nonumber\\&& \leq \frac{\beta}{2} \|\left|\u+\g y\right|^{\frac{r-1}{2}}(\bar{\u}_{i,2}+\bar{\g}_{i,2}y)\|^2_{\L^2(\mathcal{O}_{\sqrt{2}k})} +C \lambda_{i+1}^{-\frac{1}{3}}\bigg[\|\u+\g y\|^{r+1}_{\V}+\|\u+\g y\|^{\frac{3(r-1)(r+1)}{3r-2}}_{\wi\L^{3(r+1)}}\nonumber\\&&\quad+\|\u+\g y\|^{2(r+1)}_{\H}  + \|\u+\g y\|^{r+1}_{\wi\L^{r+1}}\bigg].
	\end{eqnarray}
		Now, combining \eqref{FL2-A}-\eqref{FL5-A}, applying the variation of constant formula, using Lemma \ref{Absorbing-A} (\eqref{AB1-A} and \eqref{AB11-A}-\eqref{AB18V-A}) and passing limit $i\to\infty$, ($\lambda_{i+1}\to0$ as $i\to\infty$), we demonstrate \eqref{FL-P-A}, as desired (see the proof of Lemma \ref{Flattening}), which  completes the proof.
	\end{proof}
	
	\subsection{Proof of Theorem \ref{MT1}}
	This subsection is devoted to the proof of main result of this section, that is, the existence of pullback $\mathfrak{D}$-random attractors and their asymptotic autonomy for the solution of the system \eqref{SCBF} with $S(\v)=\g\in\D(\A)$. For both the cases given in Table \ref{Table}, the existence of pullback $\mathfrak{D}$-random attractors for non-autonomous 3D stochastic CBF equations driven by additive noise on the whole space is established in \cite{KM7}. For both the cases given in Table \ref{Table}, as the existence of a unique pullback random attractor is known for each $\tau$, one can obtain the existence of a unique random attractor for an autonomous 3D stochastic CBF equations driven by additive noise on the whole space (cf. \cite{KM7}).
	
	In view of Propositions \ref{Back_conver-A} and \ref{IRAS-A}, and Lemmas \ref{largeradius-A} and \ref{Flattening-N}, the proof of Theorem \ref{MT1} can be obtained by applying similar arguments as in the proof of \cite[Theorem 1.6]{RKM} (Subsection 3.5 in \cite{RKM}) and \cite[Theorem 5.2]{CGTW}.

	\medskip\noindent
	{\bf Acknowledgments:}    The first author would like to thank the Council of Scientific $\&$ Industrial Research (CSIR), India for financial assistance (File No. 09/143(0938)/2019-EMR-I). M. T. Mohan would  like to thank the Department of Science and Technology (DST), Govt of India for Innovation in Science Pursuit for Inspired Research (INSPIRE) Faculty Award (IFA17-MA110). Renhai Wang was supported by China Postdoctoral Science Foundation under grant numbers 2020TQ0053 and 2020M680456.

\begin{thebibliography}{99}
		
		\bibitem{SNA}	S.N. Antontsev and H.B. de Oliveira, The Navier-Stokes problem modified by an absorption term, \emph{Appl. Anal.}, {\bf 89}(12),  2010, 1805--1825.
%
%
		\bibitem{Arnold}	L. Arnold, \emph{Random Dynamical Systems}, Springer-Verlag, Berlin, Heidelberg, New York, 1998.
%
%
		\bibitem{Ball} J.M. Ball, Global attractors for damped semilinear wave equations, \emph{Discrete Contin. Dyn. Syst.}, \textbf{10}(1--2) (2004), 31--52.
%
%
		\bibitem{BCD} H. Bahouri, J-Y Chemin and R. Danchin, \emph{Fourier Analysis and Nonlinear Partial Differential Equations}, Fundamental Principles of Mathematical Sciences, 343, Springer, Heidelberg, 2011.
%
%
%
		\bibitem{BCL}	M.C. Bortolan, A.N. Carvalho and J.A. Langa, \emph{Attractors under autonomous and non- autonomous perturbations}, Mathematical Surveys and Monographs, AMS, 2020.
%
%
		\bibitem{BCF}  	Z. Brze\'zniak, M. Capi\'nski and F. Flandoli, Pathwise global attractors for stationary random dynamical systems, \emph{Probab. Theory Related Fields}, \textbf{95}(1) (1993), 87--102.
%
%
%
		\bibitem{BCLLLR} Z. Brz\'ezniak, T. Caraballo, J.A. Langa, Y. Li, G. Lukaszewicz and J. Real, Random attractors for stochastic 2D Navier-Stokes equations in some unbounded domains, \emph{J. Differential Equations}, \textbf{255}(11) (2013), 3897--3919.
%
%
		\bibitem{BL} Z. Brz\'ezniak and Y. Li, Asymptotic compactness and absorbing sets for 2D stochastic Navier-Stokes equations in some unbounded domains, \emph{Trans. Amer. Math. Soc.}, \textbf{358}(12) (2006), 5587--5629.
%
%
%
%
%
		\bibitem{CGSV} T. Caraballo, M.J. Garrido-Atienza, B. Schmalfuss and J. Valero, Asymptotic behaviour of a stochastic semilinear dissipative functional equation without uniqueness of solutions, \emph{Discrete Contin. Dyn. Syst. Ser. B}, \textbf{14}(2) (2010), 439--455.
%
%
%
		\bibitem{CGTW} T. Caraballo, B. Guo, N. Tuan and R. Wang, Asymptotically autonomous robustness of random attractors for a class of weakly dissipative stochastic wave equations on unbounded domains, \emph{Proc. Roy. Soc. Edinburgh Sect. A}, \textbf{151}(6) (2021), 1700--1730.
%
		\bibitem{CLR} T. Caraballo, G. Lukaszewicz and J. Real, Pullback attractors for asymptotically compact non-autonomous dynamical systems, \emph{Nonlinear Anal.} \textbf{64}(3) (2006), 484--498.
	\bibitem{CLR1} T. Caraballo, G. Lukaszewicz and J. Real, Pullback attractors for non-autonomous 2D-Navier-Stokes equations in some unbounded domains, \emph{C. R. Acad. Sci. Paris, Ser. I}, \textbf{342}(4) (2006), 263--268.
%
\bibitem{CLR1HGHDGF} T. Caraballo, L. Mchiri, M. Rhaima, Ulam-Hyers-Rassias stability of neutral stochastic functional differential equations, \emph{Stochastics}, {\bf 94} (2022) 959--971.
%
\bibitem{CLR1HFJFFGHDGF} T. Caraballo, F. Ezzine, M.A. Hammami and L. Mchiri, Practical stability with respect to a part of variables of stochastic differential equations, \emph{Stochastics},  \textbf{93}(5) (2021) 647--664.
%
%
		\bibitem{Caraballo201hf6na} T. Caraballo, X. Han, B. Schmalfu\ss and J. Valero, Random attractors for stochastic lattice dynamical systems with infinite multiplicative white noise, \emph{Nonlinear Anal.}, \textbf{130} (2016) 255--278.
%
%
		\bibitem{Caraballomb2}T. Caraballo, P.E. Kloeden and B. Schmalfu\ss, Exponentially stable stationary solutions for stochastic evolution equations and their perturbation, \emph{Appl. Math. Optim.}, \textbf{50} (2004), 183--207.		
%
%
%
%
%
%
%
		\bibitem{CLR2} A. Carvalho, J.A. Langa and J. Robinson, \emph{Attractors for Infinite-dimensional Non-autonomous Dynamical Systems}, Netherlands: Springer New York, 2013.
%
%
%
%
%
		\bibitem{Chueshov2} I. Chueshov and I. Lasiecka,  \emph{Long-Time Behavior of Second Order Evolution Equations with Nonlinear Damping}, Memoirs of the American Mathematical Society, \textbf{195}, 2008.
%
%
		\bibitem{CV2} V.V. Chepyzhov and M.I. Vishik, \emph{Attractors for Equations of Mathematical Physics}, American Mathematical Society, Providence, Rhode Island, 2002.
%
%
%
		\bibitem{chenp} P. Chen, B. Wang, R. Wang and X. Zhang, Multivalued random dynamics of Benjamin-Bona-Mahony equations driven by nonlinear colored noise on unbounded domains, \emph{Math. Ann.}, (2022), https://doi.org/10.1007/s00208-022-02400-0.
%
%
%
		\bibitem{CDF}	H. Crauel, A. Debussche and F. Flandoli, Random attractors, \emph{J. Dynam. Differential Equations}, \textbf{9}(2) (1995), 307--341.
%
%
%
		\bibitem{CF}	H. Crauel and F. Flandoli, Attractors for random dynamical systems, \emph{Probab. Theory Related Fields}, \textbf{100}(3) (1994), 365--393.
%
%
%
%
%
%
		\bibitem{HCPEK}	H. Cui and P.E. Kloeden, Convergence rate of random attractors for 2D Navier-Stokes equation towards the deterministic singleton attractor,  Chapter 10 in \emph{Contemporary Approaches and Methods in Fundamental Mathematics and Mechanics}, Springer, 2021.
%
%
%
		\bibitem{CLL} H. Cui, J.A. Langa and Y. Li, Measurability of random attractors for quasi strong-to-weak continuous random dynamical systems, \emph{J. Dynam. Differential Equations}, \textbf{30}(4) (2018), 1873--1898.
%
%
	\bibitem{LCE}	 L.C. Evans, \emph{Partial differential equations}, 2nd Ed.,  American Mathematical Society, RI, 2010.
%
%
		\bibitem{sum_and_inter} R. Farwig, H. Kozono and H. Sohr, An $L^q$-approach to Stokes and Navier-Stokes equations in general domains, \emph{Acta Math.}, \textbf{195} (2005), 21--53.
%
%
%
		\bibitem{FAN} X. Fan, Attractors for a damped stochastic wave equation of the sine-Gordon type with sublinear multiplicative noise, \emph{Stoch. Anal. Appl.}, \textbf{24}(4) (2006), 767--793.
%
%
		\bibitem{FHR} 	C.L. Fefferman, K.W. Hajduk and J.C. Robinson, Simultaneous approximation in Lebesgue and Sobolev norms via eigenspaces, \emph{Proc. London Math. Soc.}, \textbf{3} (2022), 1--19.
%
%
		\bibitem{FY}	X. Feng and B. You, Random attractors for the two-dimensional stochastic g-Navier-Stokes equations, \emph{Stochastics}, \textbf{92}(4) (2020), 613--626.
%
%
%
%
%
		\bibitem{FS} F. Flandoli and B. Schmalfu\ss, Random attractors for the 3D stochastic Navier-Stokes equation with multiplicative noise, \emph{Stoch. Stoch. Rep.}, \textbf{59}(1--2) (1996), 21--45.
%
%
%
%
%
%
%
		\bibitem{GGW}  A. Gu, B. Guo and B. Wang, Long term behavior of random Navier-Stokes equations driven by colored noise, \emph{Discrete Contin. Dyn. Syst. Ser. B}, \textbf{25}(7) (2020), 2495--2532.
%
%
%
%
%
		\bibitem{GLW} A. Gu, K. Lu and B. Wang, Asymptotic behavior of random Navier-Stokes equations driven by Wong-Zakai approximations, \emph{Discrete Contin. Dyn. Syst. Ser. B}, \textbf{39}(1) (2019), 185--218.
%
%
%
%
%
%
%
		\bibitem{HR}	K.W. Hajduk and J.C. Robinson, Energy equality for the 3D critical convective Brinkman-Forchheimer equations, \emph{J. Differential Equations}, {\bf 263}(11) (2017), 7141--7161.
%
%
		\bibitem{HZ} Z. Han and S. Zhou, Random exponential attractor for the 3D non-autonomous stochastic damped Navier-Stokes equation, \emph{J. Dynam. Differential Equations}, (2021).
%
%
%
%
%
		\bibitem{Kloeden2007prsl} P.E. Kloeden, J.A. Langa, Flattening, squeezing and the existence of random attractors, \emph{Proc. R. Soc. Lond. Ser. A Math. Phys. Eng. Sci.}, \textbf{463}(2077) (2007), 163--181.
%
		\bibitem{KM} K. Kinra and M. T. Mohan, Random attractors and invariant measures for stochastic convective Brinkman-Forchheimer equations on 2D and 3D unbounded domains, Accepted in \emph{Discrete Contin. Dyn. Syst. Ser. B}, (2023), \url{https://arxiv.org/pdf/2010.08753.pdf}.
%
%
%
%
%
%
		\bibitem{KM2} K. Kinra and M.T. Mohan, $\H^1$-Random attractors for 2D stochastic convective Brinkman-Forchheimer equations in unbounded domains, Accepted in \emph{Adv. Differential Equations}, (2022), \url{https://arxiv.org/pdf/2111.07841.pdf}.
%
%
%
%
%
%
%
%
		\bibitem{KM4} K. Kinra and M.T. Mohan, Weak pullback mean random attractors for the stochastic convective Brinkman-Forchheimer equations and locally monotone stochastic partial differential equations, \emph{Infin. Dimens. Anal. Quantum Probab. Relat. Top.}, \textbf{25}(1) (2022), 2250005.
%
%
%
%
%
%
%
		\bibitem{KM6} K. Kinra and M.T. Mohan, Existence and upper semicontinuity of random pullback attractors for 2D and 3D non-autonomous stochastic convective Brinkman-Forchheimer equations on whole space, \emph{Differential Integral Equations}, \textbf{36}(5--6) (2023), 367--412. 
%
%
		\bibitem{KM7} K. Kinra and M.T. Mohan, Long term behavior of 2D and 3D non-autonomous random convective Brinkman-Forchheimer equations driven by colored noise, \emph{Submitted}, \url{https://arxiv.org/pdf/2107.08890.pdf}.
%
%
%
		\bibitem{KRM} K. Kinra, R. Wang and M.T. Mohan, Asymptotic autonomy of random attractors in regular spaces for non-autonomous stochastic Navier-Stokes equations, \emph{Submitted}, \url{https://arxiv.org/pdf/2205.02099.pdf}.
%
%
		\bibitem{KM8}	K. Kinra and M.T. Mohan, Bi-spatial random attractor, ergodicity and a random Liouville type theorem for stochastic Navier-Stokes equations on the whole space, \emph{Submitted}, \url{https://arxiv.org/pdf/2209.08915.pdf}.
%
%
%
%
%
%
%
%
		\bibitem{Kuratowski} K. Kuratowski, Sur les espaces complets, \emph{Fund. Math.}, \textbf{1}(15) (1930), 301--309.
%
%
		\bibitem{LGL} Y. Li, A. Gu and J. Li, Existence and continuity of bi-spatial random attractors and application to stochastic semilinear Laplacian equations, \emph{J. Dynam. Differential Equations}, \textbf{258}(2) (2015), 504--534.
%
%
%
%
%
%
%
%
%
%
%
%
%
		\bibitem{LG} H. Liu and H. Gao, Ergodicity and dynamics for the stochastic 3D Navier-Stokes equations with damping, \emph{Commun. Math. Sci.}, \textbf{16}(1) (2018), 97--122.
%
		\bibitem{LX1}	F. Li and D. Xu, Asymptotically autonomous dynamics for non-autonomous stochastic $\boldsymbol{g}$-Navier-Stokes equation with additive noise, \emph{Discrete Contin. Dyn. Syst. Ser. B}, \textbf{28}(1) (2023), 516--537.
%
%
%
%
%
		\bibitem{YR} Y. Li and R. Wang, Asymptotic autonomy of random attractors for BBM equations with Laplace-multiplier noise, \emph{J. Appl. Anal. Comput.}, \textbf{10}(4) (2020), 1199--1222.
%
%
		\bibitem{MWZ} Q. Ma, S. Wang and C. Zhong, Necessary and sufficient conditions for the existence of global attractors for semigroups and applications, \emph{Indiana Univ. Math. J.},  \textbf{51}(6) (2002), 1541--1559.
%
%
%
%
%
		\bibitem{PAM}	P.A. Markowich, E.S. Titi and S. Trabelsi,	Continuous data assimilation for the three-dimensional Brinkman-Forchheimer-extended Darcy model, \emph{Nonlinearity}, {\bf 29}(4), (2016), 1292--1328.
%
%
%
		\bibitem{DMDZ} D. Mitrovic and D. Zubrinic, \emph{Fundamentals of Applied Functional Analysis: Distributions-Sobolev Spaces-Nonlinear Elliptic Equations}, Pitman Monographs and Surveys in Pure and Applied Mathematics 91, 1998.
%
%
%
		\bibitem{MTM}	 M.T. Mohan, On the convective Brinkman-Forchheimer equations, \emph{Submitted}.
%
%
		\bibitem{MTM1}	 M.T. Mohan, Stochastic convective Brinkman-Forchheimer equations, \emph{Submitted}, \url{https://arxiv.org/abs/2007.09376}.
%
%
		\bibitem{MTM2}	M.T. Mohan, Asymptotic analysis of the 2D convective Brinkman-Forchheimer equations in unbounded domains: Global attractors and upper semicontinuity, \emph{Submitted}, \url{https://arxiv.org/abs/2010.12814}.
%
%
		\bibitem{MTM3} M.T. Mohan, The $\H^1$-compact global attractor for the two dimentional convective Brinkman-Forchheimer equations in unbounded domains, \emph{J. Dyn. Control. Syst.}, \textbf{28}(4)  (2022), 791--816.
%
%
%
%
%
		\bibitem{MTM6} M.T. Mohan, $\mathbb{L}^p$-solutions of deterministic and stochastic convective Brinkman-Forchheimer equations, \emph{Anal. Math. Phys.}, {\bf 11}(4)  (2022), Paper No. 164, 33 pp.
%
%
%
%
		\bibitem{MTSS} 	M.T. Mohan and S.S. Sritharan, Stochastic Euler equations of fluid dynamics with L\'evy noise,	\emph{Asymptot. Anal.}, {\bf 99}(1-2) (2016), 67--103.
%
%
%
%
%
%
%
%
%
%
		\bibitem{Rakocevic} V. Rako$\check{c}$evi\'c, Measures of noncompactness and some applications, \emph{Filomat}, \textbf{12} (1998), 87--120.
%
%
%
		\bibitem{Robinson2} J.C. Robinson, \emph{Infinite-Dimensional Dynamical Systems, An Introduction to Dissipative Parabolic PDEs and the Theory of Global Attractors}, Cambridge Texts in Applied Mathematics, 2001.
%
%
		\bibitem{Robinson1} J.C. Robinson, \emph{Dimensions, Embeddings and Attractors}, \textbf{186}, Cambridge University Press, Cambridge, 2010.
%
%
%
%
%
%
%
%
%
%
		\bibitem{Schmalfussr} B. Schmalfu{\ss}, Backward cocycle and attractors of stochastic differential equations, In  International Seminar on Applied Mathematics Nonlinear Dynamics: Attractor Approximation and Global Behavior (V. Reitmann, T. Riedrich, and N. Koksch, eds.), Technische Universit\"{a}t Dresden, 1992, 185--192.
%
%
%
		\bibitem{Tuan1} N.H. Tuan and T. Caraballo, On initial and terminal value problems for fractional nonclassical diffusion equations, \emph{Proc. Amer. Math. Soc.}, \textbf{149} (2021), 143--161.
%
		\bibitem{R.Temam}	R. Temam, \emph{Infinite-Dimensional Dynamical Systems in Mechanics and Physics,} vol. 68, Applied Mathematical Sciences,	Springer, 1988.
%
%
%
%
%
%
		\bibitem{UTE-Wang} B. Wang, Attractors for reaction-diffusion equations in unbounded domains, \emph{Physica D}, \textbf{128}(1) (1999), 41--52.
%
%
%
		\bibitem{Wang2011Tran} B. Wang, Asymptotic behavior of stochastic wave equations with critical exponents on $\mathbb{R}^{3}$, \emph{Tran. Amer. Math. Soc.}, \textbf{363}(7) (2011), 3639--3663.
%
		\bibitem{PeriodicWang} B. Wang, Periodic random attractors for stochastic Navier-Stokes equations on unbounded domain, \emph{Electronic J. Differential Equations}, \textbf{2012}(59) (2012), 1--18.
%
%
%
%
		\bibitem{SandN_Wang} B. Wang, Sufficient and necessary criteria for existence of pullback attractors for non-compact random dynamical systems, \emph{J. Differential Equations}, \textbf{253}(5) (2012), 1544--1583.
%
%
		\bibitem{Wang}	B. Wang, Weak pullback attractors for mean random dynamical systems in Bochner spaces, \emph{J. Dynam. Differential Equations}, {\bf 31} (2019), 2177--2204.
%
%
%
		\bibitem{Wang1}	B. Wang, Weak pullback attractors for stochastic Navier-Stokes equations with nonlinear diffusion terms, \emph{Proc. Amer. Math. Soc.}, \textbf{147}(4) (2019), 1627--1638.
%
%
		\bibitem{Wang9} R. Wang, Long-time dynamics of stochastic lattice plate equations with nonlinear noise and damping, \emph{J. Dynam. Differential Equations}, \textbf{33}(2) (2021), 767--803.
%
%
		\bibitem{rwang1}R. Wang, Y. Li and B. Wang, Random dynamics of fractional nonclassical diffusion equations driven by colored noise, \emph{Discrete Contin. Dyn. Syst.}, \textbf{39}(7) (2019), 4091--4126.
%
		\bibitem{WangGUOWANG}  R. Wang, B. Guo, B. Wang, Well-posedness and dynamics of fractional FitzHugh-Nagumo systems on $\mathbb{R}^N$ driven by nonlinear noise, \emph{Sci. China Math.},   \textbf{64}(11) (2021), 2395--2436. 
%
		\bibitem{rwang2} R. Wang, L. Shi and  B. Wang, Asymptotic behavior of fractional nonclassical diffusion equations driven by nonlinear colored noise on $\mathbb{R}^N$, \emph{Nonlinearity}, \textbf{32}(11) (2019), 4524--4556.
%
%
%
%
		\bibitem{WL} S. Wang and Y. Li, Longtime robustness of pullback random attractors for stochastic magneto-hydrodynamics equations, \emph{Physica D}, \textbf{382} (2018), 46--57.
%
%
%
%
%
		\bibitem{RKM} R. Wang, K. Kinra  and M.T. Mohan, Asymptotically autonomous robustness in probability of random attractors for  stochastic Navier-Stokes equations on unbounded Poincar\'e domains, Accepted in \emph{SIAM J. Math. Anal.}, (2022), \url{https://arxiv.org/pdf/2208.06808.pdf}.
%
%
%
%
%
%
%
%
%
%
%
		\bibitem{WSY}   S. Wang, M. Si and R. Yang, Random attractors for non-autonomous stochastic Brinkman-Forchheimer equations on unbounded domains, \emph{Commun. Pure Appl. Anal.}, \textbf{21}(5) (2022), 1621--1636.
%
%
%
%
%
%
%
%
%
%
%
%
		\bibitem{XC} J. Xu and T. Caraballo, Long time behavior of stochastic nonlocal partial differential equations and Wong-Zakai approximations, \emph{SIAM J. Math. Anal.}, \textbf{54}(3) (2022), 2792--2844.
%
%
%
%
%
%
		\bibitem{ZL} Q. Zhang and Y. Li, Regular attractors of asymptotically autonomous stochastic 3D Brinkman-Forchheimer equations with delays, \emph{Commun. Pure Appl. Anal.}, \textbf{20}(10) (2021), p.3515.
%
%
%
%
%
	\end{thebibliography}
\end{document}